\documentclass[11pt, a4paper,leqno]{amsart}
\usepackage{latexsym}
\usepackage{amssymb}
\usepackage{amsmath}
\usepackage[english]{babel}
\usepackage{txfonts}

\usepackage[top=1.25in, bottom=1.25in, left=1.25in, right=1.25in]{geometry}

\numberwithin{equation}{section}
\theoremstyle{plain}
\newtheorem{theorem}{Theorem}[section]
\newtheorem{lemma}[theorem]{Lemma}
\newtheorem{corollary}[theorem]{Corollary}
\newtheorem{proposition}[theorem]{Proposition}
\theoremstyle{definition}
\newtheorem{definition}[theorem]{Definition}
\theoremstyle{remark}
\newtheorem{remark}[theorem]{Remark}
\theoremstyle{plain}
\newtheorem{mytheorem}{Theorem}

\newcommand{\rn}[1]{\mathbb{R}^{#1}} \newcommand{\re}{ \mathbb{R}}
\newcommand{\ep}{\epsilon} \newcommand{\es}{\emptyset} 
\newcommand{\al}{\alpha} \newcommand{\ga}{\gamma} \newcommand{\de}{\delta}
\newcommand{\ds}{\displaystyle} \newcommand{\ran}{\rangle} \newcommand{\lan}{\langle}
\newcommand{\Ga}{\Gamma} \newcommand{\la}{\lambda} \newcommand{\La}{\Lambda}
\newcommand{\ar}{\partial} \newcommand{\si}{\sigma} \newcommand{\om}{\omega}
\newcommand{\Om}{\Omega} \newcommand{\be}{\beta} \newcommand{\ph}{\phi}
\newcommand{\he}{\theta} \newcommand{\hs}[1]{\mbox{$ \hspace{#1}$}}
\newcommand{\sem}{\setminus} \newcommand{\ze}{\zeta} \newcommand{\De}{\Delta}
\newcommand{\ti}{\tilde} \newcommand{\cp}{\mathcal{C}_p}
\newcommand{\capp}{\mbox{Cap}_{p}}

\numberwithin{equation}{section}
\allowdisplaybreaks
\usepackage[colorlinks=true,pdfstartview={FitV 0},linkcolor=red,citecolor=red,urlcolor=blue,backref=page]{hyperref} 

\renewcommand*{\backref}[1]{}
\renewcommand*{\backrefalt}[4]{%
 \ifcase #1 (Not cited.)%
   \or        (Cited on page~#2.)%
    \else      (Cited on pages~#2.)%
    \fi}

\begin{document}

\title[Brunn--Minkowski inequality for $p$-capacity of compact sets]{Brunn--Minkowski inequality for \texorpdfstring{$p$}{p}-capacity of compact sets}

\author[M. Akman]{Murat Akman}
\address{{\bf Murat Akman}\\
Department of Mathematical Sciences \\
University of Essex \\
Colchester CO4 3SQ, United Kingdom}
\email{murat.akman@essex.ac.uk}

\begin{abstract}
We prove the capacitary Brunn--Minkowski inequality for arbitrary
nonempty compact sets in $\mathbb R^n$, $n\geq2$, and every
$1<p<\infty$, with no assumption of convexity, connectedness,
interior, or boundary regularity. For $1<p<n$ the inequality is for
the variational $p$-capacity. For $p>n$ it is for the capacity
determined by the constant term at infinity of the normalized $p$-harmonic
Green's function with pole at infinity, and for $p=n$ for an
outer regularization of this quantity, which agrees with it for
$n$-regular sets and in the plane is a constant multiple of
logarithmic capacity. In particular, it holds for
Newtonian capacity of all compact sets in $\mathbb R^n$, $n\geq3$, and for logarithmic capacity
of all planar compact sets. Previously these inequalities were known
for convex sets and, for logarithmic capacity, for planar continua.
When both capacities are positive, equality holds for $0<\lambda<1$ if and only if the
filled sets, obtained by adjoining the bounded components of the
complement, are convex and equal up to translation and dilation. This
characterization is new for nonconvex sets, even for Newtonian and
logarithmic capacity.
\end{abstract}
\maketitle

\section{Introduction}
The classical Brunn--Minkowski inequality states that the $n$-th root
of volume is concave under Minkowski addition. This was shown first for convex bodies
and then for arbitrary nonempty compact sets. Here
$E+F=\{x+y:x\in E,\ y\in F\}$ is the Minkowski sum, and
$tE=\{tx:x\in E\}$ for $t\geq0$. See Gardner's survey \cite{G}. For Newtonian capacity, Borell \cite{B1} proved the
corresponding concavity of $\mbox{Cap}_2^{1/(n-2)}$ for convex bodies
in $\rn{n}$, $n\geq3$. Caffarelli, Jerison, and Lieb \cite{CJL}
established its equality case. Borell \cite{B2} also proved the
Brunn--Minkowski inequality for logarithmic capacity of planar convex
sets. Colesanti and Salani \cite{CS} extended the capacitary inequality
to $\mbox{Cap}_p^{1/(n-p)}$ for $1<p<n$ and convex bodies, and
characterized equality. Here
\begin{equation}\label{eqn1.4}
\mbox{Cap}_p(E)=\inf\left\{
\int_{\rn{n}}|\nabla v|^p\,dx:
 v\in C^\infty_0(\rn{n}),\ v\geq1\ \mbox{on }E
\right\}
\end{equation}
is the variational $p$-capacity. Colesanti and Cuoghi \cite{CC}
treated $p=n$ for convex bodies, using an $n$-dimensional logarithmic
capacity defined through the $n$-Laplace equation, and characterized
equality. For compact convex sets of positive capacity, the author
with Gong, Hineman, Lewis, and Vogel \cite{AGHLV} ($1<p<n$) and with
Lewis, Saari, and Vogel \cite{ALSV} ($p\geq n$) proved the inequality
for capacities of $\mathcal A$-harmonic equations modelled on the
$p$-Laplace equation, and characterized equality under structural
hypotheses satisfied by the $p$-Laplace operator. These two works allow
convex sets with empty interior. 

In the present article we extend these results, including the
characterization of equality, to arbitrary nonempty compact sets and
every $1<p<\infty$.  To describe these results, we note that the capacity $\Phi_p$ in \eqref{eq:degree-one} is homogeneous of degree
one. It equals $\mbox{Cap}_p^{1/(n-p)}$ when $1<p<n$, the outer
capacity $\bar{\mathcal C}_n$ when $p=n$, and $\mathcal C_p^{1/(p-n)}$
when $p>n$, where $\mathcal C_p$ is the capacity defined through the
$p$-harmonic Green's function with pole at infinity. The
\textit{filled} set $\hat E$ adjoins to $E$ the bounded components of
its complement. A compact set is $p$-regular when every boundary point
of its unbounded exterior domain is regular for continuous Dirichlet
data. Precise definitions are in Section~\ref{NSR}.

\begin{mytheorem}
\label{thmA}
Let $E_1$ and $E_2$ be nonempty compact sets in $\rn{n}$, $n\geq 2$, and let $\la\in [0,1]$.

\begin{itemize}
\item[(a)] If $1<p<n$ then
\end{itemize}
\begin{align}
\label{BMp}
\left[\capp(\lambda E_1+(1-\lambda) E_2)\right]^{\frac{1}{n-p}}\geq\lambda \left[\capp(E_{1})\right]^{\frac{1}{n-p}}+(1-\lambda)\left[\capp(E_2)\right]^{\frac{1}{n-p}}.
\end{align}
\begin{itemize}
\item[(b)] If $p=n$ and $E_1,E_2$ are $n$-regular, then
\end{itemize}
\begin{align}
\label{BMn}
\cp(\lambda E_1+(1-\lambda)E_2)\geq \lambda\,  \cp(E_1) +(1-\lambda)\, \cp(E_2).
\end{align}
\begin{itemize}
\item[(c)] If $n<p<\infty$ then
\end{itemize}

\begin{align}
\label{BMpn}
[\cp(\lambda E_1+(1-\lambda)E_2)]^{\frac{1}{p-n}}\geq
\lambda \, [\cp(E_1)]^{\frac{1}{p-n}} +(1-\lambda)\, [\cp(E_2)]^{\frac{1}{p-n}}.
\end{align}
\begin{itemize}
\item[(d)] If $1<p<\infty$, $0<\lambda<1$, and
$\Phi_p(E_1),\Phi_p(E_2)>0$, then
\end{itemize}
\begin{equation}\label{eq:equality}
\Phi_p(\lambda E_1+(1-\lambda)E_2)
=\lambda\Phi_p(E_1)+(1-\lambda)\Phi_p(E_2)
\end{equation}
if and only if $\hat E_1$ and $\hat E_2$ are convex and equal up to
translation and dilation.
\end{mytheorem}

The hypothesis in part $(b)$ holds, in particular, if no component of
$E_1$ or of $E_2$ is a single point, by Lemma~\ref{lem9.6} $(c)$.
Under the assumptions of parts $(b)$ and $(c)$, the Green's function
for the Minkowski combination exists, so the left-hand sides of
\eqref{BMn} and \eqref{BMpn} are well-defined (Section~\ref{section:proofA}).

In part $(d)$, we note that filling adds only the bounded components of the complement and generally
differs from taking the convex hull. The unit sphere $S$ is
nonconvex and has empty interior, but $\hat S=\overline{B(0,1)}$. For $0<\lambda<1$ the set
$\lambda S+(1-\lambda)S=\{x:|1-2\lambda|\leq|x|\leq1\}$ also has filling
$\overline{B(0,1)}$. Since $\Phi_p$ is unchanged by filling
(Lemma~\ref{lem:filling-invariance}), equality
holds in Theorem~\ref{thmA} $(d)$ for $E_1=E_2=S$. For the union $E$
of two disjoint closed balls, $\hat E=E$ is nonconvex, so
$\lambda E+(1-\lambda)E$ has strictly larger capacity than $E$ for
$0<\lambda<1$. Replacing the fillings by
convex hulls would therefore lose the equality characterization.

\begin{remark}\label{rmk:A-b}
For $p=n$, Theorem~\ref{thmB} $(c)$ below removes the $n$-regularity in
part $(b)$ by passing to the outer capacity $\bar{\mathcal C}_n=\Phi_n$.
Thus $\Phi_p(\lambda E_1+(1-\lambda)E_2)\geq\lambda\Phi_p(E_1)+(1-\lambda)\Phi_p(E_2)$
for all nonempty compact sets and every $1<p<\infty$.
\end{remark}

\begin{remark}\label{rmk:pgeqn}
For $p\geq n$ the capacity \eqref{eqn1.4} vanishes on every compact
set, and $1/(n-p)$ is undefined or negative. We therefore use
$\mathcal C_p$ of \eqref{eqn2.10}, defined through the Green's
function as in \cite{ALSV}. Relative capacities $\capp(E,B(0,R))$ are
not homogeneous under dilation of $E$ and serve only as tools, as in
\eqref{eqn2.17}.
\end{remark}

Without convexity, Pommerenke \cite{P} proved
\begin{equation}\label{eqn1.5}
\mbox{cap}(E_1+E_2)\geq\mbox{cap}(E_1)+\mbox{cap}(E_2)
\end{equation}
for planar continua, where $\mbox{cap}$ is logarithmic capacity.
For Brunn--Minkowski inequalities for other functionals, or under
radial addition, see
\cite{Bo,C,CCS,CF,GH}. To our knowledge, apart from Pommerenke's
theorem, Theorems~\ref{thmA} and~\ref{thmB} are the first capacitary
Brunn--Minkowski inequalities under Minkowski addition for nonconvex
sets, and we know of no earlier equality characterization for
nonconvex sets, even for Newtonian and logarithmic capacity.

At $p=n$, a continuous Green's function need not exist for an arbitrary
compact set. We therefore use the outer capacity
$\bar{\mathcal C}_n$ of \eqref{eqn2.14}, the limit along
the closed neighbourhoods $E_\delta=\{x:d(x,E)\leq\delta\}$ of
\eqref{eqn2.5}. Below, $\mbox{Cap}_n(E,B(z_0,r))$ is the relative
capacity \eqref{eqn2.15} and $\gamma=|\mathbb S^{n-1}|^{-1/(n-1)}$ as in
\eqref{eqn2.7}.

\begin{mytheorem}
\label{thmB}
Let $p=n\geq 2$ and let $\bar{\mathcal{C}}_n(\cdot)$ be as in \eqref{eqn2.14}.

\begin{itemize}
\item[(a)]  If $E\subset E'$ are nonempty compact sets, $\rho>0$, and $x\in\rn{n}$, then $\bar{\mathcal{C}}_n(E)\leq \bar{\mathcal{C}}_n(E')$ and $\bar{\mathcal{C}}_n(\rho E+x)=\rho\, \bar{\mathcal{C}}_n(E)$.

\item[(b)] If $E$ is a nonempty compact set for which there exists a function $G$ satisfying \eqref{eqn2.9}, in particular if $E$ is $n$-regular, then $\bar{\mathcal{C}}_n(E)=\mathcal{C}_n(E)$.

\item[(c)] If $E_1,E_2$ are nonempty compact sets in $\rn{n}$ and $\la\in[0,1]$, then
\begin{align}
\label{eqn2.16}
\bar{\mathcal{C}}_n(\lambda E_1+(1-\lambda)E_2)\geq \lambda\,  \bar{\mathcal{C}}_n(E_1) +(1-\lambda)\, \bar{\mathcal{C}}_n(E_2).
\end{align}

\item[(d)] If $E$ is a nonempty compact set and $z_0\in\rn{n}$, then
\begin{align}
\label{eqn2.17}
\bar{\mathcal{C}}_n(E)=\lim_{r\to\infty}\, e^{1/\ga}\, r\, \exp\left(-\frac{1}{\ga}\left[\mbox{\rm Cap}_n(E,B(z_0,r))\right]^{-\frac{1}{n-1}}\right),
\end{align}
where the exponential is zero when $\mbox{\rm Cap}_n(E,B(z_0,r))=0$.
In particular, $\bar{\mathcal{C}}_n(E)>0$ if and only if
$\mbox{\rm Cap}_n(E,B(z_0,r))>0$ for one, equivalently every,
$r>0$ with $E\subset B(z_0,r)$.

\item[(e)] If $n=2$ and $E$ is a nonempty compact set in $\rn{2}$, then $\bar{\mathcal{C}}_2(E)=e^{2\pi}\, \mbox{\rm cap}(E)$ where $\mbox{\rm cap}$ denotes logarithmic capacity.
Consequently, if $E_1,E_2$ are nonempty compact sets in $\rn{2}$ and $\la\in[0,1]$, then
\begin{align}
\label{eqn2.18}
\mbox{\rm cap}(\lambda E_1+(1-\lambda)E_2)\geq \lambda\,  \mbox{\rm cap}(E_1) +(1-\lambda)\, \mbox{\rm cap}(E_2).
\end{align}
\end{itemize}
\end{mytheorem}

Apart from the factor $e^{1/\gamma}$, \eqref{eqn2.17} is the condenser
formula for logarithmic $n$-capacity in \cite{Be,X}. For logarithmic
capacity see Armitage--Gardiner \cite{AG}. Theorem~\ref{thmA} $(a)$
with $p=2<n$ gives the Newtonian capacity inequality for all nonempty
compact sets in $\rn{n}$, $n\geq3$, and Theorem~\ref{thmB} $(e)$
extends \eqref{eqn1.5} to all nonempty planar compact sets.

Theorem~\ref{thmC} in Section~\ref{section:proofC} gives pointwise comparisons for the capacitary and
Green's functions and their outer limits, extending
\cite[(5.13)]{AGHLV} and \cite[(6.11)]{ALSV} to arbitrary compact sets.
Remark~\ref{rmk:zero-capacity-equality} treats equality when exactly
one of $\Phi_p(E_1),\Phi_p(E_2)$ vanishes, when both vanish and $p>n$
(then both sets are points), and when both vanish and $1<p\leq n$
(then equality means that the sum is $p$-polar). The following corollaries are
proved at the end of Section~\ref{section:proofEquality}.

\begin{corollary}\label{cor:self-average}
Let $E$ be a nonempty compact set with $\Phi_p(E)>0$.
For every $0<\lambda<1$,
\begin{equation}\label{eq:self-average}
\Phi_p\bigl(\lambda E+(1-\lambda)E\bigr)=\Phi_p(E)
\quad\Longleftrightarrow\quad \hat E\text{ is convex}.
\end{equation}
\end{corollary}

\begin{corollary}\label{cor:sum-parallel}
Let $E,E'\subset\rn{n}$ be nonempty compact sets, let $\delta>0$,
and put $\bar B=\overline{B(0,1)}$. Then
\begin{align}
\label{eq:sum-capacity}
\Phi_p(E+E')&\geq\Phi_p(E)+\Phi_p(E'),\\
\label{eq:parallel-capacity}
\Phi_p(E_\delta)&\geq\Phi_p(E)+\delta\Phi_p(\bar B).
\end{align}
If $\Phi_p(E),\Phi_p(E')>0$, equality in \eqref{eq:sum-capacity}
holds if and only if $\hat E$ and $\hat{E'}$ are convex and equal
up to translation and dilation. Equality in
\eqref{eq:parallel-capacity} holds if and only if $\hat E$ is a
closed ball, allowing radius zero.
\end{corollary}

For sets of positive volume, equality in the classical inequality
characterizes convex sets equal up to translation and dilation
\cite{G}. Theorem~\ref{thmA} $(d)$ gives the same conclusion for the
filled sets under positive capacity, but the two equality cases
differ. For the annulus $E=\{x:r\leq|x|\leq1\}$, $0<r<1$, the filling
and the midpoint sum are both the closed unit ball, so the capacitary
inequality is an equality while the volume inequality is strict.
In Corollary~\ref{cor:self-average} positive capacity is needed only
for $1<p\leq n$. For these values of $p$, a two-point set $E$ has nonconvex filling, and
$\lambda E+(1-\lambda)E$ is finite, hence of capacity zero. For $p>n$ the only
sets of zero capacity are singletons, whose fillings are convex.

The main obstruction is the loss of convexity. For convex bodies the
exterior functions (the capacitary function when $1<p<n$ and the
Green's function with pole at infinity when $p\geq n$) have convex
level sets and nonzero gradients \cite{L,ALSV}. The earlier
comparison arguments \cite{CS,CC}, \cite[Lemmas~4.3--4.4]{AGHLV},
\cite[Lemma~6.1]{ALSV} use this, and in the equality case the
nonvanishing gradient lets common level sets found far from the sets
be continued analytically down to the sets \cite{CJL}. For nonconvex
sets both properties fail. The exterior function of two disjoint
equal balls has a critical point at the midpoint of their centres, and
unique continuation for two $p$-harmonic functions is open in general
when $n\geq3$ and $p\neq2$ \cite{GM}. Arbitrary compact sets may also
have irregular boundary points.

For $1<p<n$ and sets satisfying the corkscrew condition
(Definition~\ref{defn2.4}), let $u_1,u_2,u$ be the capacitary
functions of $E_1$, $E_2$ and $\lambda E_1+(1-\lambda)E_2$. Following
\cite{B1,CS,AGHLV} we prove that the supremal convolution
\begin{equation}\label{eqn1.6}
u^*(x)=\sup\left\{\min\{u_1(y),u_2(z)\}:
 x=\lambda y+(1-\lambda)z,\quad y,z\in\rn{n}\right\}
\end{equation}
satisfies $u^*\leq u$, using the viscosity method of \cite{CSq,CuS}
(see also \cite{BLS,LS,Sa15}). The new step, Lemma~\ref{lem7.1}, shows
that an optimal pair at which the two function values differ, or a
vanishing gradient, forces every test function touching $u^*$ from above
to have zero gradient. The test criterion of
\cite{JLM} then makes $u^*$ $p$-subharmonic without any assumption on
the gradients of the exterior functions. For $p\geq n$ we use the
negatives of the Green's functions, as in \cite{ALSV}. The
neighbourhoods $E_\delta$ satisfy the corkscrew condition and
$\lambda(E_1)_\delta+(1-\lambda)(E_2)_\delta
=(\lambda E_1+(1-\lambda)E_2)_\delta$, so decreasing approximation
gives the general case. For equality we continue the common convex
level sets of the exterior functions from infinity towards the sets,
using analytic continuation only where Hopf's lemma gives a nonzero
gradient. A contact of a continued level set with the sets can only
occur at an irregular boundary point when $1<p\leq n$. A blow-up and a
comparison of Riesz measures on a tangent cone exclude it
(Proposition~\ref{prop:unrestricted-contact}), so no unique
continuation theorem is needed.

Section~\ref{section:preliminaries} develops the exterior functions,
Section~\ref{section:proofA} proves the inequalities in
Theorem~\ref{thmA}, Section~\ref{section:proofB}
proves Theorems~\ref{thmB} and~\ref{thmC},
and Section~\ref{section:proofEquality} proves the equality criterion
and the corollaries.

\section*{Acknowledgement}
The author used ChatGPT to assist with exposition, the revision
and checking of mathematical arguments, and literature searches, and takes
full responsibility for the final manuscript.

\section{Notation and capacities}
\label{NSR}

\noindent Let $\mathbb S^{n-1}$ be the unit sphere in $\rn{n}$, $n\geq2$, and
$\si_{n-1}=\mathcal{H}^{n-1}(\mathbb{S}^{n-1})$. We write $\bar E$,
$\ar E$, $\mbox{diam}(E)$ and $d(y,E)$ for the closure, boundary,
diameter and distance to $E$, and use $E+F$, $E+x=E+\{x\}$ and $\rho E$
as in the introduction. $B(z,r)$ is the open ball of centre $z$ and
radius $r$, and $\lan\cdot,\cdot\ran$ the inner product. The spaces
$W^{1,q}(O)$, $W^{1,q}_0(O)$ and $C^\infty_0(O)$ are as in
\cite[Chapter~1]{HKM}. For a symmetric matrix $X$ we write
$X[e,e]=\lan Xe,e\ran$, and $D^2h$ is the Hessian of $h$.
\begin{definition}
\label{defn2.1}
	Let $p\in (1,\infty)$. Given an  open set
 $ O  $ we say that $ u $ is $p$-harmonic in $ O $ provided $ u \in W^ {1,p} ( D ) $ for each bounded open $ D $ with  $ \bar D \subset O $ and
	\begin{align}
	\label{eqn2.1}
\int | \nabla u(y)|^{p-2}\,  \lan    \nabla u(y), \nabla \he ( y ) \ran \, dy = 0 \quad \mbox{whenever} \, \,\he \in W^{1, p}_0 ( D ).
			\end{align}
We write $\nabla \cdot (|\nabla u|^{p-2}\nabla u)=0$ in $O$.
We call $u\in W^{1,p}_{\mathrm{loc}}(O)$ a weak supersolution
if the integral in \eqref{eqn2.1} is nonnegative for every
nonnegative $\theta\in W^{1,p}_0(D)$ and every bounded open
$D\Subset O$. A weak subsolution satisfies the reverse inequality.
\end{definition}
This is \cite[Definition~2.2]{ALSV} with $\mathcal{A}(\eta)=|\eta|^{p-2}\eta$.
It is invariant under $u\mapsto au+b$ and under composition with
translations, dilations and rotations.

\begin{definition}\label{def:p-superharmonic}
A function $v:O\to(-\infty,\infty]$ is $p$-superharmonic in $O$ if it
is lower semicontinuous, finite on a dense subset of $O$, and
$h\leq v$ in $D$ whenever $D\Subset O$ is open and $h\in C(\bar D)$ is
$p$-harmonic in $D$ with $h\leq v$ on $\partial D$. It is
$p$-subharmonic if $-v$ is $p$-superharmonic. We use the lower
(upper) semicontinuous representative \cite[Chapter~7]{HKM}.
\end{definition}

By \cite[Chapter~7]{HKM}, every $p$-superharmonic function $v$ in $O$ satisfies
\begin{equation}\label{eq:superharmonic-representative}
v(x)=\operatorname*{ess\,liminf}_{y\to x}v(y)\qquad(x\in O).
\end{equation}

 Given a compact set $E\subset \rn{n}$ we let
 \begin{align}
 \label{eqn2.4}
 \begin{split}
 &\Om_E \, \, \mbox{denote the unbounded component of}\, \, \rn{n}\sem E \, \, \mbox{and}\\
 &\hat E=\rn{n}\sem \Om_E.
 \end{split}
 \end{align}
Thus $\hat E$ is compact, it is the union of $E$ and the bounded
components of $\rn{n}\sem E$, and $\ar \hat E=\ar \Om_E\subset \ar E$.

 \begin{definition}[\bf Corkscrew condition]
 \label{defn2.4}
 A nonempty compact set $E\subset \rn{n}$ satisfies the corkscrew condition if there exist $c_E > 1$ and $r_E>0$ such that for each $w\in E$ and $0<r<r_E$ the ball $B(w,r)$ contains a ball $B(y, r/c_E)\subset E$.
 \end{definition}
This class is closed under Minkowski addition. Indeed, if $w=e+f$ with $e\in E$, $f\in F$, and $B(y,r/c_E)\subset B(e,r)\cap E$, then $B(y+f,r/c_E)\subset B(w,r)\cap(E+F)$.
\begin{remark}
\label{rmk:Lipcorkscrew}
It is a well-known geometric fact that the closure of a bounded Lipschitz domain satisfies the corkscrew
condition. See the construction following \cite[(8.2)]{ALSV}.
\end{remark}
For a nonempty compact set $E$ and $\de>0$ put
\begin{align}
\label{eqn2.5}
E_\de=\{x\in \rn{n}:\, d(x,E)\leq \de\}.
\end{align}
By Lemma~\ref{lem9.1}, $E_\de$ is compact and satisfies the corkscrew condition.

Let $1<p<n$. By Lemma~\ref{lem4.1}, if $E$ is $p$-regular
(Definition~\ref{defn9.5}) there is a unique continuous function $u$
on $\rn{n}$, the $p$-capacitary function for $E$, with
\begin{align}
\label{eqn2.6}
\begin{split}
(a)&\, \, u\, \, \mbox{is $p$-harmonic in}\, \,  \rn{n}\sem E,  \\
(b)&\, \,  u \equiv 1\, \, \mbox{on}\, \,  E,     \\
 (c)&\, \,  | \nabla u | \in   L^p ( \rn{n} )\, \,  \mbox{and}\, \, u \in L^{p^*}  ( \rn{n}  )\, \, \mbox{for}\, \,  p^* =  \frac{np}{n-p}.
  \end{split}
  \end{align}
 If $n\leq p<\infty$ let
 \begin{align}
 \label{eqn2.7}
 F(x)=\left\{
 \begin{array}{ll}
 \ga \log|x| +1  \, \,   \mbox{where}\, \, \ga=\si_{n-1}^{-\frac{1}{n-1}} &\mbox{when}\, \, p=n,\\
 \xi^{-1}\, \si_{n-1}^{-\frac{1}{p-1}}\, |x|^{\xi} \, \,  \mbox{where}\, \, \xi=\frac{p-n}{p-1}&\mbox{when}\, \, n<p<\infty,
 \end{array}
 \right.
 \end{align}
for $x\neq0$, and $F(0)=0$ if $p>n$. Since
$|\nabla F|^{p-2}\nabla F(y)=\si_{n-1}^{-1}|y|^{-n}y$, $F$ is
$p$-harmonic in $\rn{n}\sem\{0\}$ and
$\int|\nabla F|^{p-2}\lan\nabla F,\nabla\he\ran\,dy=-\he(0)$ for
$\he\in C^\infty_0(\rn{n})$.
It is the fundamental solution of \cite[Lemmas~4.4 and~4.6]{ALSV},
normalized as in \cite[Section~5]{ALSV}. Adapting
\cite[Definition~5.1]{ALSV} to sets whose complements may have bounded
components, we call $G$ the $p$-harmonic Green's function for
$\rn{n}\sem E$ with pole at $\infty$ if
 \begin{align}
 \label{eqn2.9}
 \begin{split}
 (a)&\, \, G\, \, \mbox{is continuous on $\rn{n}$ with $G\equiv 0$ on $E$},\\
 (b)&\, \, G\, \, \mbox{is $p$-harmonic in}\, \,  \rn{n}\sem E,  \\
(c)&\, \,  G(x)=F(x)+k(x) \, \, \mbox{where $k$ is a bounded function in a neighbourhood of $\infty$.}
 \end{split}
 \end{align}
By Lemmas~\ref{lem4.2} and~\ref{lem9.13}, $G$ exists whenever $E$ is
$p$-regular, and whenever $G$ exists it is unique and
$k(\infty)=\lim_{x\to\infty}k(x)$ exists finitely. Following
\cite[Definition~5.4]{ALSV} we put
 \begin{align}
 \label{eqn2.10}
\cp( E ) :=
\left\{
\begin{array}{ll}
  e^{- k ( \infty )/\ga} &  \mbox{when}\, \, p = n, \\
 (- k ( \infty ))^{p-1} & \mbox{when}\, \,  p > n.
\end{array}
\right.
\end{align}
For convex $E$ this is $\mathcal{C}_{\mathcal{A}}(E)$ of \cite{ALSV}
with $\mathcal{A}(\eta)=|\eta|^{p-2}\eta$. If $p>n$ then $\cp(E)$ is
defined for every nonempty compact set, and $\cp(E)=0$ if and only if
$E$ is a single point (Lemmas~\ref{lem9.6} and~\ref{lem4.2}).

At $p=n$ a continuous Green's function need not exist. For a
singleton it would extend $n$-harmonically across the point
\cite[Theorem~7.36]{HKM} and attain an interior minimum there. We
therefore define
\begin{align}
\label{eqn2.14}
\bar{\mathcal C}_n(E)=\lim_{\delta\to0}\mathcal C_n(E_\delta),
\end{align}
which exists and is finite by Lemma~\ref{lem:outer-capacity}. For a
compact set $K$ in an open set $O$, possibly $O=\rn{n}$
\cite[pp.~27--28]{HKM}, the relative capacity is
\begin{align}
\label{eqn2.15}
\capp(K,O)=\inf\left\{
\int_O|\nabla v|^p\,dx:
v\in C^\infty_0(O),\ v\geq1\ \mbox{on }K
\right\}.
\end{align}

For a nonempty compact set $E$ and $\de>0$, let $u_\de$ be the
$p$-capacitary function for $E_\de$ if $1<p<n$, and $G_\de$ the
$n$-harmonic Green's function for $\rn{n}\sem E_\de$ with pole at
$\infty$ if $p=n$. As $\de$ decreases, $u_\de$ does not increase and
$G_\de$ does not decrease (Section~\ref{section:proofC}). Put
\begin{align}
\label{eqn2.19}
\bar u_E=\lim_{\de\to 0}u_\de\, \, \mbox{when}\, \, 1<p<n\quad \mbox{and}\quad \bar G_E=\lim_{\de\to 0}G_\de\, \, \mbox{when}\, \, p=n.
\end{align}
Then $0\leq \bar u_E\leq 1$ and $0\leq \bar G_E\leq \infty$. Let
$\underline{u}_E(x)=\lim_{r\to 0}\inf_{B(x,r)}\bar u_E$ be the lower
semicontinuous regularization of $\bar u_E$. Finally, put
\begin{equation}\label{eq:degree-one}
\Phi_p(E)=
\begin{cases}
\capp(E)^{1/(n-p)},&1<p<n,\\
\overline{\mathcal C}_n(E),&p=n,\\
\cp(E)^{1/(p-n)},&p>n.
\end{cases}
\end{equation}
Its basic properties are proved in Lemma~\ref{lem:filling-invariance}.

\section{Basic estimates and exterior functions}
\label{section:preliminaries}
We recall interior estimates, introduce boundary regularity, and
construct the exterior functions for $p$-regular sets.

\subsection{Basic estimates}
\label{section3}

We use continuous representatives of $p$-harmonic functions and assume $1<p<\infty$.
Throughout the paper, $c\geq1$ denotes a constant depending only on $p,n$ and may change at each occurrence.
The constant $c(a_1,\ldots,a_m)$ may also depend on the listed quantities.
For positive quantities $a,b$, we write $a\approx b$ if $c^{-1}b\leq a\leq cb$ for a constant of this kind.

\begin{lemma}
\label{lem3.1}
   Let $ \ti u $  be  a positive
 $p$-harmonic function in $B (w,4r)$ for some $w\in \rn{n}$ and $r>0$. Then
 \begin{align}
	\label{eqn3.1}
	\begin{split}
		(i)&\, \, r^{p - n } \,\int_{B ( w, r/2)} \, | \nabla \ti  u |^{  p } \, dy \, \leq \, c \, (\sup_{B ( w,r)}   \ti  u )^p ,  \\
		(ii)&\, \,   {\ds \sup_{B ( w, r ) }} \,\ti   u \, \leq c {\ds \inf_{ B ( w, r )} }\ti u.
		\end{split}
\end{align}
Also $ \ti u$ has a  representative   with H\"{o}lder
continuous partial derivatives in $B(w,4r) $  (also denoted $\ti u$), and there exist $  \be \in (0,1]$ and  $c \geq 1 $,
depending only on $ p$ and $n$, such that if $ x, y \in B (  w,  r ), $   then
	\begin{align}
	\label{eqn3.2}
c^{ - 1} \, | \nabla \ti u ( x ) - \nabla \ti u ( y ) | \, \leq \,
 ( | x - y |/ r)^{ \be }\, \sup_{B (  w, 2r )} \, | \nabla \ti u | \leq \, c \,  r^{ - 1} \, ( | x - y |/ r )^{ \be }\, \ti u (w).
\end{align}
Finally, if $\nabla \ti u(w)\neq 0$ then $\ti u$ has H\"{o}lder continuous second partial derivatives in a neighbourhood of $w$ and
\begin{align}
\label{eqn3.3}
\De \ti u (x)+(p-2)\, |\nabla \ti u(x)|^{-2} \, D^2 \ti u(x)[\nabla \ti u(x), \nabla \ti u(x)]=0
\end{align}
for $x$ in this neighbourhood.
\end{lemma}
\begin{proof}
The estimates follow from \cite{S,T} and
\cite[Lemmas~3.3--3.4]{ALSV}, after rescaling the balls and applying
Harnack's inequality. Near a point where $\nabla\tilde u\ne0$,
continuity bounds its length away from zero, and the higher
regularity assertion of \cite[Lemma~3.4]{ALSV} applies.
Expanding the $p$-Laplace equation there gives \eqref{eqn3.3}.
\end{proof}
Equation \eqref{eqn3.3} can be written as
\begin{align}
\label{eqn3.4}
\mbox{trace}\left( \left( I+(p-2)\nu\otimes \nu\right) D^2 \ti u(x)\right)=0\quad \mbox{where}\, \, \nu=\pm \frac{\nabla \ti u(x)}{|\nabla \ti u (x)|}.
\end{align}
Here $I$ is the identity matrix and $\nu\otimes \nu$ has entries $\nu_i\nu_j$. The matrix $I+(p-2)\nu\otimes \nu$ has eigenvalues $1$ and $p-1$.
Adding a constant shows that the last sentence of Lemma \ref{lem3.1} holds for $p$-harmonic functions which are not positive.

We use the comparison principle \cite[Theorem~7.6]{HKM}. If $O$ is bounded and open, $v,h\in C(\bar O)$ are $p$-harmonic in $O$, and $v\leq h$ on $\ar O$, then $v\leq h$ in $O$.
We also use the strong maximum principle \cite[Theorem~6.5]{HKM}. A nonconstant $p$-harmonic function in a domain attains neither its supremum nor its infimum there.
The next lemma follows from a radial barrier, as in \cite{L}. See also \cite{To,V}.

\begin{lemma}
\label{lem3.3}
Let $\hat w\in \rn{n}$, $\de>0$, and let $h$ be a positive $p$-harmonic function in $B(\hat w,\de)$ which is continuous on $\bar B(\hat w,\de)$. Then there exists $c_-\geq 1$ depending only on $p$ and $n$ such that
\begin{align}
\label{eqn3.6}
c_-\, h(x) \, \geq \, \left(\inf_{\bar B(\hat w,\de/2)} h\right) \, \frac{\de-|x-\hat w|}{\de}\quad \mbox{whenever}\, \, x\in B(\hat w,\de)\sem B(\hat w,\de/2).
\end{align}
In particular,  if $w\in \ar B(\hat w,\de)$,  $h(w)=0$, and $h$ has an extension to a neighbourhood of $w$ which is differentiable at $w$,  then $\lan \nabla h(w), \hat w-w\ran>0$.
\end{lemma}

\begin{proof}
Put $m=\inf_{\bar B(\hat w,\de/2)}h>0$ and $\xi=(p-n)/(p-1)$. Let
$\psi(x)=g(|x-\hat w|/\de)$, where $g(t)=(t^{\xi}-1)/(2^{-\xi}-1)$ if $p\neq n$ and $g(t)=\log(1/t)/\log2$ if $p=n$.
This radial function is $p$-harmonic on the annulus, with values $1$ and $0$ on its inner and outer boundaries. Comparison gives $h\geq m\psi$ there.
Since $g(1)=0$ and $\inf_{[1/2,1]}(-g')>0$ depends only on $p,n$, this proves \eqref{eqn3.6}.
For the last assertion, apply this estimate at $x=w+t(\hat w-w)$, divide by $t$, and let $t\downarrow0$.
\end{proof}

The next estimate applies to $v_1-v_2$ and requires a lower bound for
only one of $|\nabla v_1|$ and $|\nabla v_2|$.
\begin{lemma}
\label{lem9.7}
Let $1<p<\infty$, $s>0$, $\mathcal U=B(0,2s)\sem \bar B(0,s/2)$,  and let $v_1,v_2$ be $p$-harmonic functions in $\mathcal U$ with $v_1\geq v_2$ in $\mathcal U$. Suppose that for some $\tau>0$ and $\hat c\geq 1$,
\begin{align}
\label{eqn9.12}
\begin{split}
&|\nabla v_1|\leq \hat c\, \tau \, \, \mbox{and}\, \,  |\nabla v_2|\leq \hat c\, \tau\, \, \mbox{in}\, \, \mathcal U,\, \, \mbox{and either}\\
&|\nabla v_1|\geq \hat c^{-1}\, \tau\, \, \mbox{in}\, \, \mathcal U\, \, \mbox{or}\, \, |\nabla v_2|\geq \hat c^{-1}\, \tau \, \, \mbox{in}\, \, \mathcal U.
\end{split}
\end{align}
Then
\begin{align}
\label{eqn9.13}
\sup_{\ar B(0,s)}(v_1-v_2)\leq c\, \inf_{\ar B(0,s)}(v_1-v_2)
\end{align}
where $c\geq 1$ depends only on $p,n$, and $\hat c$.
\end{lemma}
\begin{proof}
Let $\mathcal A(\eta)=|\eta|^{p-2}\eta$, with $\mathcal A(0)=0$, and put
$\eta_t=t\nabla v_1+(1-t)\nabla v_2$,
\[
\hat b_{ij}(x)=\int_0^1D_j\mathcal A_i(\eta_t(x))\,dt,
\qquad
\mathcal I(x)=\int_0^1|\eta_t(x)|^{p-2}\,dt.
\]
We claim that
\begin{align}
\label{eqn9.14}
c^{-1}\tau^{p-2}\leq\mathcal I(x)\leq c\tau^{p-2},
\end{align}
with $c=c(p,\hat c)$.
By symmetry assume $a=\nabla v_2(x)$ satisfies $|a|\geq\tau/\hat c$, and set $b=\nabla v_1(x)-a$.
Then $|\eta_t|\leq\hat c\tau$, $|b|\leq2\hat c\tau$, and $|\eta_t|\geq\tau/(2\hat c)$ for $0\leq t\leq(2\hat c)^{-2}$.
These bounds prove \eqref{eqn9.14} when $p\geq2$, and its lower bound when $1<p<2$.
For the remaining upper bound, if $|b|\leq|a|/2$, then $|\eta_t|\geq\tau/(2\hat c)$ for all $t$.
Otherwise, with $t_0=-\lan a,b\ran/|b|^2$, we have $|\eta_t|\geq|b||t-t_0|$ and
\[
\mathcal I(x)\leq |b|^{p-2}\int_0^1|t-t_0|^{p-2}\,dt
\leq\frac{2}{p-1}\left(\frac{\tau}{2\hat c}\right)^{p-2}.
\]
The derivative formula for $\mathcal A$ gives
\begin{align}
\label{eqn9.15}
\min\{1,p-1\}\mathcal I(x)|\vartheta|^2
\leq\sum_{i,j}\hat b_{ij}(x)\vartheta_i\vartheta_j
\leq\max\{1,p-1\}\mathcal I(x)|\vartheta|^2.
\end{align}
Integration along the segment, whose derivative is integrable by \eqref{eqn9.14}, gives
$\mathcal A_i(\nabla v_1)-\mathcal A_i(\nabla v_2)
=\sum_j\hat b_{ij}(v_1-v_2)_{x_j}$.
Subtracting the weak equations, we find that $h=v_1-v_2\geq0$ solves
$\sum_{i,j}\ar_{x_i}(\hat b_{ij}h_{x_j})=0$ weakly in $\mathcal U$, and $h\in W^{1,2}(\mathcal U)$ by \eqref{eqn9.12}.
After division of the coefficients by $\tau^{p-2}$, \eqref{eqn9.14}--\eqref{eqn9.15} give uniform ellipticity with constants depending only on $p,n,\hat c$.
The Harnack inequality of \cite{GT} in balls $B(y,s/8)$ with $y\in\ar B(0,s)$ and a chain of a bounded number of these balls on the sphere prove \eqref{eqn9.13}.
\end{proof}

\begin{lemma}
\label{lem:radial-limits}
Let $\rho_0>0$.
\begin{itemize}
\item[(a)] Let $1<p<n$, and let $u$ be positive and $p$-harmonic in $\rn{n}\sem\bar B(0,\rho_0)$, with $u(x)\to0$ as $x\to\infty$.
Put $\mathbf g(x)=|x|^{(p-n)/(p-1)}$ and, for $r>\rho_0>0$,
\[
\mathbf a(r)=\inf_{\ar B(0,r)}u/\mathbf g,
\qquad
\mathbf A(r)=\sup_{\ar B(0,r)}u/\mathbf g.
\]
Then
\begin{align}
\label{eqn9.33}
\mathbf a(r)\leq\mathbf a(s)\leq\mathbf A(s)\leq\mathbf A(r)
\quad(\rho_0<r<s),
\end{align}
and $\ell=\lim_{x\to\infty}u(x)/\mathbf g(x)$ exists with
$0<\mathbf a(r)\leq\ell\leq\mathbf A(r)<\infty$.
\item[(b)] Let $p\geq n$, and let $G$ be continuous on $\rn{n}\sem B(0,\rho_0)$ and $p$-harmonic in $\rn{n}\sem\bar B(0,\rho_0)$.
Suppose $k=G-F$ is bounded there, with $F$ as in \eqref{eqn2.7}.
For $r\geq\rho_0$ put $m(r)=\inf_{\ar B(0,r)}k$ and $M(r)=\sup_{\ar B(0,r)}k$.
Then $m$ is non-decreasing, $M$ is non-increasing, and
\begin{align}
\label{eqn9.16}
\begin{split}
&(i)\quad m(r)\leq k(x)\leq M(r)\quad(|x|\geq r\geq\rho_0),\\
&(ii)\quad k(\infty)=\lim_{x\to\infty}k(x)\ \mbox{exists finitely, and}\
m(r)\leq k(\infty)\leq M(r).
\end{split}
\end{align}
\end{itemize}
\end{lemma}
\begin{proof}
In $(a)$, compare $u$ with $\mathbf A(r)\mathbf g+\ep$ and $\mathbf a(r)\mathbf g-\ep$ on annuli with fixed inner radius $r$ and large outer radius.
Since $u,\mathbf g\to0$ at infinity, comparison and $\ep\to0$ give $\mathbf a(r)\mathbf g\leq u\leq\mathbf A(r)\mathbf g$ for $|x|\geq r$, proving \eqref{eqn9.33}.
In $(b)$, compare $G$ with
$(1+\ep)F+M(r)-\ep F(r)$ and
$(1-\ep)F+m(r)+\ep F(r)$.
These have the required values on the inner sphere and the required ordering on large outer spheres because $k$ is bounded and $F\to\infty$.
Letting $\ep\to0$ gives \eqref{eqn9.16} $(i)$ and the monotonicity of $m,M$.

Fix $\rho_1>\rho_0$.
In either case the two spherical extrema are bounded and have finite limits.
For $r\geq2\rho_1$, apply Lemma \ref{lem9.7} with $s=2r$ on $B(0,4r)\sem\bar B(0,r)$.
In $(a)$ take $v_1=u$, $v_2=\mathbf a(r)\mathbf g$, and $\tau=\mathbf a(r)r^{(1-n)/(p-1)}$.
The interior estimate $|\nabla u(x)|\leq c|x|^{-1}u(x)$, the preceding comparison, and
$|\nabla\mathbf g(x)|=\frac{n-p}{p-1}|x|^{-1}\mathbf g(x)$
verify \eqref{eqn9.12}, with $\hat c$ depending only on $p,n,\mathbf A(\rho_1)/\mathbf a(\rho_1)$.
In $(b)$ take $v_1=G$, $v_2=F+m(r)$, and $\tau=r^{(1-n)/(p-1)}$.
For $|x|\geq2\rho_1$, put $B_x=B(x,|x|/2)$ and apply the interior
gradient estimate to $G-\inf_{B_x}G+\eta$ with radius $|x|/8$,
then let $\eta\downarrow0$.
The comparison bounds give
$\mathop{\rm osc}_{B_x}G
\leq F(3|x|/2)-F(|x|/2)+M(\rho_1)-m(\rho_1)$.
The first difference is $\ga\log3$ when $p=n$ and is bounded by
$c|x|^\xi$ when $p>n$. Thus
$|\nabla G(x)|\leq c|\nabla F(x)|$, with $c$ depending only on
$p,n,\rho_1,M(\rho_1)-m(\rho_1)$. The radial derivative of $F$
gives the lower gradient bound for $v_2$.

Writing $a,A$ for $\mathbf a,\mathbf A$ in $(a)$ and for $m,M$ in $(b)$, \eqref{eqn9.13} on $\ar B(0,2r)$ gives in both cases
$A(2r)-a(r)\leq c[a(2r)-a(r)]$
with $c$ independent of $r$.
In $(a)$ the common positive factor $\mathbf g(2r)$ cancels.
Letting $r\to\infty$ shows that the two limiting extrema coincide.
The comparison bounds then give the stated limits. In $(a)$ positivity follows from $\mathbf a(\rho_1)>0$.
This is the Harnack argument used in \cite[(4.16)--(4.17)]{ALSV}.
\end{proof}

\begin{definition}
\label{defn3.6}
Let $O\subset \rn{n}$ be an open set and let $v$ be an upper semicontinuous real valued function on $O$. We say that $v$ is a viscosity $p$-subsolution in $O$ if the following holds:
whenever $x_0\in O$ and $\ph$ is a function with continuous second partial derivatives in $O$ such that $v(x_0)=\ph(x_0)$, $v(x)<\ph(x)$ for $x\in O\sem\{x_0\}$, and $\nabla \ph(x_0)\neq 0$, then
\begin{align}
\label{eqn3.8}
\De \ph (x_0)+(p-2)\, |\nabla \ph(x_0)|^{-2} \, D^2 \ph(x_0)[\nabla \ph(x_0), \nabla \ph(x_0)]\geq 0.
\end{align}
\end{definition}
The left-hand side of \eqref{eqn3.8} multiplied by $|\nabla \ph(x_0)|^{p-2}$ is $\nabla \cdot(|\nabla \ph|^{p-2}\nabla \ph)$ at $x_0$. Thus $v$ is a viscosity $p$-subsolution in $O$ if and only if $-v$ is a viscosity $p$-supersolution in $O$ in the sense of \cite{JLM}.
As in \cite{JLM}, no inequality is imposed on test functions with $\nabla\ph(x_0)=0$.
\begin{lemma}
\label{lem3.7}
Let $O$ be a domain in $\rn{n}$ and let $v$ be a viscosity $p$-subsolution in $O$. If $D$ is a domain with $\bar D\subset O$, $\bar D$ compact, and if $h$ is continuous on $\bar D$, $p$-harmonic in $D$, with $v\leq h$ on $\ar D$, then $v\leq h$ in $D$.
\end{lemma}
\begin{proof}
By the equivalence theorem of \cite{JLM},
$-v$ is $p$-superharmonic in $O$.
Apply the comparison property in Definition~\ref{def:p-superharmonic} to $-v$ and the $p$-harmonic function $-h$.
\end{proof}

\subsection{Relative capacity and boundary regularity}
\label{subsection:capacity-parallel}
We extend the relative capacity in \eqref{eqn2.15} to arbitrary subsets of an open set $O\subset\rn{n}$, which may be unbounded.
If $U\subset O$ is open then $\capp(U,O)$ is the supremum of $\capp(K,O)$ over all compact sets $K\subset U$. If $A\subset O$ is arbitrary then
$\capp(A,O)$ is the infimum of $\capp(U,O)$ over all open sets $U$ with $A\subset U\subset O$. This is the $p$-capacity of the condenser $(A,O)$ in \cite[Chapter 2]{HKM}.
For a compact set the two definitions give the same number
\cite[pp.~27--28]{HKM}. Thus if $1<p<n$ and $E$ is compact, then $\capp(E)=\capp(E,\rn{n})$.
We shall use the following facts.
\begin{align}
\label{eqn9.1}
\begin{split}
&(a) \hs{.1in} \capp(A,O)\leq \capp(A',O)\quad \mbox{whenever}\, \, A\subset A'\subset O.\\
&(b) \hs{.1in} \capp(A,O')\leq \capp(A,O)\quad \mbox{whenever}\, \, A\subset O\subset O' \, \, \mbox{and $O,O'$ are open}.\\
&(c) \hs{.1in} \capp(\rho A+x,\rho\, O+x)=\rho^{n-p}\, \capp(A,O)\quad \mbox{whenever}\, \, A\subset O, \, \, \rho>0, \, \, x\in \rn{n}.\\
&(d) \hs{.1in} \mbox{If $K\subset O$ is compact, $\psi$ is continuous on $O$, $\psi\in W_0^{1,p}(O)$, and $\psi\geq 1$ on $K$,}\\
&\hs{.3in}\mbox{then}\, \,  \capp(K,O)\leq \int_O|\nabla \psi|^p\, dx.\\
&(e)\hs{.1in} \mbox{If $0<r<R<\infty$ then}\, \, \capp(\bar B(x,r),B(x,R))=\\
&\hs{.3in} \si_{n-1}\left(\frac{|n-p|}{p-1}\right)^{p-1}\left|R^{\frac{p-n}{p-1}}-r^{\frac{p-n}{p-1}}\right|^{1-p}\, \, \mbox{when}\, \, p\neq n,\\
&\hs{.3in}\mbox{and}\, \, \mbox{Cap}_n(\bar B(x,r),B(x,R))=\si_{n-1}\left(\log (R/r)\right)^{1-n}.\\
&(f) \hs{.1in} \mbox{If $p>n$ and $r>0$ then}\, \, \capp(\{x\},B(x,r))=\si_{n-1}\left(\frac{p-n}{p-1}\right)^{p-1}r^{n-p}.
\end{split}
\end{align}
Here $(a), (b)$ are Theorem 2.2 $(i), (ii)$ in \cite{HKM}, and $(c)$ follows from a change of variables.
For $(d)$, the admissible functions in \eqref{eqn2.15} may be taken
in $W_0^{1,p}(O)\cap C(O)$ \cite[pp.~27--28]{HKM}.
Part $(e)$ is (2.13) in \cite[Example 2.12]{HKM}. Part $(f)$ follows from $(e)$ by letting the inner radius tend to $0$, since condenser capacity is continuous on decreasing compact sets \cite[Theorem~2.2 $(iv)$]{HKM}.

As in \cite[(6.17)]{HKM} we say that a set $A\subset \rn{n}$ is $p$-thick at $x_0\in\rn{n}$ if
\begin{align}
\label{eqn9.2}
\int_0^1\left(\frac{\capp(A\cap B(x_0,t),B(x_0,2t))}{\capp(B(x_0,t),B(x_0,2t))}\right)^{\frac{1}{p-1}}\, \frac{dt}{t}=\infty.
\end{align}
By \eqref{eqn9.1} $(a)$ the quotient in \eqref{eqn9.2} is at most $1$ and does not decrease when $A$ is replaced by a larger set.
By \eqref{eqn9.1} $(a),(e)$,
\begin{align}
\label{eqn9.8}
\capp(B(x_0,t),B(x_0,2t))\leq \capp(\bar B(x_0,t),B(x_0,2t))=c_1\, t^{n-p}
\end{align}
where $c_1$ depends only on $p$ and $n$.

The following lemma concerns the sets in \eqref{eqn2.4}.
\begin{lemma}
\label{lem4.3}
Let $E_1, E_2$ be nonempty compact sets in $\rn{n}$,  $0<\la<1$, and $E=\la E_1+(1-\la)E_2$. Then
\[
\la \hat E_1+(1-\la)\hat E_2\subset \hat E.
\]
If $\hat E_1$ and $\hat E_2$ are convex, this inclusion is an equality.
\end{lemma}
\begin{proof}
Let $y\in \hat E_1, z\in \hat E_2$,  and suppose that $x=\la y+(1-\la)z\in \Om_E$. Since $\Om_E$ is open, connected, and contains a neighbourhood of $\infty$, there exists a continuous function $\Ga:[0,\infty)\to \Om_E$
with $\Ga(0)=x$ and $|\Ga(t)|\to\infty$ as $t\to\infty$. Put $\Ga_1(t)=\la^{-1}(\Ga(t)-(1-\la)z)$. Then $\Ga_1(0)=y$ and $|\Ga_1(t)|\to\infty$. Some $t_1\geq 0$ satisfies $\Ga_1(t_1)\in E_1$, since
otherwise $\Ga_1([0,\infty))$ is an unbounded connected subset of $\rn{n}\sem E_1$ containing $y$, contrary to $y\in \hat E_1$.
Let $y_1=\Ga_1(t_1)$ and put $\Ga_2(t)=(1-\la)^{-1}(\Ga(t_1+t)-\la y_1)$. Then $\Ga_2(0)=z$, and as above
$z_2=\Ga_2(t_2)\in E_2$ for some $t_2\geq 0$. Thus
$\Ga(t_1+t_2)=\la y_1+(1-\la)z_2\in E$,
which contradicts $\Ga(t_1+t_2)\in \Om_E$. Hence $x\in \hat E$.
If the filled sets are convex, their Minkowski combination is a
compact convex set containing $E$, so it also contains $\hat E$.
This gives the reverse inclusion.
\end{proof}

We use boundary regularity in the sense of the Wiener criterion.
\begin{definition}
\label{defn9.5}
Let $1<p<\infty$. A nonempty compact set $E\subset\rn{n}$ is
called $p$-regular if $\hat E$ is $p$-thick at each point of
$\ar\hat E$, in the sense of \eqref{eqn9.2}. A point $x_0\in\partial\hat E$ is called regular if $\hat E$ is $p$-thick at $x_0$.
\end{definition}
\noindent For $p\leq n$, the Wiener test \cite{KM} shows that $x_0$ is
regular in this sense if and only if it is a regular boundary point
for the Dirichlet problem in $B(0,R)\setminus\hat E$, whenever
$\hat E\Subset B(0,R)$.
The sufficiency part of the Wiener criterion
\cite[Theorem~6.27]{HKM} says that the Sobolev Dirichlet solution
in $B(0,R)\setminus\hat E$ attains its continuous boundary datum
at each regular point $x_0$, whenever $\hat E\Subset B(0,R)$.
For $1<p\leq n$, the converse is proved in
\cite{KM}. For $p>n$, every boundary point is regular
by Lemma~\ref{lem9.6} $(b)$.
Since $\ar\hat E\subset E\subset \hat E$, the set $E$ is $p$-regular if $E$ is $p$-thick at each of its points.
The converse fails. For example, $E=\mathbb S^{n-1}\cup\{0\}$ has $\hat E=\bar B(0,1)$, so $E$ is $p$-regular. If $1<p\leq n$, however, $E$ is not $p$-thick at $0$, since $E\cap B(0,t)=\{0\}$ for $0<t<1$ and a point has zero relative $p$-capacity.

\begin{lemma}
\label{lem9.6}
Let $1<p<\infty$ and let $E$ be a nonempty compact set in $\rn{n}$.
\begin{align*}
&(a) \hs{.1in} \mbox{If $E$ satisfies the corkscrew condition in Definition \ref{defn2.4}, then $E$ is $p$-regular.}\\
&(b) \hs{.1in} \mbox{If $p>n$ then $E$ is $p$-regular.}\\
&(c) \hs{.1in} \mbox{If $p>n-1$ and each point of $\ar\hat E$ belongs to a connected subset of $\hat E$ which}\\
&\hs{.3in}\mbox{contains more than one point, then $E$ is $p$-regular. In particular, if $p>n-1$}\\
&\hs{.3in}\mbox{and no component of $E$ is a single point, then $E$ is $p$-regular.}\\
&(d) \hs{.1in} \mbox{If $E$ is $p$-regular, $E'$ is a nonempty compact set, and $0<\la\leq 1$, then}\\
&\hs{.3in} \mbox{$\la E+(1-\la)E'$ is $p$-regular. Also $\rho E+x$ is $p$-regular when $\rho>0$ and $x\in \rn{n}$.}\\
&(e) \hs{.1in} \mbox{Let $E$ be $p$-regular, $R>0$ with $E\subset B(0,R)$, and $a\in\re$. There exists a unique}\\
&\hs{.3in} \mbox{continuous function $v$ on $\bar B(0,R)$ such that $v$ is $p$-harmonic in $B(0,R)\sem E$,}\\
&\hs{.3in} \mbox{$v\equiv 0$ on $E$, and $v\equiv a$ on $\ar B(0,R)$. Moreover, $v\equiv0$ on $\hat E$ and $v\in W^{1,p}(B(0,R))$.}
\end{align*}
\end{lemma}
\begin{proof}
For $(a)$, let $w\in E$ and $0<r<r_E$, and choose the ball
$B(y,r/c_E)\subset E\cap B(w,r)$ from Definition~\ref{defn2.4}.
Since $B(w,2r)\subset B(y,3r)$, \eqref{eqn9.1} $(a),(b),(e)$ give
\[
\capp(E\cap B(w,r),B(w,2r))
\geq\capp(\bar B(y,r/(2c_E)),B(y,3r))
=c^{-1}r^{n-p},
\]
where $c$ depends only on $p,n,c_E$.
At each
$x_0\in\partial\hat E\subset E$, the inclusion $E\subset\hat E$
and \eqref{eqn9.8} bound the quotient in \eqref{eqn9.2} for $A=\hat E$
below by a positive constant for $0<t<r_E$. Hence $\hat E$ is $p$-thick at $x_0$.

For $(b)$ let $p>n$, $x_0\in\partial\hat E\subset E$, and $0<t\leq 1$. From \eqref{eqn9.1} $(a), (f)$ we have $\capp(E\cap B(x_0,t),B(x_0,2t))\geq \capp(\{x_0\},B(x_0,2t))= c^{-1}t^{n-p}$. With \eqref{eqn9.8}, this bounds the quotient in \eqref{eqn9.2} below by a positive constant, so $E$, and hence $\hat E$, is $p$-thick at $x_0$.

The proof of $(c)$ uses the following spherical Morrey inequality, obtained in local coordinates and by scaling from \cite{GT}. If $p>n-1$, $r>0$, and $f$ is continuously differentiable near $\ar B(0,r)$, then
\begin{align}
\label{eqn9.9}
\left(\sup_{\ar B(0,r)}f-\inf_{\ar B(0,r)}f\right)^p
\leq c\,r^{p-n+1}\int_{\ar B(0,r)}|\nabla f|^p\,d\mathcal{H}^{n-1}.
\end{align}
Let $p>n-1$, $x_0\in \ar\hat E$, and let $C\subset \hat E$ be a connected set containing $x_0$ and a point $x_1\neq x_0$. We may assume that $x_0=0$. Let $0<t\leq \min\{1,|x_1|\}$ and put $K_t=\hat E\cap\bar B(0,t/2)$.
We claim that
\begin{align}
\label{eqn9.10}
\capp(K_t,B(0,2t))\geq c^{-1}\, t^{n-p}.
\end{align}
By \eqref{eqn9.1} $(a)$ and \eqref{eqn9.8}, this bounds the quotient in \eqref{eqn9.2} for $\hat E$ at $0$ below by a positive constant for $0<t\leq \min\{1,|x_1|\}$, so $\hat E$ is $p$-thick at $0$.
To prove \eqref{eqn9.10} let $0<r<t/2$. Since $C$ is connected and meets both $B(0,r)$ and $\rn{n}\sem \bar B(0,r)$, it meets $\ar B(0,r)$.
Hence $K_t\cap \ar B(0,r)\neq \es$. Thus if $\psi\in C_0^\infty(B(0,2t))$ with $\psi\geq 1$ on $K_t$, then
\begin{align}
\label{eqn9.11}
\sup_{\ar B(0,r)}\psi\geq 1\quad \mbox{whenever}\, \, 0<r<t/2.
\end{align}
Suppose first that $\inf_{\ar B(0,r)}\psi\leq 1/2$ whenever $t/4<r<t/2$. Then from \eqref{eqn9.11} and \eqref{eqn9.9},
\[
\int_{\rn{n}}|\nabla \psi|^p\, dx\geq \int_{t/4}^{t/2}\int_{\ar B(0,r)}|\nabla \psi|^p\, d\mathcal{H}^{n-1} dr\geq c^{-1}\int_{t/4}^{t/2}r^{n-1-p}\, dr\geq c^{-2}\, t^{n-p}.
\]
Otherwise $\psi>1/2$ on $\ar B(0,r)$ for some $r\in (t/4,t/2)$. Since $\psi=0$ on $\ar B(0,2t)$, H\"{o}lder's inequality gives for $\om\in\mathbb{S}^{n-1}$
\begin{align*}
\frac{1}{2}<\int_r^{2t}|\nabla\psi(s\om)|\, ds
\leq \left(\int_r^{2t}|\nabla \psi(s\om)|^p s^{n-1}ds\right)^{\frac{1}{p}}\left(\int_{t/4}^{2t}s^{-\frac{n-1}{p-1}}ds\right)^{1-\frac{1}{p}}.
\end{align*}
The last integral is $c\, t^{(p-n)/(p-1)}$. Raising to the $p$-th power and integrating over $\mathbb{S}^{n-1}$ gives
$\int_{\rn{n}}|\nabla \psi|^p\, dx\geq c^{-1}\, t^{n-p}$.
Taking the infimum over $\psi$ proves \eqref{eqn9.10} and the first statement in $(c)$. If no component of $E$ is a single point and $x_0\in \ar \hat E$, then $x_0\in E$ and the component of $E$ containing $x_0$ is a connected subset of $\hat E$ with more than one point.

To prove $(d)$ let $E''=\la E+(1-\la)E'$ and $x\in \ar\widehat{E''}$. Then $x=\la y+(1-\la)z$ for some $y\in E$ and $z\in E'$.
By Lemma \ref{lem4.3} when $\la<1$,
$\la \hat E+(1-\la)z\subset \widehat{E''}$,
and this is also true when $\la=1$. Hence $y$ is not an interior point of $\hat E$, since otherwise $x$ would be an interior point of $\widehat{E''}$. Thus $y\in\ar\hat E$, and $\hat E$ is $p$-thick at $y$ by assumption. If $t>0$ then
\begin{align*}
(\la \hat E+(1-\la)z)\cap B(x,t)=\la\, (\hat E\cap B(y,t/\la))+(1-\la)z,\quad
B(x,2t)=\la\,  B(y,2t/\la)+(1-\la)z.
\end{align*}
By \eqref{eqn9.1} $(a), (c)$, the quotient in \eqref{eqn9.2} for $\widehat{E''}$ at $x$ and $t$ is at least the quotient for $\hat E$ at $y$ and $t/\la$.
Changing variables in the integral and using $\la\leq 1$ we conclude that $\widehat{E''}$ is $p$-thick at $x$.
Here the integral in \eqref{eqn9.2} over $[a,1]$ is finite when $0<a<1$, since the quotient is at most $1$.
The same argument shows that $\rho E+x$ is $p$-regular.

For $(e)$, note that $\hat E\subset B(0,R)$. Choose a smooth datum $\psi_0$ equal to zero near
$\hat E$ and to $a$ near $\ar B(0,R)$. Theorem~3.17 of \cite{HKM} gives the variational
solution $v$ in $O=B(0,R)\sem\hat E$ with $v-\psi_0\in W_0^{1,p}(O)$.
By $p$-regularity at $\ar\hat E$ and the exterior ball condition at the outer sphere,
\cite[Theorem~6.27]{HKM} gives the continuous boundary values. Put $v=0$ on $\hat E$. Extending $v-\psi_0$ by zero gives $v\in W^{1,p}(B(0,R))$.
The open set $B(0,R)\sem E$ is the disjoint union of the open sets $O$ and $\hat E\sem E$, and $v=0$ on the second.
Hence $v$ is $p$-harmonic in $B(0,R)\sem E$. If $\ti v$ is another such function, comparison in each bounded component of $\rn{n}\sem E$ gives $\ti v=0$ on $\hat E$, and comparison in $O$ gives $\ti v=v$.
\end{proof}

\subsection{Capacitary functions and Green's functions}
\label{section4}

We construct the exterior functions for $p$-regular sets and prove their
uniqueness and asymptotics, adapting \cite{AGHLV,ALSV}. We begin with $1<p<n$.
If $E, E'$ are compact sets with $E\subset E'$, $\rho>0$, and $x\in \rn{n}$, then from \eqref{eqn1.4},
\begin{align}
\label{eqn4.1}
\capp(E)\leq \capp(E') \quad \mbox{and}\quad \capp(\rho E+x)=\rho^{n-p}\, \capp(E)
\end{align}
when $1<p<n$. Also $\capp(\bar B(x,\rho))>0$ (see Example 2.12 in \cite{HKM}).
For $R>0$ let $\chi_R(x)=\chi_1(x/R)$, where $\chi_1\in C_0^\infty(B(0,2))$ is radial, $0\leq\chi_1\leq1$, $\chi_1\equiv 1$ on $B(0,1)$, and $|\nabla \chi_1|\leq 2$. From H\"{o}lder's inequality and $1/p-1/p^*=1/n$, if $f\in L^{p^*}(\rn{n})$ then
\begin{align}
\label{eqn4.6}
R^{-1}\, \|f\|_{L^p(B(0,2R)\sem B(0,R))}\leq c\, \|f\|_{L^{p^*}(\rn{n}\sem B(0,R))}\to 0 \quad \mbox{as}\, \, R\to\infty.
\end{align}
\begin{lemma}
\label{lem4.1}
Let $1<p<n$ and let $E$ be a $p$-regular compact set. Then $\capp(E)>0$, and there exists a unique continuous function $u$ on $\rn{n}$ satisfying \eqref{eqn2.6}. Moreover,
\begin{align}
\label{eqn4.2}
\begin{split}
&(a)\hs{.1in} 0<u<1 \, \, \mbox{in}\, \, \Om_E, \, \, u\equiv 1 \, \, \mbox{on}\, \, \hat E, \, \, \mbox{and}\, \, u(x)\to 0 \, \, \mbox{as}\, \, x\to\infty.\\
&(b) \hs{.1in} \capp(E)=\int_{\rn{n}\sem E}|\nabla u|^p\, dy.\\
&(c)\hs{.1in} \lim_{x\to\infty} |x|^{\frac{n-p}{p-1}}\, u(x)=\frac{p-1}{n-p}\, \si_{n-1}^{-\frac{1}{p-1}}\,  \left[\capp(E)\right]^{\frac{1}{p-1}}.
\end{split}
\end{align}
\end{lemma}

\begin{proof}
For a positive integer $m$ with $E\subset B(0,m)$, let $u_m=1-v$ where $v$ is the function in Lemma \ref{lem9.6} $(e)$ with $R=m$ and $a=1$, and put $u_m\equiv 0$ on $\rn{n}\sem B(0,m)$.
Then $u_m$ is continuous on $\rn{n}$, $p$-harmonic in $B(0,m)\sem E$, $u_m\equiv 1$ on $E$, and $u_m\in W_0^{1,p}(B(0,m'))$ for $m'>m$. By comparison, $0\leq u_m\leq u_{m+1}\leq 1$ on $\rn{n}$. Let $u=\lim_{m\to\infty} u_m$.
Lemma \ref{lem3.1} for $u_m+1$ and Ascoli's theorem give
\begin{align}
\label{eqn4.3}
\{u_{m}, \nabla u_m\} \, \, \mbox{converges to}\, \, \{u, \nabla u\}\, \, \mbox{uniformly on compact subsets of}\, \, \rn{n}\sem E,
\end{align}
so $u$ is $p$-harmonic in $\rn{n}\sem E$. Also $u\equiv 1$ on $E$ and $u_m\leq u\leq 1$,  so $u$ is continuous on $\rn{n}$.
If $\psi\in C_0^\infty(B(0,m))$ with $\psi\geq 1$ on $E$, then $u_m-\min\{\psi,1\}$ is continuous on the closure of $B(0,m)\sem E$ and vanishes on its boundary, so it belongs to $W_0^{1,p}(B(0,m)\sem E)$ by \cite[Lemma~1.26]{HKM}.
Testing \eqref{eqn2.1} for $u_m$ with this function and using H\"{o}lder's inequality we get
\begin{align}
\label{eqn4.4}
\int_{\rn{n}}|\nabla u_m|^p\, dy\leq \inf\left\{\int_{\rn{n}}|\nabla \psi|^p dy: \,  \psi\in C_0^\infty(B(0,m)), \, \psi \geq 1 \, \, \mbox{on}\, \, E\right\}.
\end{align}
The right-hand side decreases to $\capp(E)$ as $m\to\infty$. Since $u$ is locally in $W^{1,p}(\rn{n})$, as a weak limit of $\{u_m\}$, and $\nabla u=0$ almost everywhere on $E$, Sobolev's theorem, \eqref{eqn4.3}, and Fatou's lemma give
\begin{align}
\label{eqn4.5}
\|u\|_{p^*}\leq c \, \left[\capp(E)\right]^{1/p} \quad \mbox{and}\quad \int_{\rn{n}}|\nabla u|^p\, dy=\int_{\rn{n}\sem E}|\nabla u|^p\, dy\leq \capp(E),
\end{align}
where the norm is relative to $L^{p^*}(\rn{n})$. Thus $u$ satisfies \eqref{eqn2.6}. If $\capp(E)=0$, then \eqref{eqn4.5} and continuity give $u\equiv0$, contrary to $u=1$ on $E$. Hence $\capp(E)>0$.

To prove uniqueness, suppose that $v$ is a continuous function on $\rn{n}$ which also satisfies \eqref{eqn2.6}. Given $\ep>0$ let
$\he=\max\{u-v-\ep,0\}-\max\{v-u-\ep,0\}$.
Since $u-v$ is continuous and vanishes on $E$, $\he\, \chi_R$ has compact support in $\rn{n}\sem E$. Testing \eqref{eqn2.1} for $u$ and for $v$ with this function and subtracting, we get
\begin{align*}
\int_{\{|u-v|>\ep\}}&\lan |\nabla u|^{p-2}\nabla u-|\nabla v|^{p-2}\nabla v,\nabla u-\nabla v\ran\, \chi_R\, dy\\
&=-\int_{\rn{n}}\lan |\nabla u|^{p-2}\nabla u-|\nabla v|^{p-2}\nabla v,\nabla \chi_R\ran\, \he\, dy.
\end{align*}
By H\"{o}lder's inequality the absolute value of the right-hand side is bounded by
\[
2R^{-1}(\|\, |\nabla u|\, \|_p^{p-1}+\|\, |\nabla v|\, \|_p^{p-1})\|u-v\|_{L^p(B(0,2R)\sem B(0,R))},
\]
which tends to $0$ as $R\to\infty$ by \eqref{eqn4.6} and \eqref{eqn2.6} $(c)$. The integrand on the left-hand side is non-negative, and is positive where $\nabla u\neq\nabla v$ and $\chi_R>0$.
Hence $\nabla u=\nabla v$ almost everywhere on $\{|u-v|>\ep\}$. Also $\nabla(u-v)=0$ almost everywhere on $\{u=v\}$ \cite[Corollary~1.21]{HKM}. Since $\ep>0$ is arbitrary, $u-v$ is constant on $\rn{n}$, and so $u\equiv v$ because $u=v$ on $E$.

Next, if $E\subset B(0,|x|/2)$, apply \eqref{eqn3.1} to $u+\eta$ with center $x$ and radius $|x|/8$, integrate over $B(x,|x|/8)$, and let $\eta\to0$. This gives
\begin{align}
\label{eqn4.7}
u(x)^{p^*}\leq c\, |x|^{-n}\int_{\rn{n}\sem B(0,|x|/2)} u^{p^*}\, dy.
\end{align}
The right-hand side of \eqref{eqn4.7} tends to $0$ as $x\to\infty$ by \eqref{eqn4.5}. This fact, $u\equiv 1$ on $E$,  comparison in the bounded components of $\rn{n}\sem E$, and the strong maximum principle in $\Om_E$ give \eqref{eqn4.2} $(a)$.

To prove \eqref{eqn4.2} $(b)$ let $E\subset B(0,R)$ and $\ep>0$.
The function $(1+\ep)u\chi_R$ is continuous with compact support, has distributional gradient in $L^p(\rn{n})$, and is larger than $1$ in a neighbourhood of $E$. Approximating it by
convolutions and using \eqref{eqn1.4}, \eqref{eqn4.6} we deduce that
\[
\left[\capp(E)\right]^{1/p}\leq (1+\ep)\left(\| \, |\nabla u|\, \|_p+ c\, \|u\|_{L^{p^*}(\rn{n}\sem B(0,R))}\right).
\]
Letting $R\to\infty$ and then $\ep\to 0$ we get $\capp(E)\leq \int_{\rn{n}}|\nabla u|^p dy$. With \eqref{eqn4.5} this gives \eqref{eqn4.2} $(b)$.

The continuous extension of $u$ is $p$-superharmonic on $\rn{n}$.
Indeed, if a $p$-harmonic function $h$ on a bounded open set $D$ satisfies $h\leq u$ on $\ar D$, then $h\leq1$ in $D$.
Thus $h\leq u$ on $D\cap E$, and comparison in $D\sem E$ gives $h\leq u$ in $D$.
Since $u$ is locally in $W^{1,p}(\rn{n})$, \cite[Theorem 21.2]{HKM} gives a positive Borel measure $\mu$, as in \cite[Lemma 3.4]{AGHLV}, with
\begin{align}
\label{eqn4.8}
\int_{\rn{n}}|\nabla u|^{p-2}\lan\nabla u,\nabla\ph\ran\,dy=\int\ph\,d\mu
\end{align}
for every $\ph\in C_0^\infty(\rn{n})$, and by mollification for continuous, compactly supported $\ph\in W^{1,p}(\rn{n})$. The support of $\mu$ is contained in $E$ because $u$ is $p$-harmonic off $E$.
Let $E\subset B(0,R)$. Using \eqref{eqn4.8} with $\ph=\chi_R u$ and $u\equiv 1$ on $E$ we obtain
\[
\mu(E)=\int_{\rn{n}}|\nabla u|^p\, \chi_R\, dy+\int_{\rn{n}}u\, |\nabla u|^{p-2}\lan \nabla u,\nabla \chi_R\ran\, dy.
\]
The last integral is at most $2R^{-1}\|\, |\nabla u|\, \|_p^{p-1}\|u\|_{L^p(B(0,2R)\sem B(0,R))}$ in absolute value, which tends to $0$ as $R\to\infty$ by \eqref{eqn4.6}. By \eqref{eqn4.2} $(b)$,
\begin{align}
\label{eqn4.9}
\mu(E)=\capp(E).
\end{align}

To prove \eqref{eqn4.2} $(c)$, choose $\rho_0$ with $E\subset B(0,\rho_0)$. By \eqref{eqn4.2} $(a)$, Lemma \ref{lem:radial-limits} $(a)$ applies to $u$ and gives the limit $\ell$ of $u/\mathbf g$ at $\infty$, with $\sup_{|y|\geq r}|u/\mathbf g-\ell|\leq\mathbf A(r)-\mathbf a(r)\to0$ as $r\to\infty$.
For $R>\rho_0$ put $\hat u_R(x)=u(Rx)/\mathbf g(R)$. Since $\mathbf g(Rx)=\mathbf g(R)\mathbf g(x)$, we get $\hat u_R\to\ell\mathbf g$ uniformly on $\bar B(0,4)\sem B(0,1/2)$ as $R\to\infty$.
Lemma \ref{lem3.1} gives uniform H\"older bounds for $\nabla\hat u_R$ on $\bar B(0,2)\sem B(0,1)$, so Ascoli's theorem gives $\nabla\hat u_R\to\ell\nabla\mathbf g$ uniformly there.
Since $\mathbf g(R)^{p-1}R^{n-p}=1$, a change of variables in \eqref{eqn4.8} with $\ph=\chi_R$, together with \eqref{eqn4.9}, gives
\[
\capp(E)=\mu(E)=\int_{\rn{n}}|\nabla\hat u_R|^{p-2}\lan\nabla\hat u_R,\nabla\chi_1\ran\,dx
\longrightarrow\ell^{p-1}\int_{\rn{n}}|\nabla\mathbf g|^{p-2}\lan\nabla\mathbf g,\nabla\chi_1\ran\,dx.
\]
Here $|\nabla\mathbf g|^{p-2}\nabla\mathbf g(x)=-\left(\frac{n-p}{p-1}\right)^{p-1}|x|^{-n}x$ and $\chi_1$ is radial, equal to $1$ on $\ar B(0,1)$ and to $0$ on $\ar B(0,2)$. Integration in polar coordinates shows that the last integral equals $\si_{n-1}\left(\frac{n-p}{p-1}\right)^{p-1}$.
Solving for $\ell$ gives \eqref{eqn4.2} $(c)$.
\end{proof}

Next we consider the case $p\geq n$. Recall the definition of $F$ in \eqref{eqn2.7}. Then $F$ is radial and increasing in $|x|$, and we write $F(s)$ for its value on $\ar B(0,s)$.
The next lemma compares the functions of Lemma~\ref{lem9.6} $(e)$ with capacitary functions of condensers.
\begin{lemma}
\label{lem9.9}
Let $n\leq p<\infty$ and let $E$ be a $p$-regular compact set. Let $0<\rho_0<r<R$ with $E\subset B(0,\rho_0)$ and $F(R)>\max\{0,F(\rho_0)\}$. 
Let $v$ be the function supplied by Lemma~\ref{lem9.6} $(e)$
with $a=F(R)$ and put
\[
\kappa_r=\capp(E,B(0,r)),\quad m'=\inf_{\ar B(0,r)}v,\quad M'=\sup_{\ar B(0,r)}v.
\]
Then $\kappa_r>0$,  $0<m'\leq M'\leq F(R)$, and
\begin{align}
\label{eqn9.18}
\begin{split}
&(a)\hs{.1in} (m')^{p-1}\, \kappa_r\leq \left(\frac{F(R)}{F(R)-F(\rho_0)}\right)^{p-1},\\
&(b)\hs{.1in} \left(1-\frac{M'}{F(R)}\right)^p\left(\frac{F(R)}{F(R)-F(r)}\right)^{p-1}\leq (M')^{p-1}\, \kappa_r.
\end{split}
\end{align}
\end{lemma}

\begin{proof}
Put $\mathbf{I}=\int_{B(0,R)}|\nabla v|^p\,dy$. Comparison gives $0\leq v\leq F(R)$.
The test functions below are continuous Sobolev
functions with zero boundary values, so they belong to
$W_0^{1,p}(B(0,R)\sem E)$ by \cite[Lemma~1.26]{HKM}.
Testing with $\min\{v,t\}-(t/F(R))v$ gives
\begin{align}
\label{eqn9.19}
\int_{\{v<t\}}|\nabla v|^p\,dy=\frac{t}{F(R)}\mathbf{I}
\quad(0<t\leq F(R)).
\end{align}
Let $g=F(R)\max\{F-F(\rho_0),0\}/(F(R)-F(\rho_0))$, with $g(0)=0$.
Since $\int_{B(0,r_2)\sem B(0,r_1)}|\nabla F|^p\,dy=F(r_2)-F(r_1)$ for $0<r_1<r_2$, testing with $v-g$ and applying H\"older's inequality yields
\begin{align}
\label{eqn9.21}
\mathbf{I}\leq\int_{B(0,R)}|\nabla g|^p\,dy
=F(R)^p(F(R)-F(\rho_0))^{1-p}.
\end{align}
Comparison gives $v\geq g$ outside $B(0,\rho_0)$, hence $m'\geq g(r)>0$.
Let $U$ be the function from Lemma~\ref{lem9.6} $(e)$ in $B(0,r)$ with outer value $1$, extended by $1$ outside that ball.
As in \eqref{eqn4.4}, minimization and the admissible function $1-U$ give
\begin{align}
\label{eqn9.22}
\kappa_r=\int_{B(0,r)}|\nabla U|^p\,dy>0.
\end{align}
Positivity holds because $U$ is continuous and its values on $E$ and $\ar B(0,r)$ differ.
Comparison gives $v\geq m'$ outside $B(0,r)$, so $\psi=1-\min\{v/m',1\}$ is admissible for $\kappa_r$.
By \eqref{eqn9.19} and \eqref{eqn9.21},
\[
\kappa_r\leq (m')^{-p}\int_{\{v<m'\}}|\nabla v|^p\,dy
=(m')^{1-p}\frac{\mathbf{I}}{F(R)}
\leq (m')^{1-p}\left(\frac{F(R)}{F(R)-F(\rho_0)}\right)^{p-1}.
\]
This proves $(a)$. For $(b)$ the case $M'=F(R)$ is immediate.
Otherwise choose $M'<t\leq F(R)$ and test with $U-\min\{v,t\}/t$.
Since $v\leq M'$ in $B(0,r)$, H\"older's inequality and \eqref{eqn9.19}, \eqref{eqn9.22} give $\mathbf{I}/F(R)\leq t^{p-1}\kappa_r$.
Letting $t\downarrow M'$ gives
\begin{align}
\label{eqn9.23}
\frac{\mathbf{I}}{F(R)}\leq(M')^{p-1}\kappa_r.
\end{align}
For almost every $\om\in\mathbb S^{n-1}$, integration on a radial
segment and weighted H\"older's inequality give
\begin{align*}
F(R)-M'\leq\int_r^R|\nabla v(s\om)|\,ds
\leq\left(\int_r^R|\nabla v(s\om)|^p s^{n-1}\,ds\right)^{1/p}
\left(\int_r^R s^{-(n-1)/(p-1)}\,ds\right)^{(p-1)/p}.
\end{align*}
The last integral is
$\si_{n-1}^{1/(p-1)}(F(R)-F(r))$ by \eqref{eqn2.7}.
Raise to the $p$th power, integrate over $\mathbb S^{n-1}$, and cancel
$\si_{n-1}$ to obtain
$(F(R)-M')^p\leq
\mathbf I(F(R)-F(r))^{p-1}$.
Together with \eqref{eqn9.23}, this proves $(b)$.
\end{proof}

\begin{lemma}
\label{lem4.2}
Let $n\leq p<\infty$ and let $E$ be a $p$-regular compact set. There exists a unique $p$-harmonic Green's function $G$ for $\rn{n}\sem E$ with pole at $\infty$ in the sense of \eqref{eqn2.9}. Moreover, if $k=G-F$ then
\begin{align}
\label{eqn4.17}
\begin{split}
&(a) \hs{.1in} G>0 \, \, \mbox{in}\, \, \Om_E \, \, \mbox{and}\, \, G\equiv 0\, \, \mbox{on}\, \, \hat E.\\
&(b) \hs{.1in} \lim_{x\to\infty}k(x)=k(\infty) \, \, \mbox{exists finitely. If $p>n$ then $k(\infty)\leq 0$, with equality}\\
&\hs{.3in}\mbox{if and only if $E$ is a single point.}\\
&(c) \hs{.1in} \mbox{If $E'$ is a $p$-regular compact set with $E\subset E'$ and Green's function $G'$,}\\
&\hs{.3in}\mbox{then}\, \, G'\leq G, \, \, k'(\infty)\leq k(\infty) \, \, \mbox{for}\, \, k'=G'-F, \, \, \mbox{and}\, \,  \cp(E)\leq \cp(E').\\
&(d) \hs{.1in} \mbox{If $\rho>0$ and $x\in \rn{n}$ then}\, \, \cp(\rho E+x)=\rho \, \cp(E)\, \, \mbox{when}\, \, p=n\, \, \mbox{while} \\
&\hs{.3in} \cp(\rho E+x)=\rho^{p-n}\, \cp(E)\, \, \mbox{when}\, \, p>n.\\
&(e)\hs{.1in} \cp(E)>0, \, \, \mbox{unless $p>n$ and $E$ is a single point.}
\end{split}
\end{align}
\end{lemma}
\noindent When $E$ is a convex body, Lemma \ref{lem4.2} is contained in Lemmas 5.2, 5.3 and Remark 5.5 of \cite{ALSV}, whose proofs we follow.
\begin{proof}
We first prove existence of $G$ when
\begin{align}
\label{eqn4.18}
0\in E\subset B(0,\rho_0).
\end{align}
Let $R>4\rho_0$ be so large that $F(R)>\max\{0,2F(\rho_0)\}$, and let $v=v(\cdot, R)$ be the function in Lemma \ref{lem9.6} $(e)$ with $a=F(R)$. Comparison gives $v\geq 0$. We claim that
\begin{align}
\label{eqn4.19}
F(x)-c_1\leq v(x,R)\leq F(x)+c_1 \quad \mbox{whenever}\, \, x\in B(0,R)\sem B(0,\rho_0),
\end{align}
for all $R\geq R_0$, where $c_1,R_0$ depend only on $p, n, \rho_0$, and also on $\kappa=\mbox{Cap}_n(E,B(0,2\rho_0))$ when $p=n$. If $p>n$, then $F-F(\rho_0)$ and $F$ are $p$-harmonic in $\rn{n}\sem \{0\}$,
continuous on $\rn{n}$, and since $0\in E\subset B(0,\rho_0)$,
\[
F-F(\rho_0)\, \leq \, v\, \leq\,  F \quad \mbox{on}\, \, \ar (B(0,R)\sem E)\subset E\cup \ar B(0,R).
\]
Thus \eqref{eqn4.19} holds with $c_1=\max\{1,F(\rho_0)\}$ by comparison in $B(0,R)\sem E$. If $p=n$, put
\[
g_s(x)=F(R)\, \frac{\log(|x|/s)}{\log(R/s)}, \quad {0<s<R}.
\]
This function is $n$-harmonic in $\rn{n}\sem\{0\}$, vanishes on $\ar B(0,s)$, and equals $F(R)$ on $\ar B(0,R)$.
Since $F(R)=\ga\log R+1$, a calculation gives $|g_s(x)-\ga\log|x|\, |\leq c(s)$ for $s\leq |x|\leq R$ and large $R$.
Comparison in $B(0,R)\sem \bar B(0,\rho_0)$ gives $v\geq g_{\rho_0}>0$ there, which is the lower bound in \eqref{eqn4.19}.
For the upper bound, let $m',M'$ be as in Lemma~\ref{lem9.9} with $r=2\rho_0$, so $\kappa=\kappa_{2\rho_0}>0$.
A Harnack chain on $\ar B(0,2\rho_0)$ gives $M'\leq c\,m'$, and \eqref{eqn9.18} $(a)$ with $F(R)\geq 2F(\rho_0)$ gives $M'\leq c\,\kappa^{-1/(n-1)}$.
Comparison with $M'+g_{2\rho_0}$ in $B(0,R)\sem\bar B(0,2\rho_0)$ and with $M'$ in $B(0,2\rho_0)\sem E$ gives the upper bound in \eqref{eqn4.19}.

From $v\geq 0$, \eqref{eqn4.19} on $\ar B(0,2\rho_0)$, and comparison in $B(0,2\rho_0)\sem E$,
\begin{align}
\label{eqn4.20}
0\leq v(x,R)\leq F(2\rho_0)+c_1=: c_2 \quad \mbox{whenever}\, \, x\in \bar B(0,2\rho_0)\, \, \mbox{and}\, \, R\geq R_0.
\end{align}
Let $U$ be the function in Lemma \ref{lem9.6} $(e)$ with $R$ replaced by $2\rho_0$ and $a=1$. By \eqref{eqn4.20} and comparison,
\begin{align}
\label{eqn4.21}
0\leq v(x,R)\leq c_2\, U(x)  \quad \mbox{whenever}\, \, x\in \bar B(0,2\rho_0)\, \, \mbox{and}\, \, R\geq R_0.
\end{align}
By \eqref{eqn4.19}, \eqref{eqn4.20}, Lemma \ref{lem3.1}, and Ascoli's theorem there is a sequence $R_l\to\infty$ such that
\begin{align}
\label{eqn4.22}
\{v(\cdot,R_l), \nabla v(\cdot, R_l)\} \, \, \mbox{converges to}\, \, \{G, \nabla G\}\, \, \mbox{uniformly on compact subsets of}\, \, \rn{n}\sem E.
\end{align}
Then $G$ is $p$-harmonic in $\rn{n}\sem E$, and if we put $G\equiv 0$ on $E$, then $0\leq G\leq c_2 U$ in $\bar B(0,2\rho_0)$ by \eqref{eqn4.21}.
Since $U$ is continuous with $U\equiv 0$ on $E$, $G$ is continuous on $\rn{n}$. Finally, \eqref{eqn4.19} holds with $v$ replaced by $G$ whenever $|x|\geq \rho_0$. Thus $G$ satisfies \eqref{eqn2.9}.

For uniqueness, suppose $G_1$ also satisfies \eqref{eqn2.9}.
Comparison gives $G,G_1\geq0$, and their common leading term
$F\to\infty$ gives $G_1/G\to1$ uniformly at infinity.
For each $\epsilon>0$, compare $G_1$ with $(1+\epsilon)G$,
and $G$ with $(1+\epsilon)G_1$, on large truncated exteriors.
The inner boundary values are zero. Letting the outer radius
increase and then $\epsilon\downarrow0$ gives $G=G_1$.

For general $E$, choose $z_0\in E$, let $\hat G$ be the Green's function for $\rn{n}\sem (E-z_0)$, and put $G(x)=\hat G(x-z_0)$. Then $G$ satisfies \eqref{eqn2.9} $(a), (b)$, and if $\hat k=\hat G-F$, then
\begin{align}
\label{eqn4.23}
G(x)-F(x)=\hat k(x-z_0)+F(x-z_0)-F(x) \quad \mbox{and}\quad
\lim_{x\to\infty}[F(x-z_0)-F(x)]=0
\end{align}
by \eqref{eqn2.7}. Thus $G$ satisfies \eqref{eqn2.9} $(c)$, and $G$ is unique as above.

Now \eqref{eqn4.17} $(a)$ follows from $G\geq 0$, $G(x)\to \infty$ as $x\to\infty$, comparison in the bounded components of $\rn{n}\sem E$,  and the strong maximum principle in $\Om_E$.
The bound \eqref{eqn4.19} for $G$, \eqref{eqn4.23}, and Lemma \ref{lem:radial-limits} $(b)$ give the finite limit $k(\infty)$.

For \eqref{eqn4.17} $(c)$, note that $G'$ and $G$ are $p$-harmonic in $\rn{n}\sem E'$, $G'=0\leq G$ on $E'$, and $G'/G\to1$ at infinity. Comparing $G'$ with $(1+\ep)G$ on large truncated exteriors and letting $\ep\downarrow0$ gives $G'\leq G$, hence $k'(\infty)\leq k(\infty)$.
The capacity inequality follows from \eqref{eqn2.10}, using $k(\infty)\leq0$ when $p>n$, proved below from $G'\leq G$ alone.

We prove the remaining part of \eqref{eqn4.17} $(b)$. Let $p>n$ and $z\in \rn{n}$. The function $x\mapsto F(x-z)$ is continuous, vanishes at $z$, and is $p$-harmonic in $\rn{n}\sem\{z\}$.
By \eqref{eqn4.23} and uniqueness, it is the Green's function for $\rn{n}\sem\{z\}$ and $k(\infty)=0$ when $E=\{z\}$.
If $z\in E$ then \eqref{eqn4.17} $(c)$ with $E, E'$ replaced by $\{z\}, E$ gives
\begin{align}
\label{eqn9.26}
G(x)\leq F(x-z)\quad \mbox{whenever}\, \, x\in \rn{n}\, \, \mbox{and}\, \, z\in E.
\end{align}
Hence $k(\infty)\leq 0$. Now suppose that $E$ contains at least two points. By \eqref{eqn4.23} we may assume \eqref{eqn4.18}. Choose $z_1\in E\sem\{0\}$.
Put $h=F-G\geq0$. Let $r_1=4\rho_0$ and $x_1=r_1z_1/|z_1|$. Then $|x_1-z_1|<|x_1|$, and since $F$ is radial and strictly increasing in $|x|$, \eqref{eqn9.26} with $z=z_1$ gives
$h(x_1)\geq F(x_1)-F(x_1-z_1)>0$.
By \eqref{eqn9.26}, $0\leq G\leq F$.
For $|x|\geq2\rho_0$, apply the interior gradient estimate to $G+\eta$ in
$B(x,|x|/2)$ with radius $|x|/8$, and then let $\eta\downarrow0$. Since
$F$ is radial, this gives
\[
|\nabla G(x)|\leq c|x|^{-1}\sup_{B(x,|x|/2)}F
\leq c'|\nabla F(x)|,
\qquad
|\nabla F(x)|=\sigma_{n-1}^{-1/(p-1)}|x|^{(1-n)/(p-1)},
\]
with constants depending only on $p,n$.
Thus \eqref{eqn9.12} holds with $s=r_1$, $v_1=F$, $v_2=G$,
$\tau=r_1^{(1-n)/(p-1)}$ and a constant $\hat c$ independent of the scale.
Lemma \ref{lem9.7} gives
$\inf_{\ar B(0,r_1)} h\geq c^{-1}h(x_1)>0$.
Hence $\sup_{\ar B(0,r_1)}(G-F)<0$, and \eqref{eqn9.16} $(ii)$ gives $k(\infty)<0$.

To prove \eqref{eqn4.17} $(d)$, note from \eqref{eqn4.23} that $k(\infty)$, and hence $\cp(E)$, is invariant under translation. Let $\rho>0$ and put $\ti G(x)=G(x/\rho)$ when $p=n$ and $\ti G(x)=\rho^{\xi}G(x/\rho)$ when $p>n$.
Then $\ti G$ is $p$-harmonic in $\rn{n}\sem \rho E$, continuous on $\rn{n}$ with $\ti G\equiv 0$ on $\rho E$, and by \eqref{eqn2.7},
$\ti G(x)-F(x)=k(x/\rho)-\ga\log \rho$ when $p=n$, and $\ti G(x)-F(x)=\rho^{\xi}\, k(x/\rho)$ when $p>n$.
The set $\rho E$ is $p$-regular by Lemma \ref{lem9.6} $(d)$. By uniqueness, $\ti G$ is its Green's function, and \eqref{eqn2.10} with $\xi(p-1)=p-n$ gives \eqref{eqn4.17} $(d)$.
Part $(e)$ is immediate from \eqref{eqn2.10} when $p=n$ and follows from $(b)$ when $p>n$.

Finally, if $E=\bar B(0,\rho)$ then $G=F-F(\rho)$ in $\rn{n}\sem B(0,\rho)$ by uniqueness. Thus $k(\infty)=-F(\rho)$ and from \eqref{eqn2.7}, \eqref{eqn2.10},
\begin{align}
\label{eqn4.29}
\cp(\bar B(0,\rho))=e^{1/\ga}\, \rho \, \, \mbox{when}\, \, p=n,\qquad
\cp(\bar B(0,\rho))=\left(\xi^{-1}\si_{n-1}^{-\frac{1}{p-1}}\right)^{p-1}\rho^{p-n} \, \, \mbox{when}\, \, p>n.
\end{align}
\end{proof}

\section{Proof of the inequalities in Theorem A}
\label{section:proofA}
\label{section9}
We first prove a pointwise comparison for $p$-regular sets. It gives
parts $(b)$ and $(c)$ directly, and outer approximation gives part $(a)$
for arbitrary nonempty compact sets.
The comparison will also be used in the proof of
Theorem~\ref{thmC}.

\subsection{A supremal convolution}
\label{section5}
In this subsection and subsection~\ref{section7}, we assume that
\begin{align}
\label{eqn5.1}
\begin{split}
&E_1, E_2\, \, \mbox{are $p$-regular compact sets in $\rn{n}$,}\, \,
\la\in (0,1), \, \, \mbox{and}\, \, E=\la E_1+(1-\la)E_2.
\end{split}
\end{align}
Then $E$ is $p$-regular by Lemma \ref{lem9.6} $(d)$.
If $1<p<n$ let $u_1,u_2,u$ be the $p$-capacitary functions for $E_1,E_2,E$, and if $p\geq n$ let $G_1,G_2,G$ be the $p$-harmonic Green's functions for $\rn{n}\sem E_1$, $\rn{n}\sem E_2$, $\rn{n}\sem E$ with pole at $\infty$ (see Lemmas \ref{lem4.1}, \ref{lem4.2}). Put
\begin{align}
\label{eqn5.2}
\begin{split}
&w_1=u_1, \, \, w_2=u_2,\, \,  w=u,\, \,  \mbox{and}\, \,  \mathbf{m}=1 \quad \mbox{when}\, \, 1<p<n,\\
&w_1=-G_1,\, \, w_2=-G_2,\, \,  w=-G,\, \,  \mbox{and}\, \, \mathbf{m}=0 \quad \mbox{when}\, \, p\geq n.
\end{split}
\end{align}
By Lemmas \ref{lem4.1}, \ref{lem4.2},  and \ref{lem3.1}, if $\ti w\in \{w_1,w_2,w\}$ and $\ti E$ is the corresponding set in $\{E_1,E_2,E\}$ then
\begin{align}
\label{eqn5.3}
\begin{split}
&(a) \hs{.1in} \ti w \, \, \mbox{is continuous on $\rn{n}$}, \, \, \ti w<\mathbf{m}\, \, \mbox{in}\, \, \rn{n}\sem \hat{\ti E},\,\, \mbox{and}\, \, \ti w\equiv \mathbf{m}\, \, \mbox{on}\, \, \hat{\ti E},\\
&(b) \hs{.1in} \ti w \, \, \mbox{is $p$-harmonic in}\, \, \rn{n}\sem \ti E \, \, \mbox{and has continuous partial derivatives in}\, \, \rn{n}\sem \ti E,\\
&(c) \hs{.1in} \lim_{x\to\infty}\ti w(x)=0\, \, \mbox{when}\, \, 1<p<n\, \, \mbox{and} \, \, \lim_{x\to\infty}\ti w(x)=-\infty\, \, \mbox{when}\, \, p\geq n.
\end{split}
\end{align}
For $x\in \rn{n}$ we define
\begin{align}
\label{eqn5.4}
w^*(x)=\sup\left\{
\min\{w_1(y), w_2(z)\} :
x=\lambda y+(1-\lambda) z, \, \,  y,  z   \in \mathbb{R}^n
\right\}.
\end{align}
If $1<p<n$ then $w^*$ is the function $u^*$ in \eqref{eqn1.6}. If $p\geq n$ then $-w^*$
is the function $G^*$ in \cite[(6.11)]{ALSV}. Our goal is to prove that
\begin{align}
\label{eqn5.5}
w^*(x)\leq w(x) \quad \mbox{whenever}\, \, x\in \rn{n}.
\end{align}
We say that $(y,z)$ is an optimal pair for $x$ if $x=\la y+(1-\la)z$ and $w^*(x)=\min\{w_1(y),w_2(z)\}$.
\begin{lemma}
\label{lem5.1}
Let $E_1,E_2,\la,E,w_1,w_2,w,\mathbf{m}$, and $w^*$ be as in \eqref{eqn5.1}-\eqref{eqn5.4}. Then
\begin{align}
\label{eqn5.6}
\begin{split}
&(a) \hs{.1in} \mbox{For each $x\in \rn{n}$ there exists an optimal pair for $x$.}\\
&(b) \hs{.1in} w^*\, \, \mbox{is continuous on $\rn{n}$ and}\, \, \min\{w_1,w_2\}\leq w^*\leq \mathbf{m}.\\
&(c)\hs{.1in} \{x: w^*(x)=\mathbf{m}\}=\la\hat E_1+(1-\la)\hat E_2\subset \hat E. \, \, \mbox{Thus}\, \, w^*\leq w\, \, \mbox{on}\, \, \hat E.\\
&(d) \hs{.1in} \mbox{If $1<p<n$ then}\, \, w^*(x)\to 0\, \, \mbox{as}\, \, x\to\infty.\, \,\mbox{If $p\geq n$ and $0<\he<1$ then}\\
&\hs{.3in}  w^*(x)-\he\, w(x)\to -\infty\, \, \mbox{as}\, \, x\to\infty.
\end{split}
\end{align}
\end{lemma}
\begin{proof}
For $x,y\in \rn{n}$ let $q(x,y)=\min\{w_1(y),w_2((x-\la y)/(1-\la))\}$. Then $q$ is continuous and $w^*(x)=\sup_{y}q(x,y)$.
Taking $y=x$ gives $w^*\geq \min\{w_1,w_2\}$, and $w^*\leq \mathbf{m}$ by \eqref{eqn5.3} $(a)$. Let $K\subset \rn{n}$ be compact and $\tau=\inf_K\min\{w_1,w_2\}$.
By \eqref{eqn5.3} $(a), (c)$, and \eqref{eqn4.2} $(a)$ when $1<p<n$, $\tau$ is larger than the limit of $w_1$ at $\infty$, so $K'=\{y: w_1(y)\geq \tau\}$ is compact. Since $q(x,y)\leq w_1(y)$,
$w^*(x)=\sup_{y\in K'}q(x,y)$ for $x\in K$. This proves \eqref{eqn5.6} $(a)$ and continuity of $w^*$ on $K$, hence \eqref{eqn5.6} $(b)$.
By \eqref{eqn5.6} $(a)$ and \eqref{eqn5.3} $(a)$, $w^*(x)=\mathbf{m}$ if and only if $x=\la y+(1-\la)z$ for some $y\in \hat E_1$ and $z\in \hat E_2$. This,  Lemma \ref{lem4.3}, and
\eqref{eqn5.3} $(a)$ for $w$ give \eqref{eqn5.6} $(c)$.

For \eqref{eqn5.6} $(d)$, if $x=\la y+(1-\la)z$ then $|y|\geq |x|$ or $|z|\geq |x|$. Hence
\begin{align}
\label{eqn5.7}
w^*(x)\leq \sup_{i\in\{1,2\}}\sup\{w_i(\ze): \, |\ze|\geq |x|\}.
\end{align}
If $1<p<n$ then \eqref{eqn5.7}, \eqref{eqn5.3} $(c)$, and $w^*\geq \min\{w_1,w_2\}> 0$ give the first part of \eqref{eqn5.6} $(d)$. If $p\geq n$, \eqref{eqn2.9} $(c)$ and the monotonicity of $F$ in $|x|$ give $c'\geq 1$ and $\rho'>0$ with
$w_i(\ze)\leq -F(\ze)+c'$ and $w(\ze)\geq -F(\ze)-c'$ for $|\ze|\geq \rho'$ and $i=1,2$.
With \eqref{eqn5.7} this gives
$w^*(x)-\he w(x)\leq -(1-\he)F(x)+2c'$
for $|x|\geq \rho'$, which proves \eqref{eqn5.6} $(d)$.
\end{proof}

The next lemma is the local calculus used in the viscosity treatment of
quasi-concave envelopes \cite{CSq,CuS} and Minkowski linear combinations
\cite{LS,Sa15}. The gradient and Hessian formulas under positive
curvature and unique support hypotheses are given in \cite{BLS}. See also \cite[(4.27)]{AGHLV}. We give the short
proof using test functions, without those geometric hypotheses.
\begin{lemma}
\label{lem5.2}
Let $\la\in(0,1)$, $y, z\in \rn{n}$,  and $x=\la y+(1-\la)z$. Let $v_1,v_2,\ph$ be real valued functions defined in neighbourhoods of $y,z,x$ respectively, such that $v_1(y)=v_2(z)=\ph(x)$ and
\begin{align}
\label{eqn5.8}
\ph(\la y'+(1-\la)z')\geq \min\{v_1(y'),v_2(z')\}
\end{align}
whenever $y'$ is in a neighbourhood of $y$ and $z'$ is in a neighbourhood of $z$.

\noindent $(a)$ If $v_1,v_2,\ph$ are differentiable at $y,z,x$ respectively and $\nabla v_1(y)\neq 0, \nabla v_2(z)\neq 0$, then there exist $\nu\in \mathbb{S}^{n-1}$ and positive numbers $a,b,d$ such that
\begin{align}
\label{eqn5.9}
\nabla v_1(y)=-a\, \nu,\quad \nabla v_2(z)=-b\, \nu, \quad \nabla \ph(x)=-d\, \nu,\quad \mbox{and}\quad \frac{1}{d}=\frac{\la}{a}+\frac{1-\la}{b}.
\end{align}

\noindent $(b)$ If in addition to the assumptions in $(a)$, $v_1,v_2,\ph$ have continuous second partial derivatives in neighbourhoods of $y,z,x$ respectively, then
\begin{align}
\label{eqn5.10}
D^2\ph(x)[e,e]\geq \la \left(\frac{d}{a}\right)^3D^2v_1(y)[e,e]+(1-\la)\left(\frac{d}{b}\right)^3D^2v_2(z)[e,e]
\end{align}
whenever $e\in \rn{n}$.
\end{lemma}
\begin{proof}
Let $s=v_1(y)=v_2(z)=\ph(x)$, $g_1=\nabla v_1(y)$, $g_2=\nabla v_2(z)$, and $g=\nabla \ph(x)$. If $\vartheta,\eta\in \rn{n}$, then \eqref{eqn5.8} with $y'=y+t\vartheta$, $z'=z+t\eta$, after division by $t>0$ and letting $t\to 0$, gives
\begin{align}
\label{eqn5.11}
\lan g, \la\vartheta+(1-\la)\eta\ran\geq \min\{\lan g_1,\vartheta\ran, \lan g_2,\eta\ran\}\quad \mbox{whenever}\, \, \vartheta,\eta\in \rn{n}.
\end{align}
Put $\vartheta_0=g_1/|g_1|^2$ and $\eta_0=g_2/|g_2|^2$.
In \eqref{eqn5.11}, take $\vartheta=\vartheta_0+\tau\ze$, $\eta=\eta_0$, with $\ze\perp g_1$ and arbitrary $\tau\in\re$. It follows that $g\perp\ze$.
Likewise, $\vartheta=\vartheta_0$ and
$\eta=\eta_0+\tau\ze$ with $\ze\perp g_2$ give $g\perp\ze$.
Hence $g=\al_1g_1=\al_2g_2$ for real $\al_1,\al_2$.
Taking $(\vartheta,\eta)=\pm(\vartheta_0,\eta_0)$ yields
\begin{align}
\label{eqn5.12}
\la\al_1+(1-\la)\al_2=1.
\end{align}
Taking $(\vartheta_0,T\eta_0)$ and $(T\vartheta_0,\eta_0)$, then letting $T\to\infty$, shows that $\al_1,\al_2\geq0$.
If either is zero then $g=0$, so both are zero, contrary to \eqref{eqn5.12}. Thus \eqref{eqn5.9} holds with
$d=|g|$, $\nu=-g/|g|$, $a=d/\al_1$, and $b=d/\al_2$.

To prove $(b)$ let $\al=d/a,\be=d/b$, and $\omega_0=\la d/a$. Then by \eqref{eqn5.9},
\begin{align}
\label{eqn5.13}
\al\, a=\be\, b=d, \quad \la\al+(1-\la)\be=1,\quad 0<\omega_0<1,\quad \mbox{and}\quad 1-\omega_0=(1-\la)d/b.
\end{align}
Fix $e\in \rn{n}$, let
$\kappa=(\al^2 D^2v_1(y)[e,e]-\be^2D^2v_2(z)[e,e])/(2a)$,
and for small $t>0$ put
$y_t=y+t\al e+t^2\kappa\nu$ and $z_t=z+t\be e$.
By Taylor's theorem, \eqref{eqn5.9}, and \eqref{eqn5.13}, as $t\to 0$,
\begin{align}
\label{eqn5.14}
\begin{split}
&v_1(y_t)=s-t\, d\, \lan \nu,e\ran+t^2\left(-a\kappa+\frac{\al^2}{2}D^2v_1(y)[e,e]\right)+o(t^2),\\
&v_2(z_t)=s-t\, d\, \lan \nu,e\ran+t^2\, \frac{\be^2}{2}D^2v_2(z)[e,e]+o(t^2).
\end{split}
\end{align}
By the choice of $\kappa$ the coefficients of $t^2$ in \eqref{eqn5.14} are equal, so $v_1(y_t)-v_2(z_t)=o(t^2)$. Using
$\min\{A,B\}\geq \omega_0 A+(1-\omega_0)B-|A-B|$, we get
\begin{align}
\label{eqn5.15}
\min\{v_1(y_t),v_2(z_t)\}\geq  s-t\, d\, \lan \nu,e\ran+t^2\, Q+o(t^2),
\end{align}
where
$Q=-\omega_0 a\kappa+\frac{\omega_0 \al^2}{2}D^2v_1(y)[e,e]+\frac{(1-\omega_0)\be^2}{2}D^2v_2(z)[e,e]$.
On the other hand, $\la y_t+(1-\la)z_t=x+te+t^2\la\kappa\nu$ by \eqref{eqn5.13},  so
\begin{align}
\label{eqn5.16}
\ph(\la y_t+(1-\la)z_t)=s-t\, d\, \lan \nu,e\ran+t^2\left(-\la d\kappa+\frac{1}{2}D^2\ph(x)[e,e]\right)+o(t^2).
\end{align}
Now $\omega_0 a=\la d$. Using \eqref{eqn5.15}, \eqref{eqn5.16} in \eqref{eqn5.8} with $(y',z')=(y_t,z_t)$, dividing by $t^2$,  and letting $t\to 0$ we conclude that
\[
D^2\ph(x)[e,e]\geq \omega_0\al^2D^2v_1(y)[e,e]+(1-\omega_0)\be^2D^2v_2(z)[e,e].
\]
This is \eqref{eqn5.10}, since $\omega_0\al^2=\la(d/a)^3$ and $(1-\omega_0)\be^2=(1-\la)(d/b)^3$ by \eqref{eqn5.13}.
\end{proof}

\begin{lemma}
\label{lem5.3}
Let $p\in(1,\infty)$ and let $\la,y,z,x$ be as in Lemma \ref{lem5.2}. Let $\ti w_1,\ti w_2$ be $p$-harmonic functions in neighbourhoods of $y,z$ respectively with $\ti w_1(y)=\ti w_2(z)=s$,
$\nabla \ti w_1(y)\neq 0$, and $\nabla \ti w_2(z)\neq 0$.
Suppose that $\ph$ has continuous second partial derivatives near $x$,
that $\ph(x)=s$, and that \eqref{eqn5.8} holds with
$v_1=\ti w_1$ and $v_2=\ti w_2$. Then $\nabla \ph(x)\neq0$ and
\begin{align}
\label{eqn5.17}
\De \ph (x)+(p-2)\, |\nabla \ph(x)|^{-2} \, D^2 \ph(x)[\nabla \ph(x), \nabla \ph(x)]\geq0.
\end{align}
\end{lemma}
\begin{proof}
By Lemma \ref{lem3.1} and the remark after \eqref{eqn3.4}, $\ti w_1,\ti w_2$ have continuous second partial derivatives near $y,z$.
Apply Lemma \ref{lem5.2} with $v_i=\ti w_i$, and let $\nu,a,b,d$ be as in \eqref{eqn5.9}. Then
$\nabla\ph(x)=-d\nu\neq0$. For a symmetric $n\times n$ matrix $X$ put
$N(X)=\mbox{trace}\left( \left( I+(p-2)\nu\otimes \nu\right) X\right)$.
Since $I+(p-2)\nu\otimes \nu$ is positive definite, \eqref{eqn5.10} gives
\begin{align}
\label{eqn5.18}
N(D^2\ph(x))\geq \la \left(\frac{d}{a}\right)^3N(D^2v_1(y))+(1-\la)\left(\frac{d}{b}\right)^3N(D^2v_2(z)).
\end{align}
Both gradients are parallel to $\nu$ by \eqref{eqn5.9}, so \eqref{eqn3.4} gives
$N(D^2v_1(y))=N(D^2v_2(z))=0$.
Thus \eqref{eqn5.18} proves \eqref{eqn5.17}, whose left-hand side is
$N(D^2\ph(x))$.
\end{proof}

We end this subsection by showing that \eqref{eqn5.5} implies \eqref{BMp}, \eqref{BMn}, and \eqref{BMpn}. For convex sets this is done in \cite{B1,CS,AGHLV,ALSV}.
\begin{lemma}
\label{lem5.4}
Let $E_1,E_2,\la,E,w_1,w_2,w$ be as in \eqref{eqn5.1}, \eqref{eqn5.2},  and suppose that \eqref{eqn5.5} holds. Then
\begin{align}
\label{eqn5.19}
\begin{split}
\capp(E)&\geq \min\{\capp(E_1),\capp(E_2)\}\, \, \mbox{when}\, \, 1<p<n,\\
\cp(E)&\geq \min\{\cp(E_1),\cp(E_2)\}\, \, \mbox{when}\, \, p\geq n.
\end{split}
\end{align}
\end{lemma}
\begin{proof}
By \eqref{eqn5.5} and \eqref{eqn5.6} $(b)$, $w\geq\min\{w_1,w_2\}$.
For $1<p<n$, multiply by $|x|^{(n-p)/(p-1)}$ and use \eqref{eqn4.2} $(c)$ at infinity.
For $p\geq n$, the same inequality gives $G-F\leq\max\{G_1-F,G_2-F\}$. Take limits using \eqref{eqn4.17} $(b)$ and apply \eqref{eqn2.10}.
\end{proof}

\begin{lemma}
\label{lem5.5}
Let $\mathcal{F}$ be a family of $p$-regular compact sets with positive $\capp$ when $p<n$ and positive $\cp$ when $p\geq n$, such that $\rho E\in \mathcal{F}$ whenever $E\in \mathcal{F}$ and $\rho>0$.
If \eqref{eqn5.19} holds whenever $E_1,E_2\in \mathcal{F}$, $\la\in (0,1)$,  and $E=\la E_1+(1-\la)E_2$, then \eqref{BMp}, \eqref{BMn}, and \eqref{BMpn} hold whenever $E_1,E_2\in \mathcal{F}$ and $\la\in[0,1]$.
\end{lemma}

\begin{proof}
Let $0<\la<1$, and put $\mathbf a_i=\capp(E_i)^{1/(n-p)}$ for $p<n$, $\mathbf a_i=\cp(E_i)$ for $p=n$, and $\mathbf a_i=\cp(E_i)^{1/(p-n)}$ for $p>n$.
Set $A=\la\mathbf a_1+(1-\la)\mathbf a_2$, $E_i''=\mathbf a_i^{-1}E_i$, and $\la''=\la\mathbf a_1/A$. Then $E_i''\in\mathcal F$, $0<\la''<1$, and
$\la''E_1''+(1-\la'')E_2''=A^{-1}\big(\la E_1+(1-\la)E_2\big)$.
By the scaling laws \eqref{eqn4.1}, \eqref{eqn4.17} $(d)$, the two normalized sets have capacity one. Apply \eqref{eqn5.19} to them and scale back.
The endpoints $\la=0,1$ are immediate.
\end{proof}

\subsection{Comparison for p-regular sets}
\label{section7}

We prove \eqref{eqn5.5} under \eqref{eqn5.1} for every $1<p<\infty$.
The next lemma uses the strong maximum principle and Hopf's lemma to
show that a test function with nonzero gradient can touch $w^*<\mathbf m$
only at a pair with equal values and nonzero gradients of $w_1,w_2$.
Lemma \ref{lem5.3} then gives the viscosity inequality in
Definition \ref{defn3.6}.

\begin{lemma}
\label{lem7.1}
Let $E_1,E_2,\la,E,w_1,w_2,\mathbf{m}$, and $w^*$ be as in \eqref{eqn5.1}-\eqref{eqn5.4}. Let $x\in \rn{n}$ with $w^*(x)<\mathbf{m}$ and let $(y,z)$ be an optimal pair for $x$.
Let $\ph$ be a function with continuous second partial derivatives in a neighbourhood of $x$ such that $\ph\geq w^*$ in this neighbourhood and $\ph(x)=w^*(x)$. If either
\begin{align}
\label{eqn7.1}
\begin{split}
&(\al)\hs{.1in} w_1(y)\neq w_2(z), \, \, \mbox{or}\\
&(\be)\hs{.1in} w_1(y)= w_2(z)\, \, \mbox{and one of}\, \, \nabla w_1(y), \nabla w_2(z)\, \, \mbox{is zero},
\end{split}
\end{align}
then $\nabla \ph(x)=0$.
\end{lemma}

\begin{proof}
Put $s=w^*(x)$. In case $(\al)$, suppose $w_1(y)>w_2(z)=s$.
For small $\eta$, the pair $\big(y-(1-\la)\eta/\la,z+\eta\big)$ still represents $x$, and its first value exceeds $s$ by continuity. Optimality therefore gives $w_2(z+\eta)\leq s$.
Since $w_2(z)=s<\mathbf m$, we have $z\notin\hat E_2$.
The strong maximum principle in a small ball gives $w_2\equiv s$ on that ball. Consequently
\[
\ph(x+(1-\la)\eta)\geq w^*(x+(1-\la)\eta)\geq\min\{w_1(y),w_2(z+\eta)\}=s=\ph(x)
\]
for small $\eta$. Hence $\nabla\ph(x)=0$.
If $w_2(z)>w_1(y)=s$, use the pair
$\big(y+\eta,z-\la\eta/(1-\la)\big)$ instead. It gives
$w_1\equiv s$ near $y$ and $\ph(x+\la\eta)\geq\ph(x)$,
so again $\nabla\ph(x)=0$.

In case $(\be)$, $w_1(y)=w_2(z)=s<\mathbf m$, so
$y\notin\hat E_1$ and $z\notin\hat E_2$.
Suppose $\nabla w_1(y)=0$.
For a unit vector $\hat e$ and small $\rho>0$ with $\bar B(y-\rho\hat e,\rho)\subset\rn{n}\sem\hat E_1$, we claim that
\begin{align}
\label{eqn7.2}
\text{there is }\vartheta\neq0\text{ with }y+\vartheta\in B(y-\rho\hat e,\rho)\text{ and }w_1(y+\vartheta)\geq s.
\end{align}
Otherwise $h=s-w_1$ is positive and $p$-harmonic in the ball, continuous on its closure, differentiable at $y$, and zero there.
Lemma \ref{lem3.3} contradicts $\nabla h(y)=0$, proving the claim. The ball condition gives
\begin{align}
\label{eqn7.3}
\lan\vartheta,\hat e\ran<-|\vartheta|^2/(2\rho),\qquad |\vartheta|<2\rho.
\end{align}
Since $w_2(z)=s$, the touching condition and \eqref{eqn5.4} give, for small $\rho$,
\begin{align}
\label{eqn7.4}
\ph(x+\la\vartheta)\geq w^*(x+\la\vartheta)\geq s=\ph(x).
\end{align}
If $g=\nabla\ph(x)\neq0$, choose $\hat e=g/|g|$ and a bound $\hat M\geq1$ for $|D^2\ph[e,e]|$, $|e|=1$, in a neighbourhood of $x$. Taylor's theorem and \eqref{eqn7.3} imply
\[
\ph(x+\la\vartheta)<\ph(x)-\frac{\la|\vartheta|^2}{2}\left(\frac{|g|}{\rho}-\la\hat M\right).
\]
Choose $\rho$ small enough that the ball and touching conditions hold and $\rho<|g|/\hat M$. This contradicts \eqref{eqn7.4}, so $g=0$.
If $\nabla w_2(z)=0$, apply the same argument in
$B(z-\rho\hat e,\rho)$, keeping $y$ fixed. The touching inequality
becomes $\ph(x+(1-\la)\vartheta)\geq\ph(x)$. The same Taylor
estimate, with $\la$ replaced by $1-\la$, gives
$\nabla\ph(x)=0$.
\end{proof}

\begin{lemma}
\label{lem7.2}
Let $E_1,E_2,\la,E,w_1,w_2,\mathbf{m}$, and $w^*$ be as in \eqref{eqn5.1}-\eqref{eqn5.4}. Then $w^*$ is a viscosity $p$-subsolution in $\rn{n}\sem \hat E$ in the sense of Definition \ref{defn3.6}.
\end{lemma}
\begin{proof}
By \eqref{eqn5.6} $(b)$, $w^*$ is continuous. Let $x\in \rn{n}\sem\hat E$ and suppose $\ph$ has continuous second partial derivatives in $\rn{n}\sem\hat E$ with  $\ph\geq w^*$ in $\rn{n}\sem \hat E$,  $\ph(x)=w^*(x)$, and $\nabla \ph(x)\neq 0$.
Let $s=w^*(x)$. Then $s<\mathbf{m}$ by \eqref{eqn5.6} $(c)$. Let $(y,z)$ be an optimal pair for $x$. By Lemma \ref{lem7.1}, neither $(\al)$ nor $(\be)$ in \eqref{eqn7.1} holds. Thus
$w_1(y)=w_2(z)=s$, $\nabla w_1(y)\neq 0$, and $\nabla w_2(z)\neq 0$.
By \eqref{eqn5.3}, $w_1,w_2$ are $p$-harmonic in neighbourhoods of $y,z$, and \eqref{eqn5.4} gives \eqref{eqn5.8} with $v_1=w_1$ and $v_2=w_2$. Lemma \ref{lem5.3} with $\ti w_i=w_i$ gives \eqref{eqn3.8}
with $x_0$ replaced by $x$.
\end{proof}

\begin{proposition}
\label{prop7.3}
Under \eqref{eqn5.1}--\eqref{eqn5.4}, \eqref{eqn5.5} holds. Thus \eqref{BMp}, \eqref{BMn}, and \eqref{BMpn} hold in their respective ranges whenever $E_1,E_2$ are $p$-regular and $\la\in[0,1]$.
\end{proposition}
\begin{proof}
Let $\he=1$ when $1<p<n$ and let $\he\in (0,1)$ when $p\geq n$. Given $\ep>0$ let
$O_\ep=\{x\in \rn{n}:\, w^*(x)-\he w(x)>\ep\}$.
By \eqref{eqn5.3}, \eqref{eqn5.6}, $O_\ep$ is a bounded open set whose closure does not intersect $\hat E$, and
\begin{align}
\label{eqn7.5}
w^*=\he w+\ep\quad \mbox{on}\, \, \ar O_\ep.
\end{align}
Suppose that $O_\ep$ is not empty and let $D$ be a component of $O_\ep$. Then $\bar D$ is a compact subset of the domain $\rn{n}\sem \hat E=\Om_E$,  and $\ar D\subset \ar O_\ep$.
Also $\he w+\ep$ is continuous on $\bar D$ and $p$-harmonic in $D$ by \eqref{eqn5.3}. By \eqref{eqn7.5} and Lemmas \ref{lem7.2} and \ref{lem3.7}, $w^*\leq \he w+\ep$ in $D$, which contradicts $D\subset O_\ep$.
Thus $O_\ep$ is empty. Since $\ep>0$ is arbitrary, $w^*\leq \he w$ on $\rn{n}$. If $p\geq n$ we let $\he\to 1$. This proves \eqref{eqn5.5}.

The $p$-regular sets with positive capacity form a family closed under dilations by Lemma \ref{lem9.6} $(d)$, \eqref{eqn4.1}, and \eqref{eqn4.17} $(d)$.
Lemmas \ref{lem5.4} and \ref{lem5.5} give the stated inequalities for this family, including the endpoints.
By Lemmas \ref{lem4.1} and \ref{lem4.2} $(e)$, the only other $p$-regular sets are single points when $p>n$. If $E_1=\{a\}$ then $\la E_1+(1-\la)E_2=\la a+(1-\la)E_2$, and \eqref{eqn4.17} $(d)$ gives \eqref{BMpn} with equality. The case $E_2=\{a\}$ is symmetric.
\end{proof}
\noindent By Remark \ref{rmk:Lipcorkscrew} and Lemma \ref{lem9.6} $(a)$, Proposition \ref{prop7.3} applies to closures of bounded Lipschitz domains.

\subsection{\texorpdfstring{Closed neighbourhoods and Theorem~A\,(a)--(c)}{Closed neighbourhoods and Theorem A (a)-(c)}}
For the rest of this section, $E, E_1, E_2$ denote nonempty compact sets in $\rn{n}$, $n\geq 2$, and $E_\de$ is the set in \eqref{eqn2.5} for $\de>0$.

\begin{lemma}
\label{lem9.1}
Let $E,E_1,E_2$ be nonempty compact sets in $\rn{n}$ and $\de>0$.
\begin{align*}
&(a)\hs{.1in} E_\de=E+\bar B(0,\de), \, \, \mbox{$E_\de$ is compact, and $E_\de$ satisfies the corkscrew condition}\\
&\hs{.3in}\mbox{in Definition \ref{defn2.4} with}\, \, c_{E_\de}=2\, \, \mbox{and}\, \, r_{E_\de}=\de.\\
&(b)\hs{.1in}  \la (E_1)_\de+(1-\la)(E_2)_\de=(\la E_1+(1-\la)E_2)_\de\, \, \mbox{when}\, \, 0\leq \la\leq 1. \, \, \mbox{Also}\\
&\hs{.3in} (\rho E+x)_\de=\rho\, E_{\de/\rho}+x\, \, \mbox{when}\, \, \rho>0\, \, \mbox{and}\, \, x\in\rn{n}.\\
&(c)\hs{.1in} E\subset E_{\de'}\subset E_\de \, \, \mbox{when}\, \, 0<\de'<\de,\, \,   \bigcap_{\de>0}E_\de=E,\, \, \mbox{and}\, \, \hat E\subset \widehat{E_\de}.\\
& (d) \hs{.1in}\mbox{If $O$ is an open set containing $E$ then $E_\de\subset O$ for all small $\de>0$.}
\end{align*}
\end{lemma}
\begin{proof}
Compactness gives $E_\de=E+\bar B(0,\de)$.
For $w\in E_\de$ and $0<r<\de$, choose $e\in E$ with
$|w-e|\leq\de$ and put $y=w+\frac{r}{2\de}(e-w)$. Then
$|y-w|\leq r/2$ and $|y-e|\leq\de-r/2$, so
$B(y,r/2)\subset E_\de\cap B(w,r)$.
This proves $(a)$. Part $(b)$ follows from $(a)$ and
$\la\bar B(0,\de)+(1-\la)\bar B(0,\de)=\bar B(0,\de)$, and from
$(\rho E+x)_\de=\rho(E+\bar B(0,\de/\rho))+x$.
In $(c)$, the intersection is $E$ because $E$ is closed, and
$\Om_{E_\de}\subset\Om_E$ gives the inclusion of filled sets.
Part $(d)$ holds because $d(E,\rn{n}\sem O)>0$ when $O\ne\rn{n}$.
\end{proof}

\begin{lemma}
\label{lem9.2}
Let $1<p<\infty$, let $O\subset \rn{n}$ be an open set, and let $E\subset O$ be a nonempty compact set. Then
\begin{align}
\label{eqn9.5}
\capp(E_\de,O)\, \, \mbox{decreases to}\, \,  \capp(E,O)\, \, \mbox{as $\de$ decreases to $0$}.
\end{align}
In particular, if $1<p<n$ then $\capp(E_\de)\to \capp(E)$ as $\de\to 0$.
\end{lemma}

\begin{proof}
Choose $\de_i\downarrow0$ with $E_{\de_i}\subset O$, using Lemma~\ref{lem9.1} $(d)$.
These compact sets decrease to $E$, so \cite[Theorem~2.2 $(iv)$]{HKM} gives $\capp(E_{\de_i},O)\to\capp(E,O)$.
Monotonicity in \eqref{eqn9.1} $(a)$ gives the full limit as $\de\downarrow0$.
The last assertion is the case $O=\rn{n}$.
\end{proof}

\begin{proof}[Proof of Theorem \ref{thmA} $(a)$]
Let $1<p<n$, let $E_1$ and $E_2$ be nonempty compact sets in $\rn{n}$, and let $\la\in[0,1]$ and $\de>0$.
By Lemma \ref{lem9.1} $(a)$ and Lemma \ref{lem9.6} $(a)$, $(E_1)_\de$ and $(E_2)_\de$ are $p$-regular. Thus \eqref{BMp} holds for $(E_1)_\de, (E_2)_\de$ and $\la$ by Proposition \ref{prop7.3}.
With Lemma \ref{lem9.1} $(b)$ and \eqref{eqn4.1} this gives
\begin{align*}
\left[\capp((\lambda E_1+(1-\lambda) E_2)_\de)\right]^{\frac{1}{n-p}}\geq\lambda \left[\capp(E_{1})\right]^{\frac{1}{n-p}}+(1-\lambda)\left[\capp(E_2)\right]^{\frac{1}{n-p}}.
\end{align*}
Letting $\de\to 0$ and using Lemma \ref{lem9.2} for $\la E_1+(1-\la)E_2$ we obtain \eqref{BMp}.
\end{proof}

\begin{remark}
\label{cor9.4}
Let $1<p<n$. For an arbitrary set $A\subset\rn{n}$ put $\mbox{Cap}^{*}_{p}(A)=\sup\{\capp(K):\, K\subset A\ \mbox{compact}\}$.
If $K_i\subset A_i$ are compact, then $\la K_1+(1-\la)K_2\subset \la A_1+(1-\la)A_2$ is compact. Theorem \ref{thmA} $(a)$ and independent suprema over $K_1,K_2$ therefore give \eqref{BMp} for $\mbox{Cap}^{*}_{p}$ and arbitrary nonempty sets, with $0\cdot\infty=0$.
Analytic sets, in particular Borel sets, are capacitable \cite[Theorem~2.5]{HKM}, and $\la A_1+(1-\la)A_2$ is analytic when $A_1,A_2$ are, as a continuous image of $A_1\times A_2$. Thus \eqref{BMp} holds for nonempty analytic sets $A_1,A_2$ with $\capp(A)=\capp(A,\rn{n})$.
\end{remark}

\begin{proof}[Proof of Theorem \ref{thmA} $(b), (c)$]
Let $p\geq n$. If $p=n$, the sets $E_1,E_2$ are $n$-regular by hypothesis.
If $p>n$, they are $p$-regular by Lemma~\ref{lem9.6} $(b)$.
Proposition \ref{prop7.3} gives \eqref{BMn} when $p=n$ and \eqref{BMpn} when $p>n$.
\end{proof}

\subsection{Convergence as delta tends to zero}
\begin{lemma}
\label{lem9.13}
Let $n\leq p<\infty$ and let $E$ be a nonempty compact set for which there exists a function $G$ satisfying \eqref{eqn2.9}. Then $G$ is unique, $k(\infty)=\lim_{x\to\infty}(G-F)(x)$ exists finitely, and if $G_\de$ denotes the $p$-harmonic Green's function for
$\rn{n}\sem E_\de$ with pole at $\infty$ and $k_\de=G_\de-F$, then
\begin{align}
\label{eqn9.32}
\begin{split}
&(a)\hs{.1in} G_\de\leq G_{\de'}\leq G\, \, \mbox{on}\, \, \rn{n}\, \, \mbox{when}\, \, 0<\de'<\de,\\
&(b)\hs{.1in} G_\de\to G\, \, \mbox{uniformly on compact subsets of $\rn{n}$ as}\, \, \de\to 0,\\
&(c)\hs{.1in} k_\de(\infty)\to k(\infty)\, \, \mbox{as}\, \, \de\to 0.
\end{split}
\end{align}
In particular, this is the case if $E$ is $p$-regular, and then $\cp(E_\de)\to\cp(E)$ as $\de\to 0$.
\end{lemma}
\begin{proof}
The uniqueness proof in Lemma \ref{lem4.2} uses only \eqref{eqn2.9}. Thus $G$ is unique, $G\geq 0$, and $k(\infty)$ exists by Lemma \ref{lem:radial-limits}. The sets $E_\de$ are $p$-regular by Lemmas \ref{lem9.1} $(a)$ and \ref{lem9.6} $(a)$, so $G_\de$ exists.
The proof of \eqref{eqn4.17} $(c)$ uses only that $G'\geq 0$ and $G\geq 0$ satisfy \eqref{eqn2.9} for $E'\supset E$. Thus \eqref{eqn9.32} $(a)$ holds and $k_\de(\infty)\leq k_{\de'}(\infty)\leq k(\infty)$ when $0<\de'<\de$.
Let $G_{\mathrm{lim}}=\lim_{\de\to 0}G_\de$. Then $0\leq G_{\mathrm{lim}}\leq G$, so $G_{\mathrm{lim}}\equiv 0$ on $E$ and $G_{\mathrm{lim}}$ is continuous at each point of $E$.
Each compact subset of $\rn{n}\sem E$ has a neighbourhood which does not meet $E_\de$ for small $\de>0$.
Hence \eqref{eqn9.32} $(a)$, Lemma \ref{lem3.1} for $G_\de+1$, and Ascoli's theorem give, as in \eqref{eqn4.3}, that $\{G_\de,\nabla G_\de\}$ converges to $\{G_{\mathrm{lim}},\nabla G_{\mathrm{lim}}\}$ uniformly on compact subsets of $\rn{n}\sem E$.
Thus $G_{\mathrm{lim}}$ is continuous on $\rn{n}$ and $p$-harmonic in $\rn{n}\sem E$. Also $G_\de-F\leq G_{\mathrm{lim}}-F\leq G-F$, so $G_{\mathrm{lim}}-F$ is bounded near $\infty$. Uniqueness gives $G_{\mathrm{lim}}=G$. By Dini's theorem the
monotone convergence of the continuous functions $G_{\de_j}$ to $G$ along any sequence $\de_j\downarrow0$ is uniform on compact sets. Monotonicity then
gives \eqref{eqn9.32} $(b)$.

Let $\rho_0>0$ be such that $E_\de\subset B(0,\rho_0)$ for $0<\de\leq 1$ and let $r\geq \rho_0$. By \eqref{eqn9.16} $(ii)$ for $G_\de$, $k_\de(\infty)\geq \inf_{\ar B(0,r)}(G_\de-F)$. Letting $\de\to 0$ and using \eqref{eqn9.32} $(b)$ we get
$k(\infty)\geq \lim_{\de\to 0}k_\de(\infty)\geq \inf_{\ar B(0,r)}(G-F)$.
The right-hand side tends to $k(\infty)$ as $r\to\infty$, which proves \eqref{eqn9.32} $(c)$. The last sentence follows from Lemma \ref{lem4.2} and \eqref{eqn2.10}.
\end{proof}

\subsection{Properties of \texorpdfstring{$\Phi_p$}{Phi\_p}}
\begin{lemma}\label{lem:outer-capacity}
Let $p=n$ and let $E\subset E'$ be nonempty compact sets, $\rho>0$, and $x\in\rn{n}$. Then the limit in \eqref{eqn2.14} exists and is finite, $\bar{\mathcal C}_n(E)\leq\bar{\mathcal C}_n(E')$, and $\bar{\mathcal C}_n(\rho E+x)=\rho\,\bar{\mathcal C}_n(E)$.
\end{lemma}
\begin{proof}
By Lemma \ref{lem9.1} $(a), (c)$, Lemma \ref{lem9.6} $(a)$, and \eqref{eqn4.17} $(c), (e)$, $\de\mapsto \mathcal{C}_n(E_\de)$ is non-decreasing and positive on $(0,\infty)$.
If $E+\bar B(0,1)\subset B(0,\rho_0)$, then for $0<\de\leq1$
monotonicity and \eqref{eqn4.29} give
$0<\mathcal C_n(E_\de)\leq
\mathcal C_n(\bar B(0,\rho_0))=e^{1/\ga}\rho_0$.
Thus the limit in \eqref{eqn2.14} exists and is finite.
If $E\subset E'$ then $E_\de\subset E'_\de$, and monotonicity follows from \eqref{eqn4.17} $(c)$. By Lemma \ref{lem9.1} $(b)$ and \eqref{eqn4.17} $(d)$, $\mathcal{C}_n((\rho E+x)_\de)=\rho\, \mathcal{C}_n(E_{\de/\rho})$. Letting $\de\to 0$ gives the homogeneity and translation invariance.
\end{proof}

We now prove the basic properties of the functional $\Phi_p$ in \eqref{eq:degree-one}, which are used in Section~\ref{section:proofEquality}.
\begin{lemma}\label{lem:filling-invariance}
The functional $\Phi_p$ is translation invariant, monotone, and
homogeneous of degree one. For every nonempty compact set $E$,
\begin{equation}\label{eq:filling-invariance}
\Phi_p(E)=\Phi_p(\hat E).
\end{equation}
For every $\delta>0$,
\begin{equation}\label{eq:filled-neighbourhoods}
\widehat{E_\delta}=\widehat{(\hat E)_\delta}.
\end{equation}
\end{lemma}
\begin{proof}
For $p<n$, the first three assertions follow from \eqref{eqn4.1},
and for $p=n$ they follow from Lemma~\ref{lem:outer-capacity}.
For $p>n$, they follow from \eqref{eqn4.17} $(c)$, $(d)$, which hold for every nonempty compact
set by Lemmas~\ref{lem9.6} $(b)$ and~\ref{lem4.2}.

Lemma~\ref{lem4.3}, applied to $E$ and a ball and then rescaled,
gives
$E_\delta\subset(\hat E)_\delta\subset\widehat{E_\delta}$.
Taking fillings proves \eqref{eq:filled-neighbourhoods}.
By Lemmas~\ref{lem9.1} $(a)$ and \ref{lem9.6} $(a)$, both neighbourhoods are
$p$-regular. Their exterior functions are constant on their
common filling by \eqref{eqn4.2} $(a)$ and \eqref{eqn4.17} $(a)$, so the function for $(\hat E)_\delta$ also solves
the exterior problem for $E_\delta$. Uniqueness in Lemmas~\ref{lem4.1} and~\ref{lem4.2} identifies the two functions,
so their capacities are equal. Letting
$\delta\downarrow0$ proves \eqref{eq:filling-invariance} for
$p<n$ by Lemma~\ref{lem9.2}, and for $p=n$ by \eqref{eqn2.14}.
For $p>n$, every nonempty compact set is $p$-regular. The Green's
function $G_E$ vanishes on $\hat E$ and therefore satisfies
\eqref{eqn2.9} for $\hat E$. Uniqueness gives
$G_E=G_{\hat E}$ and the same constant at infinity.
\end{proof}

\section{Proofs of Theorems B and C}
\label{section:proofB}
For Theorem~\ref{thmB} we express the constant at infinity through
relative capacities and then pass from $E_\delta$ to $E$.

\subsection{The condenser formula}
For $p=n$, the formula below is the condenser characterization in
\cite{X}, with the normalization in
Theorem~\ref{thmB} $(d)$. We give the argument for all $p\geq n$,
since the comparison estimate is also used below.
If $E\subset B(0,r)$ is a nonempty compact set and $p\geq n$ we put
\begin{align}
\label{eqn9.28}
\La_E(r)=\left[\capp(E,B(0,r))\right]^{-\frac{1}{p-1}}-F(r)\in(-\infty,\infty],
\end{align}
where $\La_E(r)=\infty$ when the capacity is zero.
\begin{lemma}
\label{lem9.12}
Let $n\leq p<\infty$ and let $E$ be a nonempty compact set with $E\subset B(0,\rho_0)$.
\begin{itemize}
\item[(a)] $\La_E(\cdot)$ is non-decreasing on $[\rho_0,\infty)$. Also $\La_E(r)\geq \La_{E'}(r)$ when $E'$ is compact and $E\subset E'\subset B(0,r)$.
\item[(b)] If $E$ is $p$-regular and $G$ is the $p$-harmonic Green's function for $\rn{n}\sem E$ with pole at $\infty$, then
\begin{align}
\label{eqn9.29}
\inf_{\ar B(0,r)}G\leq \left[\capp(E,B(0,r))\right]^{-\frac{1}{p-1}}\leq \sup_{\ar B(0,r)}G \quad \mbox{whenever}\, \, r>\rho_0.
\end{align}
Consequently, if $k=G-F$ then
\begin{align}
\label{eqn9.30}
k(\infty)=\lim_{r\to\infty}\La_E(r)=\sup_{r\geq \rho_0}\La_E(r).
\end{align}
\end{itemize}
\end{lemma}
\begin{proof}
To prove $(a)$ let $\rho_0\leq r<r'$ and put
\[
\kappa_1=\capp(E,B(0,r)),\qquad
\kappa_2=\capp(\bar B(0,r),B(0,r')).
\]
By \eqref{eqn9.1} $(e)$ and \eqref{eqn2.7}, $\kappa_2^{-1/(p-1)}=F(r')-F(r)$. We claim that
\begin{align}
\label{eqn9.31}
\left[\capp(E,B(0,r'))\right]^{-\frac{1}{p-1}}\geq \kappa_1^{-\frac{1}{p-1}}+\kappa_2^{-\frac{1}{p-1}}.
\end{align}
If $\kappa_1=0$, monotonicity gives $\capp(E,B(0,r'))=0$, so
\eqref{eqn9.31} holds with infinite inverse powers. Suppose that
$\kappa_1>0$. Given $\ep>0$, choose $v\in C_0^\infty(B(0,r))$
with $v\geq1$ on $E$ and
$\int|\nabla v|^p\,dx\leq\kappa_1+\ep$, and set
$v_1=\min\{\max\{v,0\},1\}$. Let $v_2$ be the radial capacitary
function of the condenser $(\bar B(0,r),B(0,r'))$, extended by $1$
on $\bar B(0,r)$ and by $0$ outside $B(0,r')$. Its energy is
$\kappa_2$. For $0\leq\theta\leq1$, the continuous function
$(1-\theta)v_1+\theta v_2$ belongs to
$W_0^{1,p}(B(0,r'))$ and equals $1$ on $E$.
The gradients of $v_1$ and $v_2$ have disjoint supports up to null sets, so \eqref{eqn9.1} $(d)$ gives
$\capp(E,B(0,r'))\leq
(1-\theta)^p(\kappa_1+\ep)+\theta^p\kappa_2$.
With $a=(\kappa_1+\ep)^{-1/(p-1)}$ and
$b=\kappa_2^{-1/(p-1)}$, the right-hand side is minimized at
$\theta=b/(a+b)$, with minimum $(a+b)^{1-p}$.
Letting $\ep\downarrow0$ proves
\eqref{eqn9.31}, which with \eqref{eqn9.28} gives $\La_E(r')\geq\La_E(r)$.
The second statement in $(a)$ follows from \eqref{eqn9.1} $(a)$.

To prove $(b)$ let $v(\cdot,R)$ be the function in Lemma~\ref{lem9.9}. For large $R$ put
\[
a_R=F(R)/\max_{\ar B(0,R)}G,\qquad
b_R=F(R)/\min_{\ar B(0,R)}G.
\]
Since $G-F$ is bounded near infinity, both tend to $1$.
Comparison in $B(0,R)\sem E$ gives
$a_R G\leq v(\cdot,R)\leq b_R G$,
so $v(\cdot,R)\to G$ uniformly on compact subsets of $\rn{n}$.
Fix $r>\rho_0$ and let $m'(R)$, $M'(R)$, $\kappa_r$ be as in Lemma \ref{lem9.9}. Then $m'(R)\to \inf_{\ar B(0,r)}G$, $M'(R)\to \sup_{\ar B(0,r)}G$, and $F(R)\to\infty$ as $R\to\infty$.
Letting $R\to\infty$ in \eqref{eqn9.18} $(a), (b)$ we obtain
\[
\left(\inf_{\ar B(0,r)}G\right)^{p-1}\kappa_r\leq 1\leq \left(\sup_{\ar B(0,r)}G\right)^{p-1}\kappa_r,
\]
which is \eqref{eqn9.29}. Since $F$ is constant on $\ar B(0,r)$, \eqref{eqn9.29} reads $m(r)\leq \La_E(r)\leq M(r)$, with $m(r),M(r)$ as in Lemma \ref{lem:radial-limits}.
By that lemma $m(r)$ and $M(r)$ tend to $k(\infty)$ as $r\to\infty$. With $(a)$ this gives \eqref{eqn9.30}.
\end{proof}

\subsection{Proof of Theorem B}
\label{subsection:outer-capacity}
\begin{proof}[Proof of Theorem \ref{thmB}]
Part $(a)$ is Lemma~\ref{lem:outer-capacity}.
The first statement in $(b)$ follows from Lemma \ref{lem9.13} and \eqref{eqn2.10}. If $E$ is $n$-regular, then $G$ exists by Lemma~\ref{lem4.2}. To prove $(c)$ let $0<\la<1$ and $\de>0$. By Lemmas \ref{lem9.1} $(a)$, \ref{lem9.6} $(a)$ and Proposition \ref{prop7.3}, \eqref{BMn} holds for $(E_1)_\de,(E_2)_\de$. By Lemma \ref{lem9.1} $(b)$,
\[
\mathcal{C}_n((\lambda E_1+(1-\lambda)E_2)_\de)\geq \lambda\,  \mathcal{C}_n((E_1)_\de) +(1-\lambda)\, \mathcal{C}_n((E_2)_\de).
\]
Letting $\de\to 0$ gives \eqref{eqn2.16}, which is trivial when $\la\in\{0,1\}$.

To prove $(d)$ first assume that $z_0=0\in E$. Choose $\rho_0>0$ with $E_\de\subset B(0,\rho_0)$ for $0<\de\leq 1$, and let $G_\de$ be the $n$-harmonic Green's function for $\rn{n}\sem E_\de$ with pole at $\infty$ and $k_\de=G_\de-F$. The sets $E_\de$ are $n$-regular by Lemma \ref{lem9.6} $(a)$.
By \eqref{eqn9.30} for $E_\de$ and Lemma \ref{lem9.12} $(a)$,
\[
\lim_{\de\to 0}k_\de(\infty)=\sup_{0<\de\leq 1}\, \sup_{r\geq \rho_0}\La_{E_\de}(r)=\sup_{r\geq \rho_0}\, \sup_{0<\de\leq 1}\La_{E_\de}(r)=\sup_{r\geq \rho_0}\La_E(r)=\lim_{r\to\infty}\La_E(r).
\]
Here $k_\de(\infty)$ and $\La_{E_\de}(r)$ do not decrease as $\de$ decreases, by \eqref{eqn4.17} $(c)$ and Lemma \ref{lem9.12} $(a)$, and $\La_{E_\de}(r)\to \La_E(r)$ by Lemma \ref{lem9.2}. By \eqref{eqn2.14}, \eqref{eqn2.10}, \eqref{eqn9.28}, and $F(r)=\ga\log r+1$, this is \eqref{eqn2.17} when $z_0=0\in E$.
If $z_0\in E$ we apply this to $E-z_0$ and use $(a)$ and \eqref{eqn9.1} $(c)$. For arbitrary $z_0\in \rn{n}$, choose $z\in E$ and put $\mathbf{d}=|z-z_0|$. If $r>\mathbf{d}$ then $B(z,r-\mathbf{d})\subset B(z_0,r)\subset B(z,r+\mathbf{d})$, so by \eqref{eqn9.1} $(b)$,
\[
\mbox{Cap}_n(E,B(z,r+\mathbf{d}))\leq \mbox{Cap}_n(E,B(z_0,r))\leq \mbox{Cap}_n(E,B(z,r-\mathbf{d}))
\]
for large $r$. Since $(r\pm \mathbf{d})/r\to 1$, \eqref{eqn2.17} holds for $z_0$.

For the last sentence in $(d)$, \cite[Section~2.7]{HKM} and \cite[Lemma~2.9]{HKM} show that if
 $\mbox{Cap}_n(E,O)=0$ for one bounded open set $O$ containing $E$, then this is true for every such $O$. Thus if $\mbox{Cap}_n(E,B(z_0,r))=0$ for some $r$ with $E\subset B(z_0,r)$ then
$\bar{\mathcal{C}}_n(E)=0$ by \eqref{eqn2.17}. Otherwise we may assume by $(a)$ that $0\in E$. Let $\rho_0$, $G_\de$, $k_\de$ be as above, so that $\kappa:=\mbox{Cap}_n(E,B(0,2\rho_0))>0$, and let $0<\de\leq 1$.
Then $G_\de$ is a positive $n$-harmonic function in $\rn{n}\sem \bar B(0,\rho_0)$. Harnack's inequality \eqref{eqn3.1} in balls $B(y,\rho_0/4)$ with $y\in \ar B(0,2\rho_0)$,  \eqref{eqn9.29} for $E_\de$, and \eqref{eqn9.1} $(a)$ give
\[
\sup_{\ar B(0,2\rho_0)}G_\de\leq c\inf_{\ar B(0,2\rho_0)}G_\de\leq c\left[\mbox{Cap}_n(E_\de,B(0,2\rho_0))\right]^{-\frac{1}{n-1}}\leq c\, \kappa^{-\frac{1}{n-1}}
\]
with $c$ depending only on $n$. By \eqref{eqn9.16} $(ii)$, $k_\de(\infty)\leq c\, \kappa^{-1/(n-1)}-F(2\rho_0)$ for $0<\de\leq 1$, and so $\bar{\mathcal{C}}_n(E)>0$ by \eqref{eqn2.10}, \eqref{eqn2.14}.

For $(e)$ we use \cite{AG}. If $K\subset \rn{2}$ is a compact set which is not polar, there is a unique non-negative harmonic function $h_K$ on $\rn{2}\sem K$ which has limit $0$ quasi-everywhere on $\ar K$, is bounded near $\ar K$,
and such that $h_K(x)-\log|x|$ has a finite limit $r(K)$ at $\infty$. The logarithmic capacity of $K$ is $e^{-r(K)}$.
Let $n=2$, so $\ga=1/(2\pi)$, and let $\de>0$. The set $E_\de$ contains a ball, so it is not polar.
The function $2\pi G_\de$ is non-negative and harmonic on $\rn{2}\sem E_\de$, continuous with value $0$ on $\ar E_\de$, and $2\pi G_\de(x)-\log|x|\to 2\pi(1+k_\de(\infty))$ as $x\to\infty$.
Hence $h_{E_\de}=2\pi G_\de$, and by \eqref{eqn2.10},
$\mbox{cap}(E_\de)=e^{-2\pi(1+k_\de(\infty))}=e^{-2\pi}\, \mathcal{C}_2(E_\de)$.
By \cite{AG}, compact sets are log-capacitable and logarithmic capacity is monotone. Since each open set containing $E$ contains $E_\de$ for small $\de$, $\mbox{cap}(E_\de)\to \mbox{cap}(E)$. Thus $\bar{\mathcal{C}}_2(E)=e^{2\pi}\, \mbox{cap}(E)$, and \eqref{eqn2.16} gives \eqref{eqn2.18}.
\end{proof}
\noindent From \eqref{eqn2.18} with $\la=1/2$ and the homogeneity of logarithmic capacity, which follows from Theorem \ref{thmB} $(a), (e)$, we see that \eqref{eqn1.5} holds for all nonempty compact sets $E_1,E_2$ in the plane, connected or not.

\subsection{Pointwise comparisons}
\label{section:proofC}
For arbitrary compact $E$, the outer limits $\bar u_E$, $\bar G_E$ and the regularization $\underline u_E$ are defined in and after \eqref{eqn2.19}.

\begin{mytheorem}
\label{thmC}
Let $1<p<\infty$, let $E_1,E_2$ be nonempty compact sets in $\rn{n}$, $n\geq 2$, let $0<\la<1$, and $E=\la E_1+(1-\la)E_2$.

\begin{itemize}
\item[(a)] Suppose that $E_1,E_2$ are $p$-regular.
If $1<p<n$ let $u_1,u_2,u$ be the $p$-capacitary functions for $E_1,E_2,E$, and if $p\geq n$ let $G_1,G_2,G$ be the $p$-harmonic Green's functions for $\rn{n}\sem E_1$, $\rn{n}\sem E_2$, $\rn{n}\sem E$ with pole at $\infty$.
Then for all $y,z\in\rn{n}$,
\begin{align}
\label{eqn2.20}
\begin{split}
u(\la y+(1-\la)z)&\geq \min\{u_1(y),u_2(z)\}\quad \mbox{when}\, \, 1<p<n,\\
G(\la y+(1-\la)z)&\leq \max\{G_1(y),G_2(z)\}\quad \mbox{when}\, \, p\geq n.
\end{split}
\end{align}
\item[(b)] If $1<p<n$ and $\bar u_{E_1},\bar u_{E_2},\bar u_E$ are as in \eqref{eqn2.19}, then
\begin{align}
\label{eqn2.21}
\bar u_E(\la y+(1-\la)z)\geq \min\{\bar u_{E_1}(y),\bar u_{E_2}(z)\}\quad \mbox{whenever}\, \, y,z\in\rn{n}.
\end{align}
The same inequality holds for the lower semicontinuous regularizations:
\begin{align}
\label{eqn2.22}
\underline{u}_E(\la y+(1-\la)z)\geq \min\{\underline{u}_{E_1}(y),\underline{u}_{E_2}(z)\}\quad \mbox{whenever}\, \, y,z\in\rn{n}.
\end{align}
\item[(c)] If $p=n$ and $\bar G_{E_1},\bar G_{E_2},\bar G_E$ are as in \eqref{eqn2.19}, then
\begin{align}
\label{eqn2.23}
\bar G_E(\la y+(1-\la)z)\leq \max\{\bar G_{E_1}(y),\bar G_{E_2}(z)\}\quad \mbox{whenever}\, \, y,z\in\rn{n}.
\end{align}
\end{itemize}
\end{mytheorem}

By Lemma~\ref{lem9.6}, part $(a)$ includes all nonempty compact sets
when $p>n$. For $p\leq n$, it suffices that each $E_i$ satisfies the
corkscrew condition or, when $p>n-1$, has
no singleton component. Part $(c)$ is understood in the extended real sense. The pointwise comparisons
for convex bodies are \cite[(5.13)]{AGHLV} and \cite[(6.11)]{ALSV}.

\begin{proof}[Proof of Theorem \ref{thmC} $(a)$]
By \eqref{eqn5.2} and \eqref{eqn5.4}, \eqref{eqn2.20} is \eqref{eqn5.5}, which holds by Proposition~\ref{prop7.3}.
\end{proof}

If $1<p<n$ and $E$ is not $p$-regular, a continuous function satisfying \eqref{eqn2.6} need not exist. We replace it by a limit of capacitary functions. The next lemma identifies the limit as a reduced function and determines
its coefficient at infinity.

\begin{lemma}
\label{lem9.15}
Let $1<p<n$, let $E$ be a nonempty compact set in $\rn{n}$, and for $\de>0$ let $u_\de$ be the $p$-capacitary function for $E_\de$. Then $u_{\de'}\leq u_\de$ on $\rn{n}$ when $0<\de'<\de$, and the function $\bar u_E$ in \eqref{eqn2.19} has the following properties.
\begin{align*}
&(a)\hs{.1in} \bar u_E\, \, \mbox{is upper semicontinuous on $\rn{n}$,}\, \, 0\leq \bar u_E\leq 1, \, \, \mbox{and}\, \, \bar u_E\equiv 1\, \, \mbox{on}\, \, \hat E.\\
&(b)\hs{.1in} \bar u_E\, \, \mbox{is $p$-harmonic in $\rn{n}\sem E$, and}\, \, \{u_\de,\nabla u_\de\}\, \, \mbox{converges to}\, \, \{\bar u_E,\nabla \bar u_E\}\\
&\hs{.3in} \mbox{uniformly on compact subsets of $\rn{n}\sem E$ as $\de\to 0$.}\\
&(c)\hs{.1in} \lim_{x\to\infty} |x|^{\frac{n-p}{p-1}}\, \bar u_E(x)=\frac{p-1}{n-p}\, \si_{n-1}^{-\frac{1}{p-1}}\,  \left[\capp(E)\right]^{\frac{1}{p-1}}.\\
&(d)\hs{.1in} \mbox{If $E$ is $p$-regular then $\bar u_E$ is the $p$-capacitary function for $E$.}\\
&(e)\hs{.1in} \bar u_E(x)=\inf\{\ti v(x):\, \ti v\, \, \mbox{is $p$-superharmonic in $\rn{n}$,}\, \, \ti v\geq 0, \, \, \ti v\geq 1\, \, \mbox{on}\, \, E\}\, \, \mbox{for}\, \, x\in\rn{n}.
\end{align*}
\end{lemma}
\noindent Here $p$-superharmonicity is as in
Definition~\ref{def:p-superharmonic}. Part $(e)$ identifies $\bar u_E$ with the reduced function
$R_E^1(\rn{n})$ in \cite[pp.~157--158]{HKM}.
\begin{proof}
Let $0<\de'<\de$. Then $u_{\de'}\leq 1=u_\de$ on $E_\de$. Since $u_\de,u_{\de'}$ are $p$-harmonic in $\rn{n}\sem E_\de$ and tend to $0$ at $\infty$ by \eqref{eqn4.2} $(a)$,
comparison in $B(0,s)\sem E_\de$ and a limit as $s\to\infty$ give $u_{\de'}\leq u_\de$ on $\rn{n}$. Thus the limit in \eqref{eqn2.19} exists, $0\leq \bar u_E\leq 1$, and $\bar u_E$ is upper semicontinuous as a decreasing limit of continuous functions.
Also $u_\de\equiv 1$ on $\widehat{E_\de}\supset\hat E$ by \eqref{eqn4.2} $(a)$ and Lemma \ref{lem9.1} $(c)$. Hence $(a)$ is true. Part $(b)$ follows as in the proof of \eqref{eqn9.32} $(b)$, using Lemma \ref{lem3.1} for $u_\de+1$ and Ascoli's theorem.

To prove $(c)$ choose $\rho_0>0$ with $E_\de\subset B(0,\rho_0)$ for $0<\de\leq 1$, and let $r>\rho_0$, $0<\de\leq 1$. Let $\mathbf{g}$ be as in Lemma \ref{lem:radial-limits}, let $\ell_\de$ be the right-hand side of \eqref{eqn4.2} $(c)$ for $E_\de$, and let $\ell_0$ be the right-hand side of $(c)$.
Since $u_\de$ is positive and $p$-harmonic in $\rn{n}\sem \bar B(0,\rho_0)$, and $u_\de/\mathbf{g}\to \ell_\de$ at $\infty$ by \eqref{eqn4.2}, Lemma \ref{lem:radial-limits} gives
$\inf_{\ar B(0,r)}u_\de/\mathbf{g}\leq \ell_\de\leq \sup_{\ar B(0,r)}u_\de/\mathbf{g}$.
Letting $\de\to 0$ and using $(b)$ and Lemma \ref{lem9.2}, we get
\begin{align}
\label{eqn9.34}
\inf_{\ar B(0,r)}\bar u_E/\mathbf{g}\leq \ell_0\leq \sup_{\ar B(0,r)}\bar u_E/\mathbf{g}\quad \mbox{whenever}\, \, r>\rho_0.
\end{align}
If $\bar u_E$ vanishes at some point of $\rn{n}\sem \bar B(0,\rho_0)$, then $\bar u_E\equiv 0$ there by the strong maximum principle, and $\ell_0=0$ by \eqref{eqn9.34}, so $(c)$ is true.
Otherwise $\bar u_E$ is positive and $p$-harmonic in $\rn{n}\sem \bar B(0,\rho_0)$, and $\bar u_E(x)\to 0$ as $x\to\infty$ since $\bar u_E\leq u_\de$. By Lemma \ref{lem:radial-limits}, $\bar u_E/\mathbf{g}$ has a limit $\ell$ at $\infty$,
and the two sides of \eqref{eqn9.34} tend to $\ell$ as $r\to\infty$. Hence $\ell=\ell_0$, which is $(c)$.

To prove $(d)$ let $E$ be $p$-regular and let $u$ be its $p$-capacitary function (Lemma \ref{lem4.1}). As above, $u\leq u_\de$ for $\de>0$, so $u\leq \bar u_E$. Given $\ep\in(0,1)$ the open set $\{u>1-\ep\}$ contains $E$,
so by Lemma \ref{lem9.1} $(d)$ it contains $E_\de$ for small $\de>0$. For such $\de$, $u_\de=1\leq u/(1-\ep)$ on $E_\de$, and then $u_\de\leq u/(1-\ep)$ on $\rn{n}$ by comparison. Hence $\bar u_E\leq u/(1-\ep)$. Letting $\ep\to 0$ we get $\bar u_E=u$.

To prove $(e)$ let $\mathcal V$ denote the family of functions $\ti v$ in $(e)$.
The comparison argument in the proof of Lemma~\ref{lem4.1} makes
$u_\de$ $p$-superharmonic on $\rn{n}$, so $u_\de\in\mathcal V$ and
$\inf_{\mathcal V}\ti v\leq\bar u_E$.
Now let $\ti v\in \mathcal V$ and $\ep\in(0,1)$. Since $\ti v$ is lower semicontinuous, $\{\ti v>1-\ep\}$ is open and contains $E$. Choose $\de>0$ so small that $\ti v>1-\ep$ on $E_\de$.
For $s>0$ with $E_\de\subset B(0,s)$ let $\ep_s=\sup_{\ar B(0,s)}u_\de$ and $D=B(0,s)\sem E_\de$. Then $h=u_\de-\ep_s$ is continuous on $\bar D$, $p$-harmonic in $D$, $h\leq 1\leq \ti v/(1-\ep)$ on $\ar E_\de$, and $h\leq 0\leq \ti v/(1-\ep)$ on $\ar B(0,s)$.
Since $\ti v/(1-\ep)$ is $p$-superharmonic, $u_\de-\ep_s\leq \ti v/(1-\ep)$ in $D$. Letting $s\to\infty$ gives $u_\de\leq \ti v/(1-\ep)$ in $\rn{n}\sem E_\de$, and this also holds on $E_\de$.
Thus $\bar u_E\leq \ti v/(1-\ep)$. Letting $\ep\to 0$ and taking the infimum over $\ti v\in \mathcal V$ proves $(e)$.
\end{proof}

For arbitrary compact $E$, $p=n$, and $0<\de'<\de$, both closed neighbourhoods $E_\de$ and $E_{\de'}$
are $n$-regular. Exterior comparison as in
the proof of \eqref{eqn4.17} $(c)$ gives $G_\de\leq G_{\de'}$,
without assuming that a Green's function for $E$ exists.
Thus $\bar G_E$ is defined in $[0,\infty]$.
It vanishes on $\hat E$ by Lemma~\ref{lem9.1} $(c)$ and \eqref{eqn4.17} $(a)$, and equals $G$ when the continuous Green's function exists, by \eqref{eqn9.32} $(b)$.

\begin{proof}[Proof of Theorem \ref{thmC} $(b), (c)$]
Let $\de>0$. The sets $(E_1)_\de,(E_2)_\de$ are $p$-regular and $\la (E_1)_\de+(1-\la)(E_2)_\de=E_\de$ by Lemmas \ref{lem9.1} and \ref{lem9.6} $(a)$. Thus \eqref{eqn5.5} holds for these sets by Proposition \ref{prop7.3}.
If $1<p<n$ and $u_{1,\de},u_{2,\de},u_\de$ denote the $p$-capacitary functions for $(E_1)_\de,(E_2)_\de,E_\de$, then \eqref{eqn5.5} and \eqref{eqn5.4} give
\[
u_\de(\la y+(1-\la)z)\geq \min\{u_{1,\de}(y),u_{2,\de}(z)\}\quad \mbox{whenever}\, \, y,z\in\rn{n}.
\]
Letting $\de\to 0$ we get \eqref{eqn2.21}. To prove \eqref{eqn2.22} let $y,z\in\rn{n}$, $x=\la y+(1-\la)z$,  and $r>0$. If $\vartheta\in B(0,r)$ then $x+\vartheta=\la(y+\vartheta)+(1-\la)(z+\vartheta)$, so by \eqref{eqn2.21},
\[
\bar u_E(x+\vartheta)\geq \min\{\bar u_{E_1}(y+\vartheta),\bar u_{E_2}(z+\vartheta)\}\geq \min\{\inf_{B(y,r)}\bar u_{E_1},\inf_{B(z,r)}\bar u_{E_2}\}.
\]
Taking the infimum over $\vartheta\in B(0,r)$ and letting $r\to 0$ we obtain \eqref{eqn2.22}. If $p=n$ then \eqref{eqn5.5} for the Green's functions of $(E_1)_\de,(E_2)_\de,E_\de$, together with \eqref{eqn5.2} and \eqref{eqn5.4}, gives
\eqref{eqn2.23} with $\bar G_{E_1},\bar G_{E_2},\bar G_E$ replaced by these functions. Letting $\de\to 0$ we get \eqref{eqn2.23}.
\end{proof}
\noindent The inequality obtained by putting $\underline u_E$ on the left of \eqref{eqn2.21} and keeping $\bar u_{E_1},\bar u_{E_2}$ on the right can fail. For example, let $1<p<n$, $E_1=E_2=\bar B(0,1)\cup\{x_0\}$ with $|x_0|=3$ and $\la=1/2$. Then $x_0$ is an isolated point of $E$ and $\rn{n}\sem E$ is connected. Since $\{x_0\}$ has zero $p$-capacity, $\bar u_E$ extends $p$-harmonically across $x_0$ \cite[Theorem~7.36]{HKM}, and the strong maximum principle and Lemma~\ref{lem9.15} $(c)$ give $\underline u_E(x_0)<1=\bar u_{E_i}(x_0)$.

\section{Proof of equality in Theorem A}
\label{section:proofEquality}

We use Theorem~\ref{thmB}, the condenser estimate \eqref{eqn9.29},
Lemma~\ref{lem9.15} and Theorem~\ref{thmC}.

\subsection{Exterior functions and level sets}\label{subsec:levels}

Write $\Delta_p v=\nabla\cdot(|\nabla v|^{p-2}\nabla v)$.
Equations with measure data are understood weakly:
\[
-\Delta_p v=\mu\quad\Longleftrightarrow\quad
\int |\nabla v|^{p-2}\langle\nabla v,\nabla\varphi\rangle\,dy
=\int\varphi\,d\mu\quad(\varphi\in C_0^\infty).
\]
A set is $p$-polar if its relative $p$-capacity is zero locally,
and quasi-everywhere means outside a $p$-polar set. We use
quasi-continuous representatives $\tilde v$ for Sobolev functions
and lower semicontinuous representatives for $p$-superharmonic
functions, whose value at $x$ is
$\operatorname*{ess\,liminf}_{y\to x}v(y)$
\cite[Chapters~4 and~7, Theorem~7.22]{HKM}.
For a bounded open set $O$ and a compact set $K\subset O$,
\eqref{eqn2.15} has the equivalent Sobolev form
\cite[Chapters~2 and~4]{HKM}, \cite{KM}
\begin{equation}\label{eq:sobolev-capacity}
\operatorname{Cap}_p(K,O)
=\inf\Bigl\{\int_O|\nabla v|^p\,dx:
 v\in W^{1,p}_0(O),\ \tilde v\geq1
 \text{ quasi-everywhere on }K\Bigr\}.
\end{equation}
For $1<p<n$, the infimum in \eqref{eqn1.4} is unchanged if the
admissible class is replaced by the functions
$v\in L^{p^*}(\rn{n})$ with $\nabla v\in L^p(\rn{n})$ and
$\tilde v\geq1$ quasi-everywhere on $E$. Indeed, with the cutoffs
$\chi_R$ chosen before \eqref{eqn4.6}, $\chi_Rv$ is admissible in
\eqref{eq:sobolev-capacity} with $O=B(0,2R)$ and
$\nabla(\chi_Rv)\to\nabla v$ in $L^p$ by \eqref{eqn4.6}.
Smooth competitors give the reverse inequality.
For a compact set $E$ with $\Phi_p(E)>0$, put
\[
(W_E,m)=
\begin{cases}
(\bar u_E,1),&p<n,\\
(-\bar G_E,0),&p=n,\\
(-G_E,0),&p>n,
\end{cases}
\]
where the outer limits are those in \eqref{eqn2.19}.
Write $w_E$ for the lower semicontinuous regularization of $W_E$.

\begin{lemma}\label{lem:equality-functions}
The function $w_E$ is locally bounded and $p$-superharmonic on
$\rn{n}$, belongs to $W^{1,p}_{\mathrm{loc}}(\rn{n})$, and is
$p$-harmonic off $E$. Moreover, $W_E\geq w_E$, the two functions
agree quasi-everywhere, $W_E=m$ on $\hat E$, and $w_E=m$
quasi-everywhere on $E$. The positive Riesz measure
$\mu_E=-\Delta_p w_E$ is supported on $E$.
\begin{enumerate}
\item[(a)] If $1<p<n$, then $w_E\in L^{p^*}(\rn{n})$, $\nabla w_E\in L^p(\rn{n})$,
$\int|\nabla w_E|^p\,dx=\capp(E)$, and $w_E$ is the unique function
with these integrability properties that minimizes the $p$-energy among
functions at least one quasi-everywhere on $E$.
\item[(b)] If $p=n$, $E_\delta\subset B(0,\rho_0)$ for $0<\delta\leq1$, and
$L\geq2\rho_0$, then $0\leq G_\delta\leq \Lambda_L$ on $B(0,L)$ for
$0<\delta\leq1$, where $G_\delta$ is the Green's function of
$E_\delta$ and $\Lambda_L$ is independent of $\delta$.
\item[(c)] For $p\geq n$,
\begin{equation}\label{eq:equality-green-asymptotic}
-w_E(x)=F(x)+k_E+o(1)\quad\hbox{as }|x|\to\infty,
\end{equation}
where $k_E=-\gamma\log\overline{\mathcal C}_n(E)$ if $p=n$,
and $k_E=k(\infty)$ in \eqref{eqn2.10} if $p>n$.
\end{enumerate}
\end{lemma}
\begin{proof}
For $p<n$, Lemma~\ref{lem9.15} identifies $W_E$ with the reduced
function. The fundamental convergence theorem
\cite[Theorem~8.2]{HKM}, applied to $u_{\delta_j}$ with
$\delta_j\downarrow0$, gives a $p$-superharmonic regularization
$w_E=\underline u_E$ with $w_E=W_E$ quasi-everywhere, so $w_E=1$
quasi-everywhere on $E$. By Lemma~\ref{lem9.2} the energies of
$u_{\delta_j}$ decrease to $\capp(E)$, so a subsequence converges
weakly, with gradients in $L^p$ and functions in $L^{p^*}$, and the
limit equals $w_E$ almost everywhere since $0\leq u_{\delta_j}\leq1$ decrease to $W_E$.
Weak lower semicontinuity and the admissibility of $w_E$ give
$\int|\nabla w_E|^p=\capp(E)$, so $w_E$ is the variational
capacitary function \cite[Theorem~9.35]{HKM}. Strict convexity of
the energy makes the minimizing gradient unique, and $L^{p^*}$
integrability fixes the constant. Being locally bounded, $w_E$ is a
weak supersolution \cite[Corollary~7.20]{HKM}, and it is
$p$-harmonic off $E$ by Lemma~\ref{lem9.15} $(b)$.
For $p>n$, Lemmas~\ref{lem9.6} $(b)$ and~\ref{lem4.2} give the
continuous Green's function $G_E$. Its negative, extended by zero on
$E$, is $p$-superharmonic on $\rn{n}$ by the comparison argument in
Lemma~\ref{lem4.1}, with maximum zero, and \cite[Corollary~7.20]{HKM}
gives $-G_E\in W^{1,p}_{\mathrm{loc}}(\rn{n})$. The remaining
assertions follow from Lemma~\ref{lem4.2} and \eqref{eqn9.32}.

Let $p=n$ and choose $\rho_0$ with $E_\delta\subset B(0,\rho_0)$
for $0<\delta\leq1$. Theorem~\ref{thmB} $(d)$ gives
$\kappa=\capp(E,B(0,2\rho_0))>0$. A finite Harnack chain outside
$B(0,\rho_0)$ \cite[Theorem~6.2]{HKM} and \eqref{eqn9.29} give
\[
\sup_{\partial B(0,2\rho_0)}G_\delta
\leq c\inf_{\partial B(0,2\rho_0)}G_\delta
\leq c\,\operatorname{Cap}_n(E_\delta,B(0,2\rho_0))^{-1/(n-1)}
\leq c\,\kappa^{-1/(n-1)}.
\]
For fixed $L\geq2\rho_0$, further Harnack chains in
$\{3\rho_0/2<|x|<2L\}$ bound $G_\delta$ on $\partial B(0,L)$
uniformly in $\delta$, and the maximum principle gives
$0\leq G_\delta\leq \Lambda_L$ on $B(0,L)$.

As for $p>n$, $-G_\delta$ is $n$-superharmonic on $\rn{n}$ and a
weak supersolution. Caccioppoli's inequality
\cite[Lemma~3.27]{HKM}, applied to $-G_\delta$, gives, for
$0\leq\zeta\in C_c^\infty(B(0,2L))$,
\[
\int\zeta^n|\nabla G_\delta|^n\,dx
\leq c\int G_\delta^n|\nabla\zeta|^n\,dx
\leq c\Lambda_{2L}^n\int|\nabla\zeta|^n\,dx.
\]
For $\delta_j\downarrow0$, the functions $-G_{\delta_j}$ decrease to
$W_E$ and are locally uniformly bounded below. By
\cite[Theorem~8.2]{HKM} their infimum has an $n$-superharmonic
regularization $w_E$ with $w_E=W_E$ quasi-everywhere, so $w_E=0$
quasi-everywhere on $E$. The bounded pointwise convergence is strong
in $L^n$ on balls, so the local energy bounds give
$w_E\in W^{1,n}_{\mathrm{loc}}$. Local uniform convergence off $E$
gives $n$-harmonicity there.

Put $r=2\rho_0$. By \eqref{eqn9.16} $(i)$ and the sphere bound,
$-F(r)\leq G_\delta(x)-F(x)\leq \Lambda_r-F(r)$ for $|x|\geq r$,
uniformly in $\delta$. Thus $\bar G_E-F$ is bounded near infinity,
and Lemma~\ref{lem:radial-limits} gives a finite limit $k_E$.
For each large fixed sphere, \eqref{eqn9.16} $(ii)$ gives
$\inf_{\partial B(0,R)}(G_\delta-F)\leq k_\delta(\infty)\leq k_E$.
Letting first $\delta\downarrow0$, using local uniform convergence
off $E$, and then $R\to\infty$, we get $k_\delta(\infty)\to k_E$,
which identifies $k_E$ by \eqref{eqn2.10} and \eqref{eqn2.14}.
The assertions about the Riesz measures in all ranges follow from
\cite[Theorem~21.2]{HKM}.
\end{proof}

We use the translation invariance, homogeneity of degree one,
monotonicity, and filling invariance of $\Phi_p$ established in
Lemma~\ref{lem:filling-invariance}.
For $0<t<1$ if $p<n$, and for $t<0$ if $p\geq n$, put
$K_E(t)=\{x:W_E(x)\geq t\}$. These compact level sets contain
$E$ and decrease with $t$.

\begin{lemma}\label{lem:regular-point-limit}
Let $\Phi_p(E)>0$ and let $x_0\in\partial\hat E$ be regular
in the sense of Definition~\ref{defn9.5}. Then
$\lim_{x\to x_0}W_E(x)=m$. Consequently, for every $T$ in the level range defined above,
a neighbourhood of $x_0$ lies in $K_E(T)$.
\end{lemma}
\begin{proof}
For $p>n$, this follows from continuity of $G_E$.
Let $1<p\leq n$, choose $\rho_0$ with
$\widehat{E_\delta}\subset B(0,\rho_0)$ for $0<\delta\leq1$, put
$R=2\rho_0$, and put
$D=B(0,R)\setminus\hat E$. Let $h$ be the variational Dirichlet
solution in $D$ with $h-\psi\in W^{1,p}_0(D)$, where $\psi$ is
smooth, zero near $\hat E$ and one near $\partial B(0,R)$. Then
$0\leq h\leq1$, $h\to1$ at the outer sphere, and $h(x)\to0$ as
$D\ni x\to x_0$ by the Wiener criterion \cite[Theorem~6.27]{HKM}.
For $p<n$ put $v_\delta=1-u_\delta$ and $\Lambda=1$, and for $p=n$
put $v_\delta=G_\delta$ with the bound $v_\delta\leq\Lambda$ on
$B(0,R)$ from Lemma~\ref{lem:equality-functions}. In
$D_\delta=B(0,R)\setminus\widehat{E_\delta}$ we have
$v_\delta=0\leq\Lambda h$ on the inner boundary and
$v_\delta\leq\Lambda=\Lambda h$ on the outer sphere. By
\eqref{eq:filled-neighbourhoods}, $(\hat E)_\delta\subset\widehat{E_\delta}$,
so $h$ is continuous on the inner boundary and the comparison gives
$v_\delta\leq\Lambda h$ in $D_\delta$, and trivially on
$D\cap\widehat{E_\delta}$. The outer limit gives
$0\leq m-W_E\leq\Lambda h$ in $D$, which with $W_E=m$ on $\hat E$
proves the claim.
\end{proof}

The following is a special case of
\cite[Lemma~7.37]{HKM}.
Let $D$ be bounded, let $u$ be bounded and $p$-superharmonic
in $D$, and let $v$ be bounded and $p$-subharmonic there. If
$\limsup_{y\to x}v(y)\leq\liminf_{y\to x}u(y)$ for
$x\in\partial D\setminus P$, where $P$ is $p$-polar and at least one function belongs to
$W^{1,p}(D)$, then $v\leq u$ in $D$. Indeed, a $p$-polar set has
zero capacity relative to $D$ in the sense of \cite[p.~144]{HKM}.
In the applications below, both functions are restrictions to the
bounded set $D$ of locally bounded $W^{1,p}_{\mathrm{loc}}(\rn{n})$
functions, so both belong to $W^{1,p}(D)$.

\begin{lemma}\label{lem:levels}
If $1<p<n$, then
\begin{equation}\label{eq:levelcapacity}
\capp(K_E(t))=t^{1-p}\capp(E),\qquad
\underline u_{K_E(t)}=\min\{\underline u_E/t,1\}
\end{equation}
for the lower semicontinuous representatives. Each level set $K_E(t)$ is filled.
For $p\geq n$, suppose that
$\partial K_E(t)\cap E=\varnothing$. Then $K_E(t)$ is filled,
and its normalized Green's function is $t-w_E$ outside it,
extended by zero on it. Thus
\begin{equation}\label{eq:levelgreen}
k_{K_E(t)}=k_E+t.
\end{equation}
\end{lemma}
\begin{proof}
Let $p<n$. If $f$ is $p$-superharmonic with $\nabla f\in L^p$ and
Riesz measure $\mu_f$ supported in $B(0,R_0)$, and $\psi\in L^{p^*}$
has $\nabla\psi\in L^p$, then
\begin{equation}\label{eq:whole-space-test}
\int|\nabla f|^{p-2}\langle\nabla f,\nabla\psi\rangle\,dx
=\int\tilde\psi\,d\mu_f.
\end{equation}
Indeed, $\mu_f$ is a continuous functional on the energy space and
does not charge polar sets, so we may test with $\chi_R\psi$, where
$\chi_R$ is the cutoff chosen before \eqref{eqn4.6} and
$R\geq R_0$. With $A_R=B(0,2R)\setminus B(0,R)$, H\"older's
inequality and \eqref{eqn4.6} bound the cutoff term by
$c\|\nabla f\|_{L^p(A_R)}^{p-1}\|\psi\|_{L^{p^*}(A_R)}\to0$.

Write $u=w_E$ and $\mu=-\Delta_pu$. Applying
\eqref{eq:whole-space-test} with $\psi=u$ and with
$\psi=\min(u,t)$, both one or $t$ quasi-everywhere on $E$, gives
$\mu(E)=\capp(E)$ and
\[
\int_{\{u<t\}}|\nabla u|^p\,dx
=\int\min(u,t)\,d\mu=t\,\operatorname{Cap}_p(E).
\]
The sets $K_E(t)$ and $\{u\geq t\}$ differ only on the polar
subset of $E$ where $W_E\ne u$. Hence $v=\min(u/t,1)$ is admissible
for $K_E(t)$ in the class of functions on $\rn{n}$ specified above. It is
$p$-superharmonic and $p$-harmonic outside $K_E(t)$. Truncate any
admissible competitor to a function $z$ with $0\leq z\leq1$. Then
$z-v=0$ quasi-everywhere on $K_E(t)$, and
\eqref{eq:whole-space-test} with $f=v$ and $\psi=z-v$ gives
$\int|\nabla v|^{p-2}\nabla v\cdot\nabla(z-v)\,dx=0$. Convexity of
the energy gives $\int|\nabla z|^p\geq\int|\nabla v|^p$, so $v$ is the
capacitary minimizer for $K_E(t)$, and
\[
\operatorname{Cap}_p(K_E(t))
=\int|\nabla v|^p\,dx
=t^{-p}\int_{\{u<t\}}|\nabla u|^p\,dx
=t^{1-p}\operatorname{Cap}_p(E).
\]
This proves \eqref{eq:levelcapacity}.

For filling, let $D$ be a bounded component of the complement of
$K_E(t)$. Since $\hat E\subset K_E(t)$, $D\subset\Omega_E$ and
$\partial D\cap E\subset\partial\Omega_E$. At points of
$\partial D\setminus E$, $w_E$ has boundary lower limit at least $t$
by continuity. A point of $\partial D\cap E$ regular for $\Omega_E$
would be interior to $K_E(t)$ by Lemma~\ref{lem:regular-point-limit},
and the irregular points form a polar set by the Kellogg property
\cite[Theorem~9.11]{HKM}. The comparison above, with the constant
$t\in W^{1,p}(D)$, gives $w_E\geq t$ in
$D$, contradicting $W_E=w_E<t$ there.

For $p\geq n$, under the stated separation, $w_E=t$ on
$\partial K_E(t)$ by continuity off $E$. On a bounded component of the
complement, $t-w_E$ would be positive, $p$-harmonic, continuous and
zero on the boundary, contrary to the maximum principle. Hence
$K_E(t)$ is filled. In its exterior $t-w_E$ is positive and
$p$-harmonic, continuous and zero on the boundary, and equals
$F+k_E+t+o(1)$ at infinity. Its zero extension satisfies
\eqref{eqn2.9}, so by exterior uniqueness it is the normalized Green's
function, and its constant term gives \eqref{eq:levelgreen}.
\end{proof}

\begin{lemma}\label{lem:convex-trace}
Suppose $K_E(T)=C$ is convex with nonempty interior and $T<m$.
Then $w_E=T$ at every point of $\partial C$, its exterior values
tend continuously to $T$ there, and $w_E\geq T$ throughout
$\operatorname{int}C$. In the exterior of $C$, $w_E=Tu_C$ if $p<n$
and $w_E=T-G_C$ if $p\geq n$.
\end{lemma}
\begin{proof}
At points of $\partial C\setminus E$, continuity gives boundary
value $T$. Since $\hat E\subset C$, points of $E\cap\partial C\subset\partial\Omega_E$ are
irregular for $\Omega_E$ by Lemma~\ref{lem:regular-point-limit}, so
they form a polar set \cite[Theorem~9.11]{HKM}.

For $p<n$, \eqref{eq:levelcapacity} gives $w_E=Tu_C$
outside $C$, since $C$ is regular and $u_C<1$ there.

For $p\geq n$, the positive functions $H=T-w_E$ and $G_C$ have the
same leading term $F$ at infinity, so for $0<\varepsilon<1$,
$(1-\varepsilon)G_C\leq H\leq(1+\varepsilon)G_C$ on every large
sphere $\partial B(0,R)$. The polar boundary comparison above in
$B(0,R)\setminus C$, in both directions and with
$G_C\in W^{1,p}(B(0,R)\setminus C)$, and the limits
$R\to\infty$, $\varepsilon\downarrow0$ give $H=G_C$.
Convex bodies are regular, so these identifications give the
continuous exterior boundary value $T$. Finally, $W_E\geq T$ in $C$
and $w_E=W_E$ quasi-everywhere, so \eqref{eq:superharmonic-representative}
gives $w_E\geq T$ in $\operatorname{int}C$. At each point of
$\partial C$ a supporting half-ball lies in the exterior, so the same
property, the exterior limit $T$ and the interior inequality give
$w_E=T$ there.
\end{proof}

\subsection{Common convex level sets}\label{subsec:common}

Let $a_i=\Phi_p(E_i)>0$, and suppose equality holds in \eqref{eq:equality}. Put
\[
\varrho=\lambda a_1+(1-\lambda)a_2,\qquad
\theta=\frac{\lambda a_1}{\varrho},\qquad A=E_1/a_1,\quad B=E_2/a_2.
\]
Set $M=(\lambda E_1+(1-\lambda)E_2)/\varrho$. Then
\begin{equation}\label{eq:normalization}
M=\theta A+(1-\theta)B,\qquad
\Phi_p(A)=\Phi_p(B)=\Phi_p(M)=1.
\end{equation}
Translation of $B$ by $b$ translates $M$ by $(1-\theta)b$
and preserves this relation. Theorem~\ref{thmC} applied to these
normalized sets gives
\begin{equation}\label{eq:pointwise}
W_M(\theta y+(1-\theta)z)
\geq\min\{W_A(y),W_B(z)\}\qquad(y,z\in\rn{n}).
\end{equation}

\begin{lemma}\label{lem:strict-convex-monotonicity}
If $K\subsetneq L$ are compact convex sets with nonempty
interiors, then $\Phi_p(K)<\Phi_p(L)$.
\end{lemma}
\begin{proof}
Let $p<n$ and suppose the capacities are equal. Then $u_L$ is
admissible for $K$ with minimal energy, so $u_K=u_L$
quasi-everywhere by uniqueness, and $u_K=1$ on a ball in
$\operatorname{int}L\setminus K$, contrary to the strong maximum
principle in the connected exterior of $K$.

For $p\geq n$, monotonicity gives $h=G_K-G_L\geq0$ in $\Omega_L$.
By \cite[Lemma~6.1]{ALSV} the gradients of the Green's functions do not vanish
in the exterior, so on each relatively compact ball in $\Omega_L$ the
coefficient calculation in the proof of Lemma~\ref{lem9.7}, which
needs a lower bound on only one of the two gradients, makes $h$
a solution of a uniformly elliptic linear equation. At a point of
$\partial L\setminus K$, $G_K>0=G_L$, so continuity and Harnack
chains give $h>0$ in $\Omega_L$. Choose $R$ with $L\subset B(0,R)$
and put $c_0=\min_{\partial B(0,R)}h>0$. For $0<\epsilon<1$,
compare $G_K$ with $(1-\epsilon)G_L+c_0$ in large exterior annuli.
The inequality on the inner sphere holds by the choice of $c_0$, and on
large outer spheres the difference is
$\epsilon F+k_K-(1-\epsilon)k_L-c_0+o(1)\to+\infty$. Letting the
outer radius increase and then $\epsilon\downarrow0$ gives
$G_K\geq G_L+c_0$ outside $B(0,R)$, so $k_K-k_L\geq c_0>0$, and the
definitions of $\Phi_p$ give the claim for $p\geq n$.
\end{proof}

\begin{lemma}\label{lem:farlevels}
There is $b\in\rn{n}$ such that, after replacing $B$ by $B+b$
and $M$ by $M+(1-\theta)b$, the functions $W_A,W_B,W_M$
agree near infinity. For every sufficiently small $t>0$ when $p<n$,
and every sufficiently negative $t$ when $p\geq n$, the three level
sets coincide and have smooth strictly convex boundaries.
\end{lemma}
\begin{proof}
Let $E$ be one of $A,B,M$. Set
\[
\beta=\frac{n-p}{p-1},\quad
\ell=\beta^{-1}\sigma_{n-1}^{-1/(p-1)}\quad(p<n),\qquad
 a_p=\xi^{-1}\sigma_{n-1}^{-1/(p-1)}\quad(p>n),
\]
where $\xi=(p-n)/(p-1)$ as in \eqref{eqn2.7}.
The expansions in Lemma~\ref{lem9.15} $(c)$ and
\eqref{eq:equality-green-asymptotic} give uniform convergence on
every fixed annulus away from the origin,
\[
f_{E,R}(z):=
\begin{cases}
 R^\beta w_E(Rz)/\ell,&p<n,\\
 [-w_E(Rz)-\gamma\log R-1-k_E]/\gamma,&p=n,\\
 [-w_E(Rz)-k_E]/(a_pR^\xi),&p>n
\end{cases}
\ \longrightarrow\ 
\begin{cases}
 |z|^{-\beta},&p<n,\\
 \log|z|,&p=n,\\
 |z|^\xi,&p>n.
\end{cases}
\]
The normalization \eqref{eq:normalization} makes the constants
$k_E$ equal for $E=A,B,M$ when $p\geq n$.
Interior $C^{1,\alpha}$ estimates give $C^1$ convergence on
$\{1/3\leq|z|\leq3\}$. The limits have nonzero gradient there, so
for large $R$ the equations are uniformly elliptic, and Schauder
estimates give $C^2$ convergence on $\{2/5\leq|z|\leq5/2\}$.
Where the gradient does not vanish, \eqref{eqn3.3} is uniformly
elliptic with coefficients analytic in the gradient, so
$p$-harmonic functions are real analytic there.

Use level $1$ for $f_{E,R}$ when $p\ne n$ and level $0$ when
$p=n$. The corresponding common level of the unscaled functions is
\[
t_R=\begin{cases}
\ell R^{-\beta},&p<n,\\
-\gamma\log R-1-k_E,&p=n,\\
-a_pR^\xi-k_E,&p>n.
\end{cases}
\]
For sufficiently large $R_0$, the continuous decreasing parameter
$t_R$ covers $(0,t_{R_0}]$ when $p<n$ and
$(-\infty,t_{R_0}]$ when $p\geq n$ as $R$ ranges over
$[R_0,\infty)$.
The expansions at infinity hold uniformly on $|x|\geq2R$ and give
$w_E<t_R$ there for large $R$, while $w_E>t_R$ on
$\partial B(0,R/2)$. The minimum principle for $w_E$ in
$B(0,R/2)$ and $W_E\geq w_E$ give
$B(0,R/2)\subset K_E(t_R)\subset B(0,2R)$.
On the annulus between, $W_E=w_E$ and the radial derivative has a
fixed sign for large $R$, so each ray meets the level once. By the
implicit function theorem and the $C^2$ convergence, the boundary
is a radial graph converging in $C^2$ to the unit sphere after
rescaling, so it has positive principal curvatures for large $R$.
These level sets are regular and their boundaries miss $E$, so
Lemma~\ref{lem:levels} applies in every range of $p$.

By \eqref{eq:levelcapacity} or \eqref{eq:levelgreen}, the three level sets for these values of $t$ have equal capacities. By \eqref{eq:pointwise} their Minkowski combination lies in the level set of $M$. Monotonicity and the convex Brunn--Minkowski inequality give
\[
\Phi_p(K_M(t))\geq
\Phi_p(\theta K_A(t)+(1-\theta)K_B(t))
\geq\theta\Phi_p(K_A(t))+(1-\theta)\Phi_p(K_B(t))
=\Phi_p(K_M(t)).
\]
For $1<p<n$, the convex equality theorem in \cite{CS} implies that $K_A(t)$ and $K_B(t)$ are homothetic. For $p\geq n$, use \cite[Theorem~A]{ALSV} for the $p$-Laplace operator.
Equal capacities make them translates, and Lemma~\ref{lem:strict-convex-monotonicity} makes their Minkowski combination equal to $K_M(t)$.

Choose one such level and translate $B$ by $b$, and $M$ by
$(1-\theta)b$, so that the three level sets agree. By exterior
uniqueness, each exterior function is $t$ times the capacitary
function of that body if $p<n$, and $t-w_E$ is its normalized
Green's function if $p\geq n$. Thus the three functions agree in
its exterior, and all level sets with smaller level values coincide, as in the
Newtonian case \cite{CJL}.
\end{proof}

Let $I=(0,1)$ if $p<n$, and let $I=(-\infty,0)$ if $p\geq n$.
After the translation in Lemma~\ref{lem:farlevels}, define
\begin{equation}\label{eq:continuation-supremum}
T_*:=\sup\left\{T\in I:
\begin{array}{l}
 K_A(t)=K_B(t)=K_M(t)\text{ is convex}\\
 \text{for every }t\in I\text{ with }t\leq T
\end{array}\right\}.
\end{equation}
The defining set is nonempty by Lemma~\ref{lem:farlevels},
and $T_*\leq m$.

\begin{lemma}\label{lem:bootstrap}
Let $T\in I$, and suppose that the three level sets coincide
and are convex for every $t\in I$ with $t\leq T$.
Suppose their common set $C$ at $T$ has nonempty interior. If
\begin{equation}\label{eq:no-contact}
(A\cup B\cup M)\cap\partial C=\varnothing,
\qquad C=K_A(T)=K_B(T)=K_M(T),
\end{equation}
then the common convex level sets continue to an interval of levels above $T$.
\end{lemma}
\begin{proof}
Equality of all the levels below $T$ gives equality of the
three functions outside $C$, and by \eqref{eq:no-contact} they are
$p$-harmonic near $\partial C$.
Let $z\in\partial C$, let $\nu$ be a supporting outward unit normal,
and choose $\rho>0$ so small that the tangent ball
$B(z+\rho\nu,\rho)$ lies in this neighbourhood.
Hopf's lemma (Lemma~\ref{lem3.3}) for $T-w_E$ in this ball gives each
$w_E$ a nonzero gradient at $z$, and these gradients coincide since
the functions agree outside $C$ and are $C^1$ near $z$. Hence in a
ball $B(z,r_z)$ each function is $p$-harmonic with nonzero gradient,
and so real analytic (see the proof of Lemma~\ref{lem:farlevels}). They agree on the open set
$B(z,r_z)\setminus C$, so by the identity theorem they agree on
$B(z,r_z)$. Their common gradient points inward, since the exterior
values are less than $T$. By the implicit function theorem
$\partial C$ is locally a smooth level surface with values greater
than $T$ on its inner side. A finite cover gives a neighbourhood of
$\partial C$ in which the functions agree, their common gradient is
nonzero, and all nearby level surfaces are smooth.

By Lemma~\ref{lem:convex-trace}, each function is at least $T$ in
$\operatorname{int}C$. The strong minimum principle, which follows
from the weak Harnack inequality \cite[Theorem~3.59]{HKM}, and the
values in the part of this neighbourhood inside $C$ make it greater
than $T$ throughout the connected set $\operatorname{int}C$.
Lower semicontinuity then gives a lower bound $T+\varepsilon$, for some
$\varepsilon>0$, on the part of $C$ outside a smaller neighbourhood
of $\partial C$.

Hence for $s>T$ close to $T$ the three entire level sets coincide.
Call the common set $K$. By \eqref{eq:pointwise},
$\theta K+(1-\theta)K\subset K$.
A closed set closed under one fixed nontrivial weighted average is convex, since a gap in its intersection with a segment would have endpoints in the set whose weighted average lies strictly inside the gap.
\end{proof}

\begin{lemma}\label{lem:continuation-endpoint}
If $T_*<m$, then
$K_A(T_*)=K_B(T_*)=K_M(T_*)=:C$ is a convex compact set
with nonempty interior. The three functions agree outside $C$,
and $(A\cup B\cup M)\cap\partial C\ne\varnothing$.
If $T_*=m$, then
\begin{equation}\label{eq:intersection}
\widehat A=\widehat B=\widehat M=\bigcap_{t\in I}K_A(t),
\end{equation}
and this set is convex.
\end{lemma}
\begin{proof}
For $t\in I$ with $t<T_*$ the three level sets are common and
convex. Suppose $T_*<m$. For $E=A,B,M$,
$K_E(T_*)=\bigcap_{t\in I,\,t<T_*}K_E(t)$, so the common value $C$
is a compact convex set containing $A\cup B\cup M$. Since
$\widehat{\partial\widehat A}=\widehat A$, filling invariance and
$\Phi_p(A)>0$ show that $\partial\widehat A$ is not $p$-polar (for $p<n$
since capacity in $\rn{n}$ is at most relative capacity, for
$p=n$ by Theorem~\ref{thmB} $(d)$, and for $p>n$ trivially). The
Kellogg property \cite[Theorem~9.11]{HKM} gives a regular point
$x_0\in\partial\widehat A$, and Lemma~\ref{lem:regular-point-limit}
puts a neighbourhood of $x_0$ in $C$. Every point outside $C$ lies
outside some common level with $t<T_*$, so the three functions agree
there. If the contact condition failed, Lemma~\ref{lem:bootstrap}
would continue the common convex levels beyond $T_*$, contrary to
\eqref{eq:continuation-supremum}.

If $T_*=m$, all levels in $I$ are common and convex.
For each $E=A,B,M$, their intersection is $\{W_E=m\}=\hat E$.
The first equality uses $W_E\leq m$, and the second follows
from $W_E=m$ on $\hat E$ and the strict maximum principle in
its unbounded exterior. This proves \eqref{eq:intersection}.
\end{proof}

In the contact arguments below, put $T=T_*<m$ and
\[
C:=K_A(T)=K_B(T)=K_M(T),\qquad
W:=W_A|_{\Omega_C}=W_B|_{\Omega_C}=W_M|_{\Omega_C}.
\]
The set $C$ is the common level set at $T_*$ in
Lemma~\ref{lem:continuation-endpoint}, and $W$ is its common
exterior function. A \emph{contact point} is a point of
$(A\cup B\cup M)\cap\partial C$.

\begin{lemma}\label{lem:regular-contact}
Every contact point lies in $M\cap\partial C$ and is an
irregular point of $\partial\Omega_M$. In particular, contact
is impossible if $M$ is $p$-regular.
\end{lemma}
\begin{proof}
Suppose $x\in A\cap\partial C$ and $x\notin M$. Since
$A\cap B\subset M$, also $x\notin B$, so $w_B,w_M$ are $p$-harmonic
near $x$, agree outside $C$, and equal $T$ at $x$ by
Lemma~\ref{lem:convex-trace}. The argument in the first paragraph of
the proof of Lemma~\ref{lem:bootstrap}, applied to them, gives a ball $B(x,\rho)$ on which they
agree with a common function $v$ satisfying $v(x)=T$ and
$\nabla v(x)\ne0$. Keeping $W_A(x)=m$ fixed in \eqref{eq:pointwise}
gives $v(\theta x+(1-\theta)y)-v(y)\geq0$ for $y$ near $x$, with
equality at $y=x$. Differentiation at $y=x$ gives
$-\theta\nabla v(x)=0$, a contradiction. For
$x\in B\cap\partial C\setminus M$, fix $W_B(x)=m$ instead, which gives
$-(1-\theta)\nabla v(x)=0$. Thus every contact point lies in
$M\cap\partial C\subset\partial\hat M=\partial\Omega_M$. At a regular
point, Lemma~\ref{lem:regular-point-limit} would put a neighbourhood
in $C$, which is impossible.
\end{proof}

\subsection{Local estimates}\label{subsec:local}

\begin{lemma}\label{lem:source-capacity}
Let $w\in W^{1,p}_{\mathrm{loc}}(U)$ be a locally bounded $p$-superharmonic function, and let $\mu=-\Delta_p w$. If $0\leq w\leq L$ in $U$, then for compact $K\subset U$,
\begin{equation}\label{eq:measure-capacity}
\mu(K)\leq p^p L^{p-1}\capp(K,U).
\end{equation}
One may replace $L$ by the oscillation of $w$ after subtracting a constant.
\end{lemma}
\begin{proof}
The bounded Sobolev tests below are admissible in the weak
equation by nonnegative approximation and truncation
\cite[Lemma~3.11]{HKM}, with quasi-continuous representatives in
the measure integrals. For nonnegative $\phi\in C_c^\infty(U)$, testing with $(L-w)\phi^p$ and using the sign of $\mu$ gives
$X:=\int\phi^p|\nabla w|^p\leq pL X^{(p-1)/p}(\int|\nabla\phi|^p)^{1/p}$,
so $X\leq p^pL^p\int|\nabla\phi|^p$. Testing with $\phi^p$ and using H\"older's inequality gives
\[
\int\phi^p\,d\mu
\leq pX^{(p-1)/p}\left(\int|\nabla\phi|^p\right)^{1/p}
\leq p^pL^{p-1}\int|\nabla\phi|^p.
\]
Take the infimum over $\phi\geq1$ on $K$.
\end{proof}

\begin{lemma}\label{lem:exterior-oscillation}
Let $1<p\leq n$ and use the common convex set $C$,
level $T=T_*<m$, and exterior function $W$ fixed after
Lemma~\ref{lem:continuation-endpoint}.
Fix $x\in\partial C$, put $v=T-W$ in $\Om_C$ and $v=0$ on $C$, and
let $\eta(r)=\sup_{B(x,8r)}v$. For small $r$, $\eta(r)>0$ and $\eta(r)\to0$. If $\nu$ is a
supporting outward unit normal to $C$ at $x$, then for every fixed
$0<b<1$,
\begin{equation}\label{eq:osc-carleson}
\eta(r)\leq c_b v(x+br\nu).
\end{equation}
For $i=A,B,M$, set $V_i=w_i-T+2\eta(r)$. Then
\begin{equation}\label{eq:osc-weak-harnack}
\left(\frac1{|B(x,r)|}\int_{B(x,r)}V_i^q\,dy\right)^{1/q}
\leq c_q\eta(r),\qquad 0<q<q_H,
\quad q_H=\begin{cases}n(p-1)/(n-p),&p<n,\\+\infty,&p=n.
\end{cases}
\end{equation}
\end{lemma}
\begin{proof}
Lemma~\ref{lem:convex-trace} identifies $v$ with $T(1-u_C)$ if
$p<n$ and with $G_C$ if $p=n$. It is positive and $p$-harmonic
outside $C$, continuous and zero on $\partial C$, and its zero
extension belongs locally to $W^{1,p}$. Since $\partial C$ is
locally Lipschitz, the boundary H\"older and Carleson estimates
\cite{LN10}, in the local graph form
\cite[Lemma~8.1, (8.4)]{ALSV} whose constants depend only
on the equation data and the Lipschitz constant, apply on a ball
containing $B(x,8r)$. The Carleson estimate bounds $\eta(r)$ by the
value at an exterior corkscrew point, and since
$\operatorname{dist}(x+t\nu,C)=t$, a Harnack chain along the normal
compares that value with $v(x+br\nu)$. This gives
\eqref{eq:osc-carleson}, and continuity gives $\eta(r)\to0$.

Inside $C$, $w_i\geq T$ by Lemma~\ref{lem:convex-trace}, and
outside $C$, $w_i=W\geq T-\eta(r)$ on $B(x,8r)$. Thus
$V_i\geq\eta(r)$ is a positive weak supersolution there, less than
$2\eta(r)$ on an exterior open subset of $B(x,r)$. The weak Harnack
inequality \cite[Theorem~3.59]{HKM}, with $B(x,r)$ and its
concentric ball $B(x,4r)$, proves \eqref{eq:osc-weak-harnack} for
every $0<q<q_H$, including $q<1$.
\end{proof}

Write $\mathcal A(\mathbf v)=|\mathbf v|^{p-2}\mathbf v$ for
$\mathbf v\in\rn{n}$ and $p'=p/(p-1)$. Whenever the distributional $p$-Laplacian of $v$ is a measure,
write $\mu_v=-\Delta_pv$. If $v$ is $p$-superharmonic and locally
bounded below, each truncation $\min(v,L)$ belongs to
$W^{1,p}_{\mathrm{loc}}$, the generalized gradient is defined almost
everywhere by these truncations, $\mathcal A(\nabla v)$ is locally
integrable, and $-\operatorname{div}\mathcal A(\nabla v)$ is the
positive Riesz measure, although $v$ need not belong to
$W^{1,1}_{\mathrm{loc}}$ for small $p$. The superharmonic
representative agrees quasi-everywhere with the quasi-continuous
one when $v$ is Sobolev \cite[Chapter~21]{HKM}.

\begin{lemma}\label{lem:contact-measures}
Let $O$ be open and let $a,b\in W^{1,p}_{\mathrm{loc}}(O)$ be locally bounded, with $a\geq b$ quasi-everywhere. Suppose $-\Delta_p a=\mu_a$ and $-\Delta_p b=\mu_b$, where $\mu_a,\mu_b$ are locally finite signed Radon measures whose positive and negative parts locally belong to $W^{-1,p'}$. Then, with quasi-continuous representatives,
\begin{equation}\label{eq:contact-measure-order}
\mu_a\big|_{\{a=b\}}\leq\mu_b\big|_{\{a=b\}}.
\end{equation}
In particular, locally bounded weak supersolutions satisfy the
measure hypotheses.
\end{lemma}
\begin{proof}
Put $f=a-b\geq0$ and take $0\leq\phi\in C_c^\infty(O)$. For $\varepsilon>0$, set $\chi_\varepsilon(t)=(1-t/\varepsilon)_+$ for $t\geq0$. The bounded Sobolev test $\phi\chi_\varepsilon(f)$ is admissible in both weak equations. Indeed, if $0\leq\mu\in W^{-1,p'}$ and $g\in W^{1,p}_0$ is bounded with compact support, then mollifications of $g$ are uniformly bounded with supports in a fixed compact set, converge to $g$ in $W^{1,p}$, and along a subsequence converge to $\tilde g$ quasi-everywhere \cite[Chapter~4]{HKM}. Since $\mu$ does not charge polar sets, dominated convergence gives $\langle\mu,g\rangle=\int\tilde g\,d\mu$. Applying this to the positive and negative parts of $\mu_a$ and $\mu_b$ and subtracting, we get
\begin{align*}
\int\phi\chi_\varepsilon(f)\,d(\mu_a-\mu_b)
&=\int\chi_\varepsilon(f)
 \bigl(\mathcal A(\nabla a)-\mathcal A(\nabla b)\bigr)\cdot\nabla\phi\,dx\\
&\quad-\frac1\varepsilon\int_{0<f<\varepsilon}\phi
 \bigl(\mathcal A(\nabla a)-\mathcal A(\nabla b)\bigr)
 \cdot(\nabla a-\nabla b)\,dx\\
&\leq\int\chi_\varepsilon(f)
 \bigl(\mathcal A(\nabla a)-\mathcal A(\nabla b)\bigr)\cdot\nabla\phi\,dx.
\end{align*}
The inequality is the monotonicity of $\mathcal A$. On $\{f=0\}$ the gradients of $a,b$ agree almost everywhere, so by dominated convergence the last integral tends to zero, while the left side tends to $\int_{\{a=b\}}\phi\,d(\mu_a-\mu_b)$, since the pointwise limits hold outside a polar set, which no measure part charges. This proves \eqref{eq:contact-measure-order}.
\end{proof}

\begin{lemma}\label{lem:truncation-locality}
Let $v$ be $p$-superharmonic and locally bounded below in $O$, and let $\mu_v=-\Delta_p v$. If $v_L=\min(v,L)$, then
\begin{equation}\label{eq:truncation-locality}
\mu_{v_L}\big|_{\{v<L\}}=\mu_v\big|_{\{v<L\}}.
\end{equation}
In particular, for $L>0$ the two measures agree on $\{v=0\}$.
\end{lemma}
\begin{proof}
Replacing $(v,L)$ by $(v+c,L+c)$ changes nothing, so locally we may
assume $v\geq0$. Write $\mu_v=\mu_0+\mu_s$, where $\mu_0$ is
absolutely continuous with respect to $p$-capacity and $\mu_s$ is
concentrated on a set of zero $p$-capacity. By the theorem of \cite{KKT}
that superharmonic functions are locally renormalized solutions, $v$ is a
local renormalized solution, that is, for every $\phi\in C_c^\infty(O)$ and every Lipschitz
function $h$ on $\re$ whose derivative has compact support,
\begin{equation}\label{eq:renormalized-test}
\int\mathcal A(\nabla v)\cdot\nabla[h(v)\phi]\,dx
=\int h(v)\phi\,d\mu_0
+h(+\infty)\int\phi\,d\mu_s,
\end{equation}
with the superharmonic representative of $v$ in $h(v)$.
The same work \cite{KKT} shows that $\mu_s$ is concentrated on $\{v=+\infty\}$.

Let $h$ be Lipschitz with compact support in $(-\infty,L)$. Then
$h(+\infty)=0$ and $h(v)=0$ on $\{v=+\infty\}$, so the right side
of \eqref{eq:renormalized-test} is $\int h(v)\phi\,d\mu_v$.
Since $v_L$ is locally bounded, $\mu_{v_L}$ has no singular part and
\eqref{eq:renormalized-test} for $v_L$ has right side
$\int h(v_L)\phi\,d\mu_{v_L}$. As $h$ and $h'$ vanish on
$[L-\varepsilon,\infty)$ for some $\varepsilon>0$, we have
$h(v_L)=h(v)$, and the left sides agree, since the gradients agree on
$\{v<L\}$ and both integrands vanish on $\{v\geq L\}$. Hence
$h(v)\mu_{v_L}=h(v)\mu_v$. Letting nonnegative such $h$ increase to
$1$ on $[0,L)$ gives \eqref{eq:truncation-locality} by monotone
convergence.
\end{proof}

\begin{lemma}\label{lem:riesz-capacity-limit}
Let $O_0\subset\rn{n}$ be open and let $k_j>0$ tend to infinity. Suppose $v_j\in W^{1,p}_{\mathrm{loc}}(O_0)$ are locally bounded weak supersolutions with $-\Delta_p v_j=\omega_j\geq0$. Let $E_j$ be relatively closed in $O_0$, with $\operatorname{supp}\omega_j\subset E_j$ and $v_j=k_j$ quasi-everywhere on $E_j$. Suppose that on every relatively compact open subset of $O_0$,
\begin{equation}\label{eq:capacity-limit-hypotheses}
-c\leq v_j\leq k_j,\qquad
\sup_j\int|v_j|^q\,dx<\infty
\end{equation}
for some $p-1<q<p$, where $c,q$ may depend on the subset.
If $\omega_j$ converge locally weakly to $\omega$, then for bounded open
sets $D\Subset O\Subset O_0$ with $\omega(\partial D)=0$,
\begin{equation}\label{eq:riesz-capacity-limit}
\lim_{j\to\infty}k_j^{p-1}
\capp(E_j\cap\overline D,O)=\omega(D).
\end{equation}
The limit is independent of the fixed outer domain $O$.
\end{lemma}
\begin{proof}
We use quasi-continuous representatives in Sobolev tests and capacity admissibility. The measures $\omega_j$ are locally continuous functionals on $W^{1,p}_0$, so $v_j=k_j$ also $\omega_j$-almost everywhere. On each compact neighbourhood, since $p-q<1$,
\begin{equation}\label{eq:capacity-small-function-energy}
\int|v_j|^p\,dx\leq c(1+k_j^{p-q})=o(k_j).
\end{equation}
For a nonnegative smooth cutoff $\zeta$, testing with $\zeta^p v_j$ gives
\[
X_j:=\int\zeta^p|\nabla v_j|^p\,dx
=k_j\int\zeta^p\,d\omega_j
-p\int\zeta^{p-1}v_j\mathcal A(\nabla v_j)\cdot\nabla\zeta\,dx
\leq ck_j+\tfrac12X_j+c\int|v_j\nabla\zeta|^p\,dx
\]
by the local weak convergence, H\"older and Young, so $X_j=O(k_j)$. Testing with $\phi v_j$ for fixed smooth compactly supported $\phi$ gives
\[
\frac1{k_j}\int\phi|\nabla v_j|^p\,dx
=\int\phi\,d\omega_j
-\frac1{k_j}\int v_j\mathcal A(\nabla v_j)\cdot\nabla\phi\,dx,
\]
and by the energy bound on a larger cutoff the last term is at most $c k_j^{-1} k_j^{(p-1)/p}(1+k_j^{p-q})^{1/p}\to0$, since $q>p-1$. Thus, locally weakly as measures,
\begin{equation}\label{eq:energy-riesz-limit}
\frac{|\nabla v_j|^p}{k_j}\,dx\rightharpoonup\omega.
\end{equation}
Testing with $-\zeta^p(v_j)_-$ gives no measure term, since $v_j=k_j>0$ on the support of $\omega_j$, and yields $\int\zeta^p|\nabla(v_j)_-|^p\leq c\int(v_j)_-^p|\nabla\zeta|^p\leq c$. Hence \eqref{eq:energy-riesz-limit} also holds with $\nabla(v_j)_+$.

For the upper bound choose $0\leq\zeta\leq1$ in $C_c^\infty(O)$ equal to one near $\overline D$. The function $f_j=\zeta(v_j)_+/k_j\in W^{1,p}_0(O)$ equals one quasi-everywhere on $E_j\cap\overline D$. The inequality $\bigl||a+b|^p-|a|^p\bigr|\leq c_p(|a|^{p-1}|b|+|b|^p)$ with $a=\zeta\nabla(v_j)_+$ and $b=(v_j)_+\nabla\zeta$, \eqref{eq:capacity-small-function-energy} and the $O(k_j)$ energy bound show that the cutoff error divided by $k_j$ tends to zero. Hence $\limsup_j k_j^{p-1}\capp(E_j\cap\overline D,O)\leq\int\zeta^p\,d\omega$. Letting the supports decrease to $\overline D$ and using $\omega(\partial D)=0$ gives the upper bound $\omega(D)$.

For the reverse bound let $e_j$ be the capacitary function of $K_j=E_j\cap\overline D$ relative to $O$. It is $p$-harmonic off $K_j$, equals one quasi-everywhere on $K_j$, and its Riesz measure is supported on $K_j$ with total mass $\capp(K_j,O)$. Choose $c_O$ with $v_j\geq-c_O$ on $O$ and put $z_j=-c_O+(k_j+c_O)e_j$. Then $(z_j-v_j)_+\in W^{1,p}_0(O)$ vanishes quasi-everywhere on $K_j$, so it belongs to $W^{1,p}_0(O\setminus K_j)$ \cite[Theorem~4.5]{HKM}. Testing the equation for $z_j$ off $K_j$ and the supersolution inequality for $v_j$ with it, strict monotonicity gives $z_j\leq v_j$ almost everywhere, hence quasi-everywhere \cite[Theorem~4.12]{HKM}. Lemma~\ref{lem:contact-measures} on $\{z_j=v_j\}\supset K_j$ (up to a polar set), with $\mu_{z_j}=(k_j+c_O)^{p-1}\mu_{e_j}$ and $\operatorname{supp}\omega_j\subset E_j$, gives $\omega_j(\overline D)=\omega_j(K_j)\leq(k_j+c_O)^{p-1}\capp(K_j,O)$. Since $\omega_j(\overline D)\to\omega(D)$, multiplying by $(k_j/(k_j+c_O))^{p-1}$ gives the lower bound in \eqref{eq:riesz-capacity-limit}.
\end{proof}

\begin{lemma}\label{lem:cone-comparison}
Let $\mathfrak C\subsetneq\rn{n}$ be a closed convex cone with nonempty interior. Suppose $Q\in C(\rn{n})\cap W^{1,p}_{\mathrm{loc}}(\rn{n})$, $Q\leq0$, $Q=0$ on $\mathfrak C$, $Q\not\equiv0$, and $Q$ is $p$-harmonic outside $\mathfrak C$ with $\sigma=-\Delta_p Q\geq0$ supported on $\partial\mathfrak C$. For $0<s<1$, there are no locally bounded below $p$-superharmonic functions $U,V$ satisfying
\begin{equation}\label{eq:cone-comparison-hypotheses}
U\geq V,\qquad \mu_U\geq\mu_V,\qquad
U=Q(s\,\cdot),\ V=Q\text{ outside }\mathfrak C,
\end{equation}
and $U,V\geq0$ in $\operatorname{int}\mathfrak C$.
\end{lemma}
\begin{proof}
The functions $U,V$ are zero at every point of $\partial \mathfrak C$, since their essential lower limits there are zero: outside $\mathfrak C$ their values tend continuously to zero, and inside they are nonnegative. Also $\partial\mathfrak C$ is Lebesgue null.

Set $Q_s(z)=Q(sz)$ and $\sigma_s=-\Delta_p Q_s$. Since $\mathfrak C$ is a cone, $s\mathfrak C=\mathfrak C$, and
\begin{equation}\label{eq:cone-measure-scaling}
\sigma_s(D)=s^{p-n}\sigma(sD)
\end{equation}
for Borel $D$. Fix $L>0$ and put $U_L=\min(U,L)$, $V_L=\min(V,L)$. These are locally bounded Sobolev weak supersolutions. The Sobolev function $P_L=V_L+Q_s-Q$ equals $V_L$ in $\mathfrak C$ and $Q_s=U_L$ outside, so $U_L\geq P_L$ almost everywhere, hence quasi-everywhere by \cite[Theorem~4.12]{HKM} applied to $(P_L-U_L)_+$. The Sobolev representatives of $U_L,V_L$ agree quasi-everywhere with their superharmonic representatives, which vanish on $\partial\mathfrak C$, and $Q_s-Q=0$ there, so $U_L$ and $P_L$ are zero quasi-everywhere on $\partial\mathfrak C$.

Almost everywhere,
$\mathcal A(\nabla P_L)=\mathcal A(\nabla V_L)+\mathcal A(\nabla Q_s)-\mathcal A(\nabla Q)$.
Inside $\mathfrak C$ the last two gradients vanish. Outside $\mathfrak C$,
$V_L=Q\leq0$, so $\mathcal A(\nabla V_L)$ and $\mathcal A(\nabla Q)$ cancel. These regions cover
$\rn{n}$ except for a set of measure zero. Consequently
\begin{equation}\label{eq:pasted-measure}
\mu_{P_L}=\mu_{V_L}+\sigma_s-\sigma.
\end{equation}
Since $\mu_{P_L}^+\leq\mu_{V_L}+\sigma_s$ and $\mu_{P_L}^-\leq\sigma$, where the right sides are locally continuous Sobolev functionals, Lemma~\ref{lem:contact-measures} applies to $U_L\geq P_L$, and restricted to $\partial\mathfrak C$ it gives
$\mu_{U_L}|_{\partial \mathfrak C}\leq\mu_{V_L}|_{\partial \mathfrak C}+\sigma_s-\sigma$.
By Lemma~\ref{lem:truncation-locality}, truncation at $L>0$ does not change either measure on its zero set $\partial \mathfrak C$. Since $\mu_U\geq\mu_V$, and $\sigma_s,\sigma$ are both supported there, we obtain the global measure inequality
\begin{equation}\label{eq:cone-boundary-order}
\sigma_s\geq\sigma.
\end{equation}
By \eqref{eq:cone-measure-scaling}, $\sigma(B(0,sr))\geq s^{n-p}\sigma(B(0,r))$. The measure $\sigma$ is nonzero, since otherwise $Q\leq0$ would be $p$-harmonic and zero on $\operatorname{int}\mathfrak C$, hence $Q\equiv0$ by the strong maximum principle. Choosing $R>0$ with $\sigma(B(0,R))>0$ and iterating gives
\begin{equation}\label{eq:cone-iteration}
\frac{\sigma(B(0,s^jR))}{(s^jR)^{n-p}}
\geq\frac{\sigma(B(0,R))}{R^{n-p}}>0\qquad(j\geq0).
\end{equation}
On the other hand, shift $Q$ by its minimum on $\overline B(0,2r)$ and apply Lemma~\ref{lem:source-capacity} on $B(0,2r)$. For compact subsets of $B(0,r)$, the relative capacity is at most that of $\overline B(0,r)$ in $B(0,2r)$, which equals a constant times $r^{n-p}$, also when $p=n$.  It follows that
\begin{equation}\label{eq:cone-small-density}
\sigma(B(0,r))
\leq c\bigl(\operatorname{osc}_{B(0,2r)}Q\bigr)^{p-1}r^{n-p}
=o(r^{n-p}),
\end{equation}
since $Q$ is continuous and $Q(0)=0$. This contradicts
\eqref{eq:cone-iteration}, also for $p=n$, where the denominators
equal one.
\end{proof}

\subsection{Rescaling at a contact point}\label{subsec:rescaling}

\begin{lemma}\label{lem:contact-compactness}
Let $1<p\leq n$ and use the common convex set $C$ and
level $T=T_*<m$ fixed after Lemma~\ref{lem:continuation-endpoint}.
Let $x\in\partial C$.
Choose a strictly decreasing sequence $r_j\downarrow0$ and put
\[
\begin{gathered}
 C_j=(C-x)/r_j,\qquad \mathfrak C=\overline{\bigcup_j C_j},\\
 \eta_j=\eta(r_j),\qquad k_j=(m-T)/\eta_j,\\
 v_{i,j}(z)=\frac{w_i(x+r_jz)-T}{\eta_j}\qquad(i=A,B,M),
\end{gathered}
\]
where $\eta$ is defined in Lemma~\ref{lem:exterior-oscillation}
and $\mu_i=-\Delta_p w_i$.
Write $\omega_{i,j}=-\Delta_pv_{i,j}$ for the rescaled Riesz measures.
Fix $q\in(p-1,\min\{p,q_H\})$, where $q_H$ is as in
\eqref{eq:osc-weak-harnack}.
After passing to one subsequence, there are
$p$-superharmonic functions $H_A,H_B,H_M:\mathbb R^n\to
\mathbb R\cup\{+\infty\}$, finite outside $p$-polar sets, such that
\begin{equation}\label{eq:contact-compactness}
\begin{aligned}
 v_{i,j}&\longrightarrow H_i\quad\hbox{almost everywhere},\\
 \omega_{i,j}&\rightharpoonup\mu_{H_i}\quad\hbox{locally as Radon measures},\\
 \omega_{i,j}(D)&=r_j^{p-n}\eta_j^{1-p}\mu_i(x+r_jD).
\end{aligned}
\end{equation}
The cone $\mathfrak C$ is closed, proper, convex, and has nonempty interior.
There is a function
\[
 Q\in C(\mathbb R^n)\cap W^{1,p}_{\mathrm{loc}}(\mathbb R^n),
 \qquad Q\leq0,\quad Q\not\equiv0,
\]
which is zero on $\mathfrak C$, is $p$-harmonic on
$\mathbb R^n\setminus \mathfrak C$, and has positive Riesz measure
$\sigma=-\Delta_pQ$ supported on $\partial \mathfrak C$.
Moreover,
\begin{equation}\label{eq:contact-cone-signs}
\begin{aligned}
 H_i&=Q\quad\hbox{on }\mathbb R^n\setminus \mathfrak C,\\
 H_i&\geq0\quad\hbox{in }\operatorname{int}\mathfrak C,\qquad
 H_i=0\quad\hbox{on }\partial \mathfrak C.
\end{aligned}
\end{equation}
For every fixed bounded open set $O$, the sequences $v_{i,j}$ have
a common lower bound on $O$ and a common $L^q(O)$ bound for the
chosen $q$, and $v_{i,j}\leq k_j\to\infty$.
\end{lemma}

\begin{proof}
Let $\nu$ be a supporting outward unit normal to $C$ at $x$.
The sets $C_j$ increase because $C$ is convex and $x\in C$, and the
closure of their union is the tangent cone at $x$. It is dilation
invariant, since for $t>0$ and $r_\ell\leq r_j/t$ convexity and
$0\in C-x$ give $tC_j\subset C_\ell$. It lies in
$\{z:z\cdot\nu\leq0\}$, so it is proper, and it has nonempty interior
since $C$ does. Each compact subset of $\operatorname{int}\mathfrak C$
lies in $C_j$ for all large $j$.

For $R\geq1$, \eqref{eq:osc-carleson} at radius $Rr$ with
$b=1/(2R)$ gives, for small $r$,
\begin{equation}\label{eq:eta-fixed-scales}
 \eta(Rr)\leq c_R v(x+r\nu/2)\leq c_R\eta(r),
\end{equation}
where $v=T-W$ is as in Lemma~\ref{lem:exterior-oscillation}.
Since $w_i\geq T$ in $C$ and $w_i=W$ outside $C$, this gives
$v_{i,j}\geq-c_R$ on $B(0,R)$, and $w_i\leq m$ gives
$v_{i,j}\leq k_j$, where $k_j\to\infty$ since $T<m$ and $\eta_j\to0$.
The weak Harnack estimate \eqref{eq:osc-weak-harnack} at radius
$Rr_j$ and \eqref{eq:eta-fixed-scales} give
\begin{equation}\label{eq:normalized-moment}
 \int_{B(0,R)}|v_{i,j}|^q\leq c_{R,q}
 \qquad(0<q<q_H).
\end{equation}

On a slightly larger ball, adding a constant to $v_{i,j}$ gives a
weak supersolution $V_{i,j}\geq1$ with the same gradient and the
bounds \eqref{eq:normalized-moment}. Choose $0<\tau<p-1$, with also
$\tau<p/(n-p)$ if $p<n$. Testing with $\zeta^p V_{i,j}^{-\tau}$, for
a smooth cutoff $\zeta$, gives
\[
 \int\zeta^p V_{i,j}^{-1-\tau}|\nabla V_{i,j}|^p
 \leq c\int V_{i,j}^{p-1-\tau}|\nabla\zeta|^p.
\]
H\"older's inequality with gradient exponent $p-1$ involves the
powers $p-1-\tau$ and $(p-1)(1+\tau)$ of $V_{i,j}$, both less than
$q_H$, so \eqref{eq:normalized-moment} yields
\begin{equation}\label{eq:normalized-flux-bound}
 \int_{B(0,R)}|\nabla v_{i,j}|^{p-1}\leq c_R.
\end{equation}
If $\zeta\geq0$ equals one on $B(0,R)$ and has support in a larger
fixed ball, the weak equation gives
$\omega_{i,j}(B(0,R))\leq\int|\nabla v_{i,j}|^{p-2}\nabla v_{i,j}\cdot\nabla\zeta\leq c_R$.
The scaling in \eqref{eq:contact-compactness} follows from the weak
equation.

Since $v_{i,j}$ need not be bounded in $W^{1,p}_{\mathrm{loc}}$, we
use the $p$-superharmonic truncations $t_{i,j,L}=\min(v_{i,j},L)$,
$L\in\mathbb N$, which are bounded on each fixed ball uniformly in
$j$. The estimate \cite[Lemma~3.27(ii), $q=0$]{HKM} for
$t_{i,j,L}-L\leq0$ gives uniform $W^{1,p}$ bounds on smaller balls.
Rellich compactness and a diagonal argument give almost everywhere
limits $f_{i,L}$ for every $L$ and every ball of integer radius, and
outside one null set $f_{i,L}=\min(f_{i,N},L)$ for $N>L$. Put
$f_i=\sup_L f_{i,L}$. If $f_i(z)<\infty$ and $L>f_i(z)$ is an
integer, then $f_{i,L}(z)=f_i(z)<L$, so $v_{i,j}(z)<L$ eventually and
$v_{i,j}(z)\to f_i(z)$. If $f_i(z)=+\infty$, then $f_{i,L}(z)=L$ for
every $L$ and $v_{i,j}(z)\to+\infty$. By Fatou's lemma and
\eqref{eq:normalized-moment}, the last case occurs only on a null set.

For each $N$, let $s_{i,N}$ be the lower semicontinuous
regularization of $\inf_{j\geq N}v_{i,j}$. By the fundamental
convergence theorem \cite[Theorem~8.2]{HKM}, which applies since the
families $\{v_{i,j}:j\geq N\}$ have common local lower bounds, $s_{i,N}$ is
$p$-superharmonic and equals that infimum outside a polar set
$P_{i,N}$. Outside $\bigcup_NP_{i,N}$ and the preceding null set,
$s_{i,N}\uparrow\liminf_jv_{i,j}=f_i$. By \cite[Lemma~7.3]{HKM}, the
increasing limit $H_i=\lim_Ns_{i,N}$ is $p$-superharmonic, since it
is finite almost everywhere, and $H_i=f_i$ almost everywhere.

We use the weak continuity theorem of \cite{TW}, after a change of
sign: if $p$-superharmonic functions $v_j,v$ have common local lower
bounds and $v_j\to v$ almost everywhere, then their Riesz measures
converge locally weakly. Applied to $v_{i,j}$ and $H_i$, it
identifies the weak limits of $\omega_{i,j}$ with $\mu_{H_i}$.

It remains to identify the common exterior function. Put
\[
 Q_j=\min(v_{A,j},0)=\min(v_{B,j},0)=\min(v_{M,j},0).
\]
These are zero on $C_j$ and equal $(W(x+r_jz)-T)/\eta_j$ outside
$C_j$, and zero on $\partial C_j$ by \eqref{eq:superharmonic-representative}
and the continuous exterior boundary value in
Lemma~\ref{lem:convex-trace}. Thus $Q_j$ is continuous,
$p$-superharmonic, locally in $W^{1,p}$, $p$-harmonic outside $C_j$,
and uniformly bounded on fixed balls by \eqref{eq:eta-fixed-scales}.
In coordinates near $x$ in which the exterior of $C$ is the upper
side of a Lipschitz graph $y_n=\varphi(y')$, the dilated graphs
$\varphi_j(z')=(\varphi(x'+r_jz')-\varphi(x'))/r_j$ have the same
Lipschitz bound and chart radii tending to infinity. Hence the
boundary H\"older estimate \cite{LN10}, in the local graph form
\cite[Lemma~8.1, (8.4)]{ALSV}, for the positive functions $-Q_j$, and
the interior H\"older estimate give a uniform modulus of continuity
on each fixed ball, which the zero extension into $C_j$ preserves.
After a further subsequence, $Q_j\to Q$ locally uniformly. The
result on locally uniform convergence in \cite[Lemma~7.3]{HKM} and
\cite[Corollary~7.20]{HKM} give $Q\in C\cap W^{1,p}_{\mathrm{loc}}$
and $p$-superharmonicity. It is zero in $\mathfrak C$ and $p$-harmonic
outside $\mathfrak C$ by interior compactness. The Carleson bound at
$b=1/2$ gives $-Q_j(\nu/2)\geq c>0$, hence $Q\not\equiv0$, and its
Riesz measure is positive and supported on $\partial\mathfrak C$.

Since $v_{i,j}\geq Q_j$, we get $H_i\geq Q$ almost everywhere, hence
everywhere since $Q$ is continuous. On compact subsets of
$\mathbb R^n\setminus\mathfrak C$, $v_{i,j}=Q_j$ for large $j$, so
$H_i=Q$ there, and on compact subsets of $\operatorname{int}\mathfrak C$
the functions $v_{i,j}$ are eventually nonnegative, so $H_i\geq0$.
At $z\in\partial\mathfrak C$, $H_i\geq Q$ gives essential lower limit
at least zero, and approach through the exterior, where $H_i=Q\to0$,
gives the reverse inequality. Since $H_i(z)$ is this essential lower
limit, \eqref{eq:contact-cone-signs} follows.
\end{proof}

\begin{proposition}\label{prop:unrestricted-contact}
Let $1<p\leq n$. There is no contact point in
$(A\cup B)\cap M\cap\partial C$, regular or irregular.
\end{proposition}
\begin{proof}
Suppose that $x\in A\cap M\cap\partial C$. Put $s=1-\theta$ and use
the subsequence and notation of Lemma~\ref{lem:contact-compactness}.
Keeping $W_A(x)=m$ fixed in \eqref{eq:pointwise} and using
$W_B\leq m$ gives $W_M(x+s(y-x))\geq W_B(y)$. Since $W_i=w_i$
quasi-everywhere and a fixed dilation preserves polar sets,
$v_{M,j}(sz)\geq v_{B,j}(z)$ almost everywhere. The almost everywhere
convergence in \eqref{eq:contact-compactness} gives
$H_M(sz)\geq H_B(z)$ almost everywhere, hence everywhere by
\eqref{eq:superharmonic-representative} applied to both sides.

Next we compare the Riesz measures. Put $E_{B,j}=(B-x)/r_j$ and
$E_{M,j}=(M-x)/r_j$. Then, for $i=B,M$, $\operatorname{supp}\omega_{i,j}\subset E_{i,j}$,
$v_{i,j}=k_j$ quasi-everywhere on $E_{i,j}$, and the bounds in
Lemma~\ref{lem:contact-compactness} give all hypotheses of
Lemma~\ref{lem:riesz-capacity-limit}. Since $x\in A$,
$sE_{B,j}\subset E_{M,j}$. Let $D$ be bounded and open, let $O$ be a
bounded open set containing $\overline D$, and suppose
$\mu_{H_B}(\partial D)=\mu_{H_M}(s\partial D)=0$. Monotonicity and
scaling of capacity give
\[
 \operatorname{Cap}_p(E_{M,j}\cap s\overline D,sO)
 \geq\operatorname{Cap}_p\bigl(s(E_{B,j}\cap\overline D),sO\bigr)
 =s^{n-p}\operatorname{Cap}_p(E_{B,j}\cap\overline D,O).
\]
Multiplying by $k_j^{p-1}$ and applying
Lemma~\ref{lem:riesz-capacity-limit} in $O$ and $sO$, we obtain
\begin{equation}\label{eq:limiting-source-contraction}
 \mu_{H_M}(sD)\geq s^{n-p}\mu_{H_B}(D).
\end{equation}
For $0\leq\phi\in C_c^\infty(\rn{n})$ and almost every $t>0$, the
boundary of $D_t=\{\phi>t\}$ is null for $\mu_{H_B}$ and for
$D\mapsto s^{p-n}\mu_{H_M}(sD)$, since these boundaries lie in
disjoint level sets of $\phi$. Applying
\eqref{eq:limiting-source-contraction} to $D_t$ and integrating in
$t$ gives $s^{p-n}\mu_{H_M}(s\,\cdot)\geq\mu_{H_B}$. Thus
$U(z)=H_M(sz)$ and $V=H_B$ satisfy $U\geq V$ and $\mu_U\geq\mu_V$, and
by \eqref{eq:contact-cone-signs}, $U=Q(s\,\cdot)$ and $V=Q$ outside
$\mathfrak C$, and $U,V\geq0$ inside $\mathfrak C$.
Lemma~\ref{lem:cone-comparison} gives a contradiction.

If $x\in B\cap M\cap\partial C$, the same argument with $A$ and $B$
interchanged, $s=\theta$, and $W_B(x)=m$ fixed in \eqref{eq:pointwise}
gives $H_M(s\,\cdot)\geq H_A$ and $s^{p-n}\mu_{H_M}(s\,\cdot)\geq\mu_{H_A}$,
and again a contradiction.
\end{proof}

\begin{proof}[Proof of Theorem~\ref{thmA} $(d)$]
Normalize an equality pair as in \eqref{eq:normalization} and
start the common convex levels using Lemma~\ref{lem:farlevels}.
If $T_*<m$, Lemma~\ref{lem:continuation-endpoint} gives
the common convex set $C$ at the level $T_*$.
Put $T=T_*$. There must be a contact with $A$ or $B$.
Indeed, if $A,B\Subset\operatorname{int}C$, convexity and
compactness imply
$M=\theta A+(1-\theta)B\Subset\operatorname{int}C$ as well,
and Lemma~\ref{lem:bootstrap} continues the levels, contrary to
\eqref{eq:continuation-supremum}.
Lemma~\ref{lem:regular-contact} puts every such contact in
$M\cap\partial C$ and makes it irregular for $\Omega_M$.
For $p>n$, all compact sets are regular, so this is impossible.
For $1<p\leq n$, Proposition~\ref{prop:unrestricted-contact}
excludes every contact in $(A\cup B)\cap M\cap\partial C$, and hence
every contact with $A$ or $B$.
Consequently $T_*=m$, and Lemma~\ref{lem:continuation-endpoint}
gives the common decreasing intersection
$\widehat A=\widehat B=\widehat M$ by
\eqref{eq:intersection}, and it is convex.
Undoing the translations and dilations proves that the filled
sets are convex and homothetic.
Conversely, suppose $\widehat E_i=C_i$ are convex and
$C_2=a+\rho C_1$, with $\rho>0$.
Lemma~\ref{lem4.3} gives
\[
\widehat{\lambda E_1+(1-\lambda)E_2}
=(1-\lambda)a+\bigl(\lambda+(1-\lambda)\rho\bigr)C_1.
\]
Translation invariance, homogeneity, and filling invariance imply
\begin{align*}
\Phi_p\bigl(\lambda E_1+(1-\lambda)E_2\bigr)
&=\bigl(\lambda+(1-\lambda)\rho\bigr)\Phi_p(C_1)\\
&=\lambda\Phi_p(E_1)+(1-\lambda)\Phi_p(E_2).
\end{align*}
This proves sufficiency for every $1<p<\infty$. This finishes the proof of Theorem~\ref{thmA} $(d)$. 
\end{proof}

\begin{remark}\label{rmk:unique-continuation}
Analytic continuation is used above only where the gradient is
nonzero. Unique continuation for two $p$-harmonic functions, open for
$n\geq3$ and $p\ne2$ \cite{GM}, would simplify the continuation
argument but would not exclude irregular contacts.
\end{remark}

\begin{remark}\label{rmk:zero-capacity-equality}
Let $E_1,E_2$ be nonempty compact sets, $0<\lambda<1$, and
$S=\lambda E_1+(1-\lambda)E_2$.
$(a)$ If $\Phi_p(E_1)=0<\Phi_p(E_2)$, then
$\Phi_p(S)=(1-\lambda)\Phi_p(E_2)$ if and only if $E_1$ is a single
point, and similarly with the indices interchanged.
$(b)$ If $\Phi_p(E_1)=\Phi_p(E_2)=0$ and $p>n$, both sets are single
points and equality holds.
$(c)$ If $\Phi_p(E_1)=\Phi_p(E_2)=0$ and $1<p\leq n$, equality holds
if and only if $S$ is $p$-polar, since the right-hand side of
\eqref{eq:equality} is zero and a compact set has $\Phi_p=0$ exactly
when it is $p$-polar (by Theorem~\ref{thmB} $(d)$ if $p=n$, and by
the cutoffs $\chi_R$, \eqref{eqn4.6} and the Sobolev inequality if $p<n$).
\end{remark}

\begin{proof}[Proof of $(a)$ and $(b)$]
A single point gives equality by translation invariance and
homogeneity, and for $p>n$ a compact set of zero capacity is a single
point by Lemma~\ref{lem4.2}. This proves $(b)$, and $(a)$ for $p>n$.
Let $1<p\leq n$. We show that for a compact set $K$ with
$\Phi_p(K)>0$ and a vector $v\ne0$,
\begin{equation}\label{eq:strict-translate-union}
\Phi_p\bigl(K\cup(K+v)\bigr)>\Phi_p(K).
\end{equation}
This gives $(a)$. Indeed, if $a\ne b$ lie in $E_1$, then
$K=\lambda a+(1-\lambda)E_2$ and $v=\lambda(b-a)$ satisfy
$K\cup(K+v)\subset S$ and $\Phi_p(K)=(1-\lambda)\Phi_p(E_2)>0$, so
monotonicity and \eqref{eq:strict-translate-union} give
$\Phi_p(S)>(1-\lambda)\Phi_p(E_2)$.

Put $L=K\cup(K+v)$ and suppose $\Phi_p(L)=\Phi_p(K)$.
Let $p<n$. By Lemma~\ref{lem:equality-functions}, $w_L$ is
admissible for $K$ and for $K+v$ in the variational problem stated
there, with energy $\capp(L)=\capp(K)=\capp(K+v)$. Uniqueness and
translation invariance give $w_K=w_L=w_{K+v}=w_K(\cdot-v)$ almost
everywhere. The sets
$D_j=\{x:j\leq\langle x,v\rangle/|v|^2<j+1\}$, $j\in\mathbb Z$,
cover $\rn{n}$, and the integrals of $|w_K|^{p^*}$ over them are
equal with finite sum. Hence $w_K=0$ almost everywhere, contrary to
$\int|\nabla w_K|^p\,dx=\capp(K)>0$.

Let $p=n$, and for $E\in\{K,L,K+v\}$ put $g_E=-w_E$. Off $E$,
$g_E=\bar G_E\geq0$ is $n$-harmonic, and
\eqref{eq:equality-green-asymptotic} gives
$g_E(x)=\gamma\log|x|+1+k_E+o(1)$ with
$k_E=-\gamma\log\Phi_n(E)$. The $C^1$ convergence in the proof of
Lemma~\ref{lem:farlevels}, which uses only these expansions, gives
\begin{equation}\label{eq:zero-case-green-gradient}
\nabla g_E(x)=\gamma\frac{x}{|x|^2}+o(|x|^{-1})
\qquad\text{as }|x|\to\infty,
\end{equation}
uniformly in direction. Since $K_\delta\subset L_\delta$, the
comparison \eqref{eqn4.17} $(c)$ for these closed neighbourhoods and
the outer limit \eqref{eqn2.19} give $g_K\geq g_L$ off $L$.
Likewise $g_{K+v}\geq g_L$ off $L$, and translating the
neighbourhoods gives $g_{K+v}=g_K(\cdot-v)$.

We claim that $g_K=g_L$ outside a large ball. Put $h=g_K-g_L\geq0$.
Since $k_K=k_L$, $h(x)\to0$ as $|x|\to\infty$.
By \eqref{eq:zero-case-green-gradient} there is $R$ such that, for
$|x|>R$ and $0\leq t\leq1$, the vector
$\mathbf v_t(x)=(1-t)\nabla g_L(x)+t\nabla g_K(x)$ has length
between $\gamma/(2|x|)$ and $2\gamma/|x|$. Subtracting the two weak
equations as in the proof of Lemma~\ref{lem9.7} gives
\[
\nabla\cdot(\mathcal B(x)\nabla h)=0,\qquad
\mathcal B(x)=\int_0^1
D_{\mathbf v}\bigl(|\mathbf v|^{n-2}\mathbf v\bigr)\big|_{\mathbf v=\mathbf v_t(x)}\,dt,
\]
in $\{|x|>R\}$, and
$c^{-1}|x|^{2-n}|\zeta|^2\leq\langle\mathcal B(x)\zeta,\zeta\rangle
\leq c|x|^{2-n}|\zeta|^2$ for $\zeta\in\rn{n}$. This linear equation
is uniformly elliptic on every relatively compact subdomain of
$\{|x|>R\}$, so the linear Harnack inequality used in
Lemma~\ref{lem9.7} and connectedness give $h=0$ or $h>0$ there.
Suppose $h>0$, choose $\rho>R$, and put
$d=\min_{\partial B(0,\rho)}h>0$. For $0<\epsilon<1$, the
$n$-harmonic function $(1-\epsilon)g_L+d$ does not exceed $g_K$ on
$\partial B(0,\rho)$, since $g_L\geq0$, and
$g_K-(1-\epsilon)g_L-d=\epsilon\gamma\log|x|+O(1)>0$ on every
sufficiently large sphere. Comparison on annuli and passage of the
outer radius to infinity give $g_K\geq(1-\epsilon)g_L+d$ for
$|x|\geq\rho$. Letting $\epsilon\downarrow0$ gives $h\geq d$ there,
contrary to $h(x)\to0$. This proves the claim.

The claim, applied also to $K+v\subset L$, gives
$g_K=g_L=g_{K+v}=g_K(\cdot-v)$ outside a ball $B(0,R')$.
Choose $t>0$ with $t|v|>R'$. Then all points $(t+j)v$,
$j=0,1,\ldots$, lie outside $B(0,R')$, and successive use of this
identity gives $g_K((t+j)v)=g_K(tv)$ for every $j$. This contradicts
$g_K(x)=\gamma\log|x|+O(1)$ and proves
\eqref{eq:strict-translate-union}.
\end{proof}

\begin{proof}[Proof of Corollary~\ref{cor:self-average}]
Apply Theorem~\ref{thmA} $(d)$ with $E_1=E_2=E$ and the given $\lambda$.
\end{proof}

\begin{proof}[Proof of Corollary~\ref{cor:sum-parallel}]
Theorem~\ref{thmA} $(a), (c)$ and Theorem~\ref{thmB} $(c)$, with $\lambda=1/2$,
and homogeneity of degree one give \eqref{eq:sum-capacity}.
Its positive-capacity equality case is Theorem~\ref{thmA} $(d)$.
Since $E_\delta=E+\delta\bar B$ by Lemma~\ref{lem9.1} $(a)$,
\eqref{eq:parallel-capacity} follows.

If $\Phi_p(E)>0$, the equality criterion for the sum with
$\delta\bar B$ says exactly that $\hat E$ is a closed ball of
positive radius. If $\Phi_p(E)=0$, part $(a)$ of
Remark~\ref{rmk:zero-capacity-equality}, applied with
$\lambda=1/2$, $E_1=E$, $E_2=\delta\bar B$ and a dilation by $2$, says that equality holds
exactly when $E$ is a single point.
Conversely, if $\hat E$ is a closed ball of radius $r\geq0$,
\eqref{eq:filled-neighbourhoods} makes
$\widehat{E_\delta}$ the concentric ball of radius $r+\delta$.
Filling invariance and homogeneity then give equality.
\end{proof}

\end{document}